\documentclass[9pt]{amsart}
\usepackage{amsmath}
\usepackage{amsxtra}
\usepackage{amscd}
\usepackage{amsthm}
\usepackage{amsfonts}
\usepackage{amssymb}
\usepackage{eucal}
\usepackage[all]{xy}
\usepackage{graphicx}

\newtheorem{cor}[subsubsection]{Corollary}
\newtheorem{lem}[subsubsection]{Lemma}
\newtheorem{prop}[subsubsection]{Proposition}

\newtheorem{thm}[subsubsection]{Theorem}

\theoremstyle{definition}

\theoremstyle{remark}
\newtheorem{rem}[subsubsection]{Remark}

\newcommand{\thmref}[1]{Theorem~\ref{#1}}

\newcommand{\secref}[1]{Sect.~\ref{#1}}
\newcommand{\lemref}[1]{Lemma~\ref{#1}}
\newcommand{\propref}[1]{Proposition~\ref{#1}}

\newcommand{\corref}[1]{Corollary~\ref{#1}}

\numberwithin{equation}{section}

\newcommand{\nc}{\newcommand}
\nc{\renc}{\renewcommand}
\nc{\ssec}{\subsection}
\nc{\sssec}{\subsubsection}
\nc{\on}{\operatorname}

\nc\ol{\overline}
\nc\wt{\widetilde}
\nc\tboxtimes{\wt{\boxtimes}}
\nc\tstar{\wt{\star}}
\nc{\alp}{a}

\nc{\ZZ}{{\mathbb Z}}
\nc{\NN}{{\mathbb N}}
\nc{\OO}{{\mathbb O}}
\renc{\SS}{{\mathbb S}}
\nc{\DD}{{\mathbb D}}
\nc{\GG}{{\mathbb G}}

\nc{\Fq}{{\mathbb F}_q}
\nc{\Fqb}{\ol{{\mathbb F}_q}}
\nc{\Ql}{\ol{{\mathbb Q}_\ell}}
\nc{\id}{\text{id}}
\nc\X{\mathcal X}

\nc{\Hom}{\on{Hom}}
\nc{\Lie}{\on{Lie}}
\nc{\Loc}{\on{Loc}}
\nc{\Pic}{\on{Pic}}
\nc{\Bun}{\on{Bun}}
\nc{\IC}{\on{IC}}
\nc{\ICs}{\on{IC}^{\frac{\infty}{2}}}
\nc{\bICs}{\overset{\bullet}{\on{IC}}{}^{\frac{\infty}{2}}}
\nc{\tICs}{\wt{\on{IC}}{}^{\frac{\infty}{2}}}
\nc{\ICsg}{\on{IC}^{\gamma+\frac{\infty}{2}}}
\nc{\ICsl}{\on{IC}^{\lambda+\frac{\infty}{2}}}
\nc{\ICslm}{\on{IC}^{\lambda+\frac{\infty}{2},-}}
\nc{\ICsm}{\on{IC}^{\frac{\infty}{2},-}}
\nc{\bICsm}{\overset{\bullet}{\on{IC}}{}^{\frac{\infty}{2},-}}
\nc{\Aut}{\on{Aut}}
\nc{\rk}{\on{rk}}
\nc{\Sh}{\on{Sh}}
\nc{\Perv}{\on{Perv}}
\nc{\pos}{{\on{pos}}}
\nc{\Conv}{\on{Conv}}
\nc{\Sph}{\on{Sph}}
\nc{\Sat}{\on{Sat}}
\nc{\Sym}{\on{Sym}}
\nc{\BunBb}{\overline{\Bun}_B}
\nc{\BunNb}{\overline{\Bun}_N}
\nc{\BunTb}{\overline{\Bun}_T}
\nc{\BunBbm}{\overline{\Bun}_{B^-}}
\nc{\BunBbel}{\overline{\Bun}_{B,el}}
\nc{\BunBbmel}{\overline{\Bun}_{B^-,el}}
\nc{\Buno}{\overset{o}{\Bun}}
\nc{\BunPb}{{\overline{\Bun}_P}}
\nc{\BunBM}{\Bun_{B(M)}}
\nc{\BunBMb}{\overline{\Bun}_{B(M)}}
\nc{\BunPbw}{{\widetilde{\Bun}_P}}
\nc{\BunBP}{\widetilde{\Bun}_{B,P}}
\nc{\GUb}{\overline{G/U}}
\nc{\GUPb}{\overline{G/U(P)}}

\nc{\Hhom}{\underline{\on{Hom}}}
\nc\syminfty{\on{Sym}^{\infty}}
\nc\lal{\ol{\kappa_x}}
\nc\xl{\ol{x}}
\nc\thl{\ol{\theta}}
\nc\nul{\ol{\nu}}
\nc\lambdal{\ol{\lambda}}
\nc\Sum\Sigma
\nc{\oX}{\overset{o}{X}{}}
\nc{\hl}{\overset{\leftarrow}h{}}
\nc{\hr}{\overset{\rightarrow}h{}}
\nc{\M}{{\mathcal M}}
\nc{\N}{{\mathcal N}}
\nc{\F}{{\mathcal F}}
\nc{\D}{{\mathcal D}}
\nc{\Q}{{\mathcal Q}}
\nc{\Y}{{\mathcal Y}}
\nc{\G}{{\mathcal G}}
\nc{\E}{{\mathcal E}}
\nc{\CalC}{{\mathcal C}}
\nc\Dh{\widehat{\D}}

\nc{\C}{{\mathcal C}}
\nc{\K}{{\mathcal K}}
\renewcommand{\H}{{\mathcal H}}

\nc{\T}{{\mathcal T}}
\nc{\V}{{\mathcal V}}
\renc{\P}{{\mathcal P}}
\nc{\A}{{\mathcal A}}
\nc{\B}{{\mathcal B}}
\nc{\U}{{\mathcal U}}

\nc{\Gr}{{\on{Gr}}}

\nc{\frn}{{\check{\mathfrak u}(P)}}

\nc{\fC}{\mathfrak C}
\nc{\p}{\mathfrak p}
\nc{\q}{\mathfrak q}
\nc\f{{\mathfrak f}}

\nc{\qo}{{\mathfrak q}}
\nc{\po}{{\mathfrak p}}
\nc{\s}{{\mathfrak s}}
\nc\w{\text{w}}

\renewcommand{\mod}{{\on{-mod}}}

\nc\mathi\iota
\nc\Spec{\on{Spec}}
\nc\Mod{\on{Mod}}
\nc{\tw}{\widetilde{\mathfrak t}}
\nc{\pw}{\widetilde{\mathfrak p}}
\nc{\qw}{\widetilde{\mathfrak q}}
\nc{\jw}{\widetilde j}

\nc{\grb}{\overline{\Gr}}
\nc{\I}{\mathcal I}

\nc{\kappach}{{\check\kappa_x}}
\nc{\Lambdach}{{\check\Lambda}{}}
\nc{\lambdach}{{\check\lambda}}
\nc{\omegach}{{\check\omega}}
\nc{\nuch}{{\check\nu}}
\nc{\etach}{{\check\eta}}
\nc{\alphach}{{\checka}}
\nc{\oblvtach}{{\check\oblvta}}
\nc{\pich}{{\check\pi}}
\nc{\ch}{{\check h}}

\nc{\Hb}{\overline{\H}}

\nc{\BA}{{\mathbb{A}}}
\nc{\BC}{{\mathbb{C}}}
\nc{\BE}{{\mathbb{E}}}
\nc{\BF}{{\mathbb{F}}}
\nc{\BG}{{\mathbb{G}}}
\nc{\BM}{{\mathbb{M}}}
\nc{\BO}{{\mathbb{O}}}
\nc{\BD}{{\mathbb{D}}}
\nc{\BN}{{\mathbb{N}}}
\nc{\BP}{{\mathbb{P}}}
\nc{\BQ}{{\mathbb{Q}}}
\nc{\BR}{{\mathbb{R}}}
\nc{\BZ}{{\mathbb{Z}}}
\nc{\BS}{{\mathbb{S}}}

\nc{\CA}{{\mathcal{A}}}
\nc{\CB}{{\mathcal{B}}}

\nc{\CE}{{\mathcal{E}}}
\nc{\CF}{{\mathcal{F}}}
\nc{\CG}{{\mathcal{G}}}
\nc{\CH}{{\mathcal{H}}}

\nc{\CL}{{\mathcal{L}}}
\nc{\CC}{{\mathcal{C}}}
\nc{\CM}{{\mathcal{M}}}
\nc{\CN}{{\mathcal{N}}}
\nc{\CK}{{\mathcal{K}}}
\nc{\CO}{{\mathcal{O}}}
\nc{\CP}{{\mathcal{P}}}
\nc{\CQ}{{\mathcal{Q}}}
\nc{\CR}{{\mathcal{R}}}
\nc{\CS}{{\mathcal{S}}}
\nc{\CT}{{\mathcal{T}}}
\nc{\CU}{{\mathcal{U}}}
\nc{\CV}{{\mathcal{V}}}
\nc{\CW}{{\mathcal{W}}}
\nc{\CX}{{\mathcal{X}}}
\nc{\CY}{{\mathcal{Y}}}
\nc{\CZ}{{\mathcal{Z}}}
\nc{\CI}{{\mathcal{I}}}
\nc{\CJ}{{\mathcal{J}}}

\nc{\csM}{{\check{\mathcal A}}{}}
\nc{\oM}{{\overset{\circ}{\mathcal M}}{}}
\nc{\obM}{{\overset{\circ}{\mathbf M}}{}}
\nc{\oCA}{{\overset{\circ}{\mathcal A}}{}}
\nc{\obA}{{\overset{\circ}{\mathbf A}}{}}
\nc{\ooM}{{\overset{\circ}{M}}{}}
\nc{\osM}{{\overset{\circ}{\mathsf M}}{}}
\nc{\vM}{{\overset{\bullet}{\mathcal M}}{}}
\nc{\nM}{{\underset{\bullet}{\mathcal M}}{}}
\nc{\oD}{{\overset{\circ}{\mathcal D}}{}}
\nc{\obD}{{\overset{\circ}{\mathbf D}}{}}
\nc{\oA}{{\overset{\circ}{\mathbb A}}{}}
\nc{\op}{{\overset{\bullet}{\mathbf p}}{}}
\nc{\cp}{{\overset{\circ}{\mathbf p}}{}}
\nc{\oU}{{\overset{\bullet}{\mathcal U}}{}}
\nc{\oZ}{{\overset{\circ}{\mathcal Z}}{}}
\nc{\ofZ}{{\overset{\circ}{\mathfrak Z}}{}}
\nc{\oF}{{\overset{\circ}{\fF}}}

\nc{\fa}{{\mathfrak{a}}}
\nc{\fb}{{\mathfrak{b}}}
\nc{\fd}{{\mathfrak{d}}}
\nc{\ff}{{\mathfrak{f}}}
\nc{\fg}{{\mathfrak{g}}}
\nc{\fgl}{{\mathfrak{gl}}}
\nc{\fh}{{\mathfrak{h}}}
\nc{\fj}{{\mathfrak{j}}}
\nc{\fl}{{\mathfrak{l}}}
\nc{\fm}{{\mathfrak{m}}}
\nc{\fn}{{\mathfrak{n}}}
\nc{\fu}{{\mathfrak{u}}}
\nc{\fp}{{\mathfrak{p}}}
\nc{\fr}{{\mathfrak{r}}}
\nc{\fs}{{\mathfrak{s}}}
\nc{\ft}{{\mathfrak{t}}}
\nc{\fz}{{\mathfrak{z}}}
\nc{\fsl}{{\mathfrak{sl}}}
\nc{\hsl}{{\widehat{\mathfrak{sl}}}}
\nc{\hgl}{{\widehat{\mathfrak{gl}}}}
\nc{\hg}{{\widehat{\mathfrak{g}}}}
\nc{\chg}{{\widehat{\mathfrak{g}}}{}^\vee}
\nc{\hn}{{\widehat{\mathfrak{n}}}}
\nc{\chn}{{\widehat{\mathfrak{n}}}{}^\vee}

\nc{\fA}{{\mathfrak{A}}}
\nc{\fB}{{\mathfrak{B}}}
\nc{\fD}{{\mathfrak{D}}}
\nc{\fE}{{\mathfrak{E}}}
\nc{\fF}{{\mathfrak{F}}}
\nc{\fG}{{\mathfrak{G}}}
\nc{\fK}{{\mathfrak{K}}}
\nc{\fJ}{{\mathfrak{J}}}
\nc{\fL}{{\mathfrak{L}}}
\nc{\fM}{{\mathfrak{M}}}
\nc{\fN}{{\mathfrak{N}}}
\nc{\fP}{{\mathfrak{P}}}
\nc{\fU}{{\mathfrak{U}}}
\nc{\fV}{{\mathfrak{V}}}
\nc{\fZ}{{\mathfrak{Z}}}

\nc{\ba}{{\mathbf{a}}}
\nc{\bb}{{\mathbf{b}}}
\nc{\bc}{{\mathbf{c}}}
\nc{\bd}{{\mathbf{d}}}
\nc{\bbf}{{\mathbf{f}}}
\nc{\be}{{\mathbf{e}}}
\nc{\bi}{{\mathbf{i}}}
\nc{\bj}{{\mathbf{j}}}
\nc{\bn}{{\mathbf{n}}}
\nc{\bo}{{\mathbf{o}}}
\nc{\bp}{{\mathbf{p}}}
\nc{\bq}{{\mathbf{q}}}
\nc{\bu}{{\mathbf{u}}}
\nc{\bv}{{\mathbf{v}}}
\nc{\bx}{{\mathbf{x}}}
\nc{\bs}{{\mathbf{s}}}
\nc{\by}{{\mathbf{y}}}
\nc{\bw}{{\mathbf{w}}}
\nc{\bA}{{\mathbf{A}}}
\nc{\bK}{{\mathbf{K}}}
\nc{\bB}{{\mathbf{B}}}
\nc{\bF}{{\mathbf{F}}}
\nc{\bC}{{\mathbf{C}}}
\nc{\bG}{{\mathbf{G}}}
\nc{\bD}{{\mathbf{D}}}
\nc{\bE}{{\mathbf{E}}}
\nc{\bH}{{\mathbf{H}}}
\nc{\bI}{{\mathbf{I}}}
\nc{\bM}{{\mathbf{M}}}
\nc{\bN}{{\mathbf{N}}}
\nc{\bO}{{\mathbf{O}}}
\nc{\bV}{{\mathbf{V}}}
\nc{\bW}{{\mathbf{W}}}
\nc{\bX}{{\mathbf{X}}}
\nc{\bZ}{{\mathbf{Z}}}
\nc{\bS}{{\mathbf{S}}}

\nc{\sA}{{\mathsf{A}}}
\nc{\sB}{{\mathsf{B}}}
\nc{\sC}{{\mathsf{C}}}
\nc{\sD}{{\mathsf{D}}}
\nc{\sF}{{\mathsf{F}}}
\nc{\sK}{{\mathsf{K}}}
\nc{\sM}{{\mathsf{M}}}
\nc{\sO}{{\mathsf{O}}}
\nc{\sW}{{\mathsf{W}}}
\nc{\sQ}{{\mathsf{Q}}}
\nc{\sP}{{\mathsf{P}}}
\nc{\sZ}{{\mathsf{Z}}}
\nc{\sV}{{\mathsf{V}}}
\nc{\sr}{{\mathsf{r}}}
\nc{\bk}{{\mathsf{k}}}
\nc{\sg}{{\mathsf{g}}}
\nc{\sff}{{\mathsf{f}}}
\nc{\sfe}{{\mathsf{e}}}
\nc{\sfj}{{\mathsf{j}}}
\nc{\sfi}{{\mathsf{i}}}
\nc{\sfb}{{\mathsf{b}}}
\nc{\sfc}{{\mathsf{c}}}
\nc{\sd}{{\mathsf{d}}}
\nc{\sv}{{\mathsf{v}}}
\nc{\sw}{{\mathsf{w}}}

\nc{\BK}{{\bar{K}}}

\nc{\tA}{{\widetilde{\mathbf{A}}}}
\nc{\tB}{{\widetilde{\mathcal{B}}}}
\nc{\tg}{{\widetilde{\mathfrak{g}}}}
\nc{\tG}{{\widetilde{G}}}
\nc{\TM}{{\widetilde{\mathbb{M}}}{}}
\nc{\tO}{{\widetilde{\mathsf{O}}}{}}
\nc{\tU}{{\widetilde{\mathfrak{U}}}{}}
\nc{\TZ}{{\tilde{Z}}}
\nc{\tx}{{\tilde{x}}}
\nc{\tbv}{{\tilde{\bv}}}
\nc{\tfP}{{\widetilde{\mathfrak{P}}}{}}
\nc{\tz}{{\tilde{\zeta}}}
\nc{\tmu}{{\tilde{\lambda}}}

\nc{\urho}{\underline{\pi}}
\nc{\uB}{\underline{B}}
\nc{\uC}{{\underline{\mathbb{C}}}}
\nc{\ui}{\underline{i}}
\nc{\uj}{\underline{j}}
\nc{\ofP}{{\overline{\mathfrak{P}}}}
\nc{\oB}{{\overline{\mathcal{B}}}}
\nc{\og}{{\overline{\mathfrak{g}}}}
\nc{\oI}{{\overline{I}}}

\nc{\eps}{\varepsilon}
\nc{\hrho}{{\hat{\pi}}}

\nc{\one}{{\mathbf{1}}}
\nc{\two}{{\mathbf{t}}}

\nc{\Rep}{{\mathop{\operatorname{\rm Rep}}}}
\nc{\Tot}{{\mathop{\operatorname{\rm Tot}}}}
\nc{\Ker}{{\mathop{\operatorname{\rm Ker}}}}
\nc{\Hilb}{{\mathop{\operatorname{\rm Hilb}}}}
\nc{\End}{{\mathop{\operatorname{\rm End}}}}
\nc{\Ext}{{\mathop{\operatorname{\rm Ext}}}}
\nc{\CHom}{{\mathop{\operatorname{{\mathcal{H}}\it om}}}}
\nc{\GL}{{\mathop{\operatorname{\rm GL}}}}
\nc{\gr}{{\mathop{\operatorname{\rm gr}}}}
\nc{\Id}{{\mathop{\operatorname{\rm Id}}}}
\nc{\de}{{\mathop{\operatorname{\rm def}}}}
\nc{\length}{{\mathop{\operatorname{\rm length}}}}
\nc{\supp}{{\mathop{\operatorname{\rm supp}}}}

\nc{\Cliff}{{\mathsf{Cliff}}}
\nc{\Fl}{\on{Fl}}
\nc{\Fib}{{\mathsf{Fib}}}
\nc{\Coh}{{\mathsf{Coh}}}
\nc{\QCoh}{{\on{QCoh}}}
\nc{\IndCoh}{{\on{IndCoh}}}
\nc{\FCoh}{{\mathsf{FCoh}}}

\nc{\reg}{{\text{\rm reg}}}

\nc{\cplus}{{\mathbf{C}_+}}
\nc{\cminus}{{\mathbf{C}_-}}
\nc{\cthree}{{\mathbf{C}_*}}
\nc{\Qbar}{{\bar{Q}}}
\nc\Eis{\on{Eis}}
\nc\Eisb{\ol\Eis{}}
\nc\Eisr{\on{Eis}^{rat}{}}
\nc\wh{\widehat}
\nc{\Def}{\on{Def_{\check{\fb}}(E)}}
\nc{\barZ}{\overline{Z}{}}
\nc{\barbarZ}{\overline{\barZ}{}}
\nc{\barpi}{\overline\iota}
\nc{\barbarpi}{\overline\barpi}
\nc{\barpip}{\overline\iota{}^+}
\nc{\barpim}{\overline\iota{}^-}

\nc{\fq}{\mathfrak q}

\nc{\fqb}{\ol{\fq}{}}
\nc{\fpb}{\ol{\fp}{}}
\nc{\fpr}{{\fp^{rat}}{}}
\nc{\fqr}{{\fq^{rat}}{}}

\nc{\hattimes}{\wh\otimes}

\nc{\bh}{{\bar{h}}}
\nc{\bOmega}{{\overline{\Omega(\check \fn)}}}

\nc{\seq}[1]{\stackrel{#1}{\sim}}

\nc{\cT}{{\check{T}}}
\nc{\cG}{{\check{G}}}
\nc{\cM}{{\check{M}}}
\nc{\cB}{{\check{B}}}
\nc{\cN}{{\check{N}}}

\nc{\ct}{{\check{\mathfrak t}}}
\nc{\cg}{{\check{\fg}}}
\nc{\cb}{{\check{\fb}}}
\nc{\cn}{{\check{\fn}}}

\nc{\cLambda}{{\check\Lambda}}

\nc{\cla}{{\check\kappa_x}}
\nc{\cmu}{{\check\lambda}}
\nc{\clambda}{{\check\lambda}}
\nc{\cnu}{{\check\nu}}
\nc{\ceta}{{\check\eta}}

\nc{\DefbE}{{\on{Def}_{\cB}(E_\cT)}}

\nc{\imathb}{{\ol{\imath}}}

\nc{\KG}{K\backslash G}
\nc{\comult}{{co\text{-}mult}}
\nc{\counit}{{co\text{-}unit}}
\nc{\uHom}{{\underline{\CHom}}}
\nc{\dgSch}{\on{Sch}}
\nc{\Sch}{\on{Sch}}
\nc{\affdgSch}{\on{Sch}^{\on{aff}}}
\nc{\affSch}{\on{Sch}^{\on{aff}}}
\nc{\Groupoids}{\on{Grpd}}
\nc{\inftygroup}{\on{Spc}}
\nc{\inftyCat}{\infty\on{-Cat}}
\nc{\StinftyCat}{\inftyCat^{\on{St}}}
\nc{\MoninftyCat}{\infty\on{-Cat}^{\on{Mon}}}
\nc{\SymMoninftyCat}{\infty\on{-Cat}^{\on{SymMon}}}
\nc{\SymMonStinftyCat}{\on{DGCat}^{\on{SymMon}}}
\nc{\MonStinftyCat}{\on{DGCat}^{\on{Mon}}}
\nc{\inftystack}{\on{Stk}}
\nc{\inftystackalg}{Stk^{1\text{-}alg}}
\nc{\inftyprestack}{\on{PreStk}}
\nc{\inftydgnearstack}{\on{NearStk}}
\nc{\inftydgstack}{\on{Stk}}
\nc{\inftydgstackalg}{DGStk^{1\text{-}alg}}
\nc{\inftydgprestack}{\on{PreStk}}
\nc{\dgindSch}{\on{indSch}}
\nc{\indSch}{{}^{\on{cl}}\!\on{indSch}}
\nc{\infSch}{\on{infSch}}
\nc{\dr}{{\on{dR}}}

\nc{\mmod}{{\on{-}\!{\mathbf{mod}}}}

\nc{\starr}{\text{\dh}}
\nc{\Spectra}{\on{Spectra}}
\nc{\Crys}{\on{Crys}}
\nc{\oblv}{{\mathbf{oblv}}}
\nc{\ind}{{\mathbf{ind}}}
\nc{\coind}{{\mathbf{coind}}}
\nc{\inv}{{\mathbf{inv}}}
\nc{\triv}{{\mathbf{triv}}}
\nc{\CMaps}{{\mathcal Maps}}
\nc{\Maps}{\on{Maps}}
\nc{\bMaps}{\mathbf{Maps}}
\nc{\BMaps}{\ul{\on{Maps}}}
\nc{\Grid}{\on{Grid}}
\nc{\hGrid}{\on{Grid}^{\geq\,\on{dgnl}}}
\nc{\Diag}{\on{Diag}}
\nc{\bDelta}{\mathbf{\Delta}}
\nc{\tCateg}{(\infty\on{-2)-Cat}}
\nc{\ul}{\underline}
\nc{\Seg}{\on{Seq}}

\nc{\triSeg}{\on{tri-Seq}}
\nc{\quadSeg}{\on{quad-Seq}}
\nc{\nSeg}{\on{n-Seq}}
\nc{\Segm}{\on{Seg}^{\on{mkd}}}
\nc{\fLm}{\fL^{\on{mkd}}}
\nc{\inftyCatm}{\inftyCat^{\on{mkd}}}
\nc{\Blocks}{\mathbf{Blocks}}
\nc{\Snakes}{\mathbf{Snakes}}

\nc{\Sets}{\on{Sets}}
\nc{\Ran}{{\on{Ran}}}
\nc{\Vect}{\on{Vect}}
\nc{\Shv}{\on{Shv}}
\nc{\unn}{\mathbf{union}}
\nc{\Spc}{\on{Spc}}
\nc{\ppart}{(\!(t)\!)}
\nc{\qqart}{[\![t]\!]}
\nc{\Dmod}{\on{D-mod}}
\nc{\cD}{\mathcal D}
\nc{\ocD}{\overset{\circ}{\cD}}
\nc{\sfp}{\mathsf{p}}
\nc{\sfq}{\mathsf{q}}
\nc{\DGCat}{\on{DGCat}}
\renc{\det}{\on{det}}
\nc{\Conf}{{\on{Conf}}}
\nc{\Whit}{\on{Whit}}
\nc{\Reg}{\on{Reg}}
\nc{\Res}{\on{Res}}
\nc{\BunNbom}{\overline\Bun_N^{\omega^\rho}} 
\nc{\BunNbox}{(\overline\Bun_N^{\omega^\rho})_{\infty\cdot x}} 
\nc{\BunNmbox}{(\overline\Bun_{N^-}^{\omega^\rho})_{\infty\cdot x}}
\nc{\Hecke}{\on{Hecke}}
\nc{\BHecke}{B\on{-Hecke}}
\nc{\BmHecke}{B^-\on{-Hecke}}
\nc{\bHecke}{\overset{\bullet}{\on{Hecke}}}
\nc{\bCZ}{\ol\CZ}
\nc{\oCZ}{\overset{\circ}\CZ} 
\nc{\boCZ}{\ol{\oCZ}}
\nc{\sotimes}{\overset{!}\otimes}
\nc{\SI}{\on{SI}}
\nc{\semiinf}{{\frac{\infty}{2}}}
\nc{\coInd}{\on{coInd}}
\nc{\Ind}{\on{Ind}}
\nc{\bCM}{\overset{\bullet}\CM{}}
\nc{\oOmega}{\overset{\circ}\Omega{}} 
\nc{\oConf}{\overset{\circ}\Conf{}}
\nc{\Quant}{\on{Quant}}
\nc{\hbart}{{[\![\hbar]\!]}}
\nc{\hbarl}{{(\!(\hbar)\!)}}
\nc{\Frob}{\on{Frob}}
\nc{\opi}{\overset{\circ}\pi}

\begin{document}

\title[Factorization algebras in quantum 
geometric Langlands]{On factorization algebras arising in the quantum \\
geometric Langlands theory}

\author{D. Gaitsgory}

\date{\today}

\begin{abstract}
We study factorization algebras on configuration spaces of points on a curve, colored by elements 
of the root lattice. Our main result says that the factorization algebra attached to Lusztig's quantum group 
can be obtained as a direct image of a twisted Whittaker sheaf on the Zastava space. 
\end{abstract}

\maketitle

\tableofcontents

\section*{Introduction}

\ssec{Framework for this paper: the FLE}

This paper is a step towards the proof of the Fundamental Local Equivalence (FLE), which is a
(still conjectural\footnote{Recently, a proof by methods different than what we envisaged, was obtained by J.~Campbell, G.~Dhillon and S.~Raskin.}) 
equivalence between the twisted Whittaker category on the affine Grassmannian 
of a reductive group $G$ and the Kazhdan-Lusztig category of its Langlands dual $\cG$. 

\medskip

In this subsection we will recall what FLE is and what role this paper plays in our strategy to prove it. 
  
\sssec{}

We start with a datum of \emph{level} for $G$, which is an $\on{Ad}_G$-invariant symmetric bilinear form $\kappa$
on the Lie algebra $\fg$ of $G$. Given $\kappa$, we can consider the twisted category of D-modules $\Dmod_\kappa(\Gr_G)$
on the affine Grassmannian $\Gr_G=G\ppart/G\qqart$ of $G$, and the subcategory
$$\Whit_\kappa(\Gr_G) \subset \Dmod_\kappa(\Gr_G),$$
obtained by imposing the condition of $N\ppart$-equivariance against a non-degenerate character. 

\medskip

Suppose now that $\kappa$ is such that the form
\begin{equation} \label{e:kappa t}
\kappa+\frac{\kappa_{\on{Kil}}}{2}|_{\ft}
\end{equation} 
on the Cartan subalgebra $\ft\subset \fg$ is non-degenerate. Then to $\kappa$ one can attach its \emph{dual level} for $\cG$, 
denoted $\check\kappa$, so that the forms 
$$\kappa+\frac{\kappa_{\on{Kil}}}{2}|_{\ft}  \text{ and } \check\kappa+\frac{\check\kappa_{\on{Kil}}}{2}|_{\ct}$$
are each other's duals under the identification 
$$\ct\simeq \ft^\vee.$$

\medskip 

We will consider the affine Kac-Moody Lie algebra 
$$0\to k \to \widehat{\cg}_{\check\kappa}\to \cg\ppart\to 0$$ attached to $\cg$ and $\check\kappa$ 
(here $k$ is the ground field), and the category, denoted $\on{KL}_{\check\kappa}(\cG)$, of
representations of the Harish-Chandra pair $(\widehat{\cg}_{\check\kappa},\cG\qqart)$. 

\medskip

The Fundamental Local Equivalence says that there exists a (canonically defined) equivalence of categories
\begin{equation} \label{e:FLE}
\Whit_\kappa(\Gr_G) \simeq \on{KL}_{\check\kappa}(\cG).
\end{equation} 

\medskip

The equivalence \eqref{e:FLE} is not an easy statement to prove because it involves two geometrically
defined categories, one for $G$ and another for $\cG$, while the relationship between $G$ and $\cG$ is
combinatorial in nature (involution on the root data). Our method of attacking the FLE consists of expressing both sides 
in terms of objects that are combinatorially attached to the root datum and then comparing these directly.

\sssec{}  \label{sss:Jacquet}

The point of departure is that when $G$ is a torus $T$, the corresponding equivalence 
\begin{equation} \label{e:FLE T}
\Whit_\kappa(\Gr_T) \simeq \on{KL}_{\check\kappa}(\cT),
\end{equation} 
is relatively easy to construct (note that for a torus $\Whit_\kappa(\Gr_T)=\Dmod_\kappa(\Gr_T)$). 

\medskip

Hence, the natural desire is to describe $\Whit_\kappa(\Gr_G)$ (resp., $\on{KL}_{\check\kappa}(\cG)$) 
in terms of $\Whit_\kappa(\Gr_T)$ (resp., $\on{KL}_{\check\kappa}(\cT)$) and to show that these
descriptions match up under the equivalence \eqref{e:FLE T}. 

\medskip

Next comes the crucial observation that the two sides of \eqref{e:FLE} are not just plain (DG) categories;
rather they naturally extend to \emph{factorization categories} over any smooth curve 
(we are thinking of the variable $t$ in $k\qqart\subset k\ppart$ as a coordinate at a point $x$ on a curve $X$).
Moreover, the conjectural equivalence \eqref{e:FLE} is supposed to be an equivalence \emph{as factorization categories}. In particular, the equivalence 
in \eqref{e:FLE T} does have this structure.

\medskip

The categories $\Whit_\kappa(\Gr_G)$ and $\on{KL}_{\check\kappa}(\cG)$ are equipped with naturally defined
Jacquet functors
$$J_{\Whit}:\Whit_\kappa(\Gr_G)\to \Whit_\kappa(\Gr_T) \text{ and }
J_{\on{KL}}:\on{KL}_{\check\kappa}(\cG)\to \on{KL}_{\check\kappa}(\cT),$$
and each of these functors has a naturally defined factorization structure.

\medskip

Define
$$\Omega_{\kappa,\Whit}:=J_{\Whit}(\one_{\Whit_\kappa(\Gr_G)})\in \Whit_\kappa(\Gr_T)
\text{ and } \Omega_{\check\kappa,\on{KL}}:=J_{\on{KL}}(\one_{\on{KL}_{\check\kappa}(\cG)})\in \on{KL}_{\check\kappa}(\cT),$$
where 
$$\one_{\Whit_\kappa(\Gr_G)}\in \Whit_\kappa(\Gr_T) \text{ and } 
\one_{\on{KL}_{\check\kappa}(\cG)}\in \on{KL}_{\check\kappa}(\cT)$$
are the unit (a.k.a. vacuum) objects. 

\medskip

The factorization structure on the functor $J_{\Whit}$ (resp., $J_{\on{KL}}$) defines on 
$\Omega_{\kappa,\Whit}$ (resp., $\Omega_{\check\kappa,\on{KL}}$) a structure of \emph{factorization algebra}
in the corresponding factorization category, i.e., $\Whit_\kappa(\Gr_T)$ (resp., $\on{KL}_{\check\kappa}(\cT)$). Furthermore, 
the functors $J_{\Whit}$ and $J_{\on{KL}}$ upgrade to functors
\begin{equation} \label{e:Jacquet enhanced Whit}
J^{\on{enh}}_{\Whit}:\Whit_\kappa(\Gr_G)\to \Omega_{\kappa,\Whit}\on{-FactMod}(\Whit_\kappa(\Gr_T))
\end{equation}
and 
\begin{equation} \label{e:Jacquet enhanced KL}
J^{\on{enh}}_{\on{KL}}: \on{KL}_{\check\kappa}(\cG)\to \Omega_{\check\kappa,\on{KL}}\on{-FactMod}(\on{KL}_{\check\kappa}(\cT)),
\end{equation}
respectively, where the notation 
$$\CA\on{-FactMod}(\CC)$$
stands for the category of factorization modules for a factorization algebra $\CA$ in a given factorization category $\CC$. 

\medskip

In the best possible scenario, the functors \eqref{e:Jacquet enhanced Whit} and \eqref{e:Jacquet enhanced KL} would be
equivalences of categories. However, in our situation, they are \emph{not} such. Yet, there exists an explicit procedure that
allows to express the LHS in \eqref{e:Jacquet enhanced Whit} (resp., \eqref{e:Jacquet enhanced KL}) via the RHS, and 
we will discuss it in a subsequent publication. 

\medskip

Hence, for now we would like to match up the right-hand sides in \eqref{e:Jacquet enhanced Whit} and \eqref{e:Jacquet enhanced KL}.
This amounts to constructing an isomorphism of factorization algebras
\begin{equation} \label{e:ident Omega}
\Omega_{\kappa,\Whit} \simeq \Omega_{\check\kappa,\on{KL}}
\end{equation}
with respect to the equivalence of factorization categories in \eqref{e:FLE T}.

\begin{rem}

A similar strategy is supposed to lead to a new proof of the Kazhdan-Lusztig equivalence
$$\on{KL}_{\check\kappa}(\cG)\simeq \Rep_q(\cG),$$
where the latter is the category of modules over Lusztig's version of the quantum group, and $q$
is the quadratic form on the weight lattice of $\cG$ (=coweight lattice of $G$) with values in $\BC^\times$ obtained by exponentiating
the form that we denote $q_k$ (here the ground field $k$ is $\BC$), see \secref{sss:intro q}.

\medskip

The composite equivalence
$$\Whit_\kappa(\Gr_G) \simeq \Rep_q(\cG)$$
was the original form of the FLE suggested by J.~Lurie and the author a number of years ago. 

\end{rem}

\sssec{}  \label{sss:what we do prelim}

What we do in this paper is prove ``a half" of \eqref{e:ident Omega}. Right below we will explain what we mean by this in terms 
of the constructions \secref{sss:Jacquet} (in \secref{sss:what we do} we will translate this into the language adopted in the main body of the paper): 

\medskip

What we do is give an a priori construction of a certain factorization algebra in $\Whit_\kappa(\Gr_T)$, denote it
$\Omega_\kappa$, and show that it is isomorphic to $\Omega_{\kappa,\Whit}$.

\medskip

The main feature of $\Omega_\kappa$ is that it is combinatorial in nature, in that its construction only involves the 
coweight lattice $\Lambda$ of $G$ and the form \eqref{e:kappa t}. 

\medskip

In a subsequent publication we will establish an isomorphism
\begin{equation} \label{e:KM side}
\Omega_\kappa\simeq \Omega_{\check\kappa,\on{KL}}.
\end{equation} 

That would be the ``second half" of \eqref{e:ident Omega}. 

\ssec{What is actually done in this paper?}  \label{ss:what we do}

\sssec{} \label{sss:what we do}
 
We will now describe the passage from what we just explained in \secref{sss:what we do prelim} to the actual contents of the paper. 

\medskip

We observe that the factorization algebra $\Omega_{\kappa,\Whit}$ belongs to the full factorization subcategory 
of $\Whit_\kappa(\Gr_T)$ where we restrict the coweights of $T$ to belong to $\Lambda^{\on{neg}}\subset \Lambda$
(i.e., linear combinations of positive simple coroots with non-positive integral coefficients)\footnote{Technically, we also observe that the coweight $0$ 
component of $\Omega_{\kappa,\Whit}$ is the unit $\one_{\Whit_\kappa(\Gr_T)}$, so $\Omega_{\kappa,\Whit}$ is naturally augmented, 
and we will actually work with its augmentation ideal rather than $\Omega_{\kappa,\Whit}$ itself.}. 

\medskip

Now, by \cite[4.6]{GLys2}, the category of factorization algebras within this subcategory can be identified with the category of factorization algebras in 
the category 
\begin{equation} \label{e:Dmod on Conf}
\Dmod_{\CG^\Conf_\kappa}(\Conf),
\end{equation} 
where $\Conf$ is the configuration space of points on $X$ colored by negative coweights (see \secref{ss:conf}) and $\CG^\Conf_\kappa$
is a (factorization) gerbe on $\Conf$ attached to the form \eqref{e:kappa t}, see \secref{ss:gerbe}. 

\medskip

So, in the main body of the paper we:

\begin{itemize}

\item Construct explicitly a factorization algebra $\Omega_\kappa$ in $\Dmod_{\CG^\Conf_\kappa}(\Conf)$,
see \secref{ss:constr by deformation};

\item Construct the factorization algebra $\Omega_{\kappa,\Whit}$ in $\Dmod_{\CG^\Conf_\kappa}(\Conf)$, using
the geometry of Whittaker sheaves on the affine Grassmannian, see \secref{ss:Omega Whit};

\item Establish a (canonical) isomorphism of factorization algebras
\begin{equation} \label{e:Whit side}
\Omega_{\kappa,\Whit}\simeq \Omega_\kappa,
\end{equation} 
this is our main result, \thmref{t:main 1}.

\end{itemize}

\medskip

\noindent {\it Notational remark:} In the main body of the paper, the above factorization algebras carry a superscript ``$\on{Lus}$",
to distinguish them from two other versions, which carry superscripts ``$\on{DK}$" and ``$\on{sml}$", respectively. All three versions
are the same when $\kappa$ is \emph{irrational}.

\sssec{}  \label{sss:intro q}

The advantage of interpreting $\Omega_{\kappa,\Whit}$ 
as an object of $\Dmod_{\CG^\Conf_\kappa}(\Conf)$ is that
$\Conf$ is a scheme of finite type, so we find ourselves in the realm of usual algebraic geometry. 

\medskip

A crucial feature of the factorization algebras $\Omega_\kappa$ and $\Omega_{\kappa,\Whit}$ is that, when viewed as 
objects of $\Dmod_{\CG^\Conf_\kappa}(\Conf)$, they are \emph{perverse}, i.e., lie in the heart of the t-structure. 
It is this fact that will eventually allow us to construct an isomorphism between them. 

\medskip

We will now supply details as to how $\Omega_\kappa$ and $\Omega_{\kappa,\Whit}$ are constructed and how
the isomorphism \eqref{e:Whit side} is established.

\medskip

However, before we proceed any further, we make the following observation: 

\medskip

The two factorization algebras appearing in \eqref{e:Whit side} are of \emph{geometric nature},
i.e., they can be made sense of in a more general sheaf theory (see \secref{sss:sh th}) and not just D-modules. 
Let explain what replaces the role of the form \eqref{e:kappa t}. 

\medskip

We note that we can interpret the form \eqref{e:kappa t} as a symmetric bilinear form on $\Lambda$ with coefficients 
in $k$, denote it by $b_k$. Consider the quadratic form 
$$q_k(\lambda):=\frac{b_k(\lambda,\lambda)}{2},$$
and let $q$ be the composition of $q_k$ with the projection $k\to k/\BZ$. One can show that the gerbe $\CG^\Conf_\kappa$
only depends on the latter form $q$.

\medskip

When we work with another sheaf theory
$$Y\mapsto \Shv(Y),$$
with a field of coefficients $\sfe$, let $\fZ$ be the abelian group introduced in \secref{sss:K}, i.e., elements $\zeta\in \fZ$
parameterize Kummer local systems on $\BG_m$ 
$$\zeta\mapsto \Psi_\zeta,$$
so for D-modules we have $\fZ=k/\BZ$, and for constructible
sheaves $\fZ=\sfe^\times$. 

\medskip

So for a general sheaf theory, our quantum parameter is a quadratic form 
$$q:\Lambda\to \fZ,$$
see \secref{sss:q}. We denote the corresponding (factorization) gerbe on $\Conf$ by $\CG_q^\Conf$
(see \secref{ss:gerbe}), and the corresponding factorization algebras in $\Shv_{\CG_q^\Conf}(\Conf)$
will be denoted
$$\Omega^?_q \text{ and } \Omega^?_{q,\Whit},$$
where $?=\on{Lus,DK,sml}$.  

\sssec{}

We now explain how the factorization algebra $\Omega_q^{\on{Lus}}$ is constructed. 

\medskip

The initial observation is that the gerbe $\CG_q^\Conf$ is canonically trivialized over the open subscheme
\begin{equation} \label{e:oConf intro}
\oConf \overset{j}\hookrightarrow \Conf
\end{equation} 
consisting of simple divisors (pairwise distinct points of $X$ each carrying a negative simple coroot). 

\medskip

We let $\oOmega\in \Perv(\oConf)$ be the \emph{sign local system}, which is the simplest kind of
factorization algebra on $\oConf$. Due to the trivialization of $\CG_q^\Conf$ over $\oConf$, we
can think of $\oOmega$ as an object of $\Perv_{\CG_q^\Conf}(\oConf)$; when thought of in this
capacity we will denote it by $\oOmega_q$.

\medskip

For any value of $q$ we let $\Omega_q^{\on{sml}}$ denote the Goresky-MacPherson extension of $\oOmega_q$
along the open embedding $j$ of \eqref{e:oConf intro}. We emphasize that the above GM-extension is taking
place in the category of $\CG_q^\Conf$-twisted sheaves (and not plain sheaves). 

\medskip

We now proceed to the definition of $\Omega_q^{\on{DK}}$ and $\Omega_q^{\on{Lus}}$.

\medskip

When the values of $q$ in $\fZ$ are non-torsion we have
$$\Omega_q^{\on{DK}}\simeq \Omega_q^{\on{sml}}\simeq \Omega_q^{\on{Lus}};$$
we denote the resulting factorization algebra simply by $\Omega_q$. 

\medskip

For $q$ that does take torsion values, we introduce a deformation of $q$ over a formal
power series ring $\sfe\hbart$, so that the image of $q$ over $\sfe\hbarl$ is non-torsion valued.
Hence, the factorization algebra $\Omega_{q_\hbarl}$ is well-defined (see the paragraph above). 

\medskip

We set
$$\Omega^{\on{DK}}_{q_\hbart}:=j_*(\oOmega_{q_\hbart})\cap \Omega_{q_\hbarl},$$
where the intersection is taking place in $j_*(\oOmega_{q_\hbarl})$. 

\medskip

By construction, $\Omega^{\on{DK}}_{q_\hbart}$ is flat over $\sfe\hbart$, and we let
$$\Omega^{\on{DK}}_q:=\Omega^{\on{DK}}_{q_\hbart}\underset{\sfe\hbart}\otimes \sfe
\in \Perv_{\CG_q^\Conf}(\Conf).$$

Finally, we let $\Omega^{\on{Lus}}_q$ be the Verdier dual of $\Omega^{\on{DK}}_{q^{-1}}$, where we note that Verdier duality acts as
an anti-equivalence from $\Perv_{\CG_q^\Conf}(\Conf)$ to $\Perv_{\CG_{q^{-1}}^\Conf}(\Conf)$. 

\sssec{}

A remarkable feature of $\Omega^{\on{Lus}}_q$ (and one that will ultimately allow us to prove the isomorphism 
with the \emph{Kac-Moody} side, see \eqref{e:KM side}) is that (under some mild assumption on the torsion, see \secref{sss:dual Cox}) one can describe
$\Omega^{\on{Lus}}_q$ explicitly as a (twisted) perverse sheaf, see \secref{sss:Omega inductive}. 

\medskip

Namely, fix $\lambda\in \Lambda^{\on{neg}}$ and let $\Conf^\lambda$ be the corresponding connected
component of $\Conf$ (configurations of total degree $\lambda$). By induction and factorization, we can 
assume that $\Omega^{\on{Lus}}_q|_{\Conf^\lambda}$ has been defined away from the main diagonal
$$X \overset{\Delta_\lambda}\hookrightarrow \Conf^\lambda, \quad x\mapsto \lambda\cdot x.$$

Let $\jmath_\lambda$ be the embedding of the complementary open locus. Then:

\medskip

\begin{itemize}

\item If $\lambda$ is of the form $w(\rho)-\rho$ with $\ell(w)=2$, we have
$$\Omega^{\on{Lus}}_q|_{\Conf^\lambda}\simeq (\jmath_\lambda)_{!*}(\Omega^{\on{Lus}}_q|_{\Conf^\lambda-\Delta_\lambda(X)});$$

\item If $\lambda$ is not of the form $w(\rho)-\rho$ with $\ell(w)=2$, we have
$$\Omega^{\on{Lus}}_q|_{\Conf^\lambda}\simeq H^0\left((\jmath_\lambda)_!(\Omega^{\on{Lus}}_q|_{\Conf^\lambda-\Delta_\lambda(X)})\right),$$
where $H^0$ refers to taking $0$th cohomology in the perverse t-structure (usual t-structure for D-modules). 

\end{itemize}

\sssec{}

We now proceed to the definition of the factorization algebras 
\begin{equation} \label{e:Omega Whit intro}
\Omega^{\on{DK}}_{q,\Whit},\,\,\Omega^{\on{sml}}_{q,\Whit},\,\, \Omega^{\on{Lus}}_{q,\Whit}.
\end{equation}

According to \secref{sss:Jacquet} (adapted to our more general sheaf-theoretic setting), $\Omega^{\on{Lus}}_{q,\Whit}$ 
is obtained by applying the functor 
$$J_{\Whit}:\Whit_{\CG^G_q}(\Gr_G)\to \Whit_{\CG^T_q}(\Gr_T)$$
to the unit object $\one_{\Whit_{\CG^G_q}(\Gr_G)}\in \Whit_{\CG^G_q}(\Gr_G)$.

\medskip

However, instead of going through all this (which would necessitate developing quite a bit of theory), we will
write down directly what this construction produces\footnote{The connection to the construction of $\Omega^{\on{Lus}}_q$ 
as $J_{\Whit}(\one_{\Whit_{\CG^G_q}(\Gr_G)})$ can be seen from the material explained in (the optional)
Sects. \ref{ss:Whit}-\ref{ss:Gauss via semiinf}.}. 

\medskip

Let $\CZ^-$ be a version of the \emph{Zastava space}, introduced in \secref{sss:usual Zastava}.
Let
$$\oCZ\overset{\bj_Z^-}\hookrightarrow \CZ^-$$
be the \emph{open Zastava space}; it is known to be smooth. 

\medskip

We have a natural projection 
$$\pi^-:\CZ^-\to \Conf,$$
and let $\opi$ denote its restriction to $\oCZ$. 

\medskip

A key feature of this situation is that the pullback of the gerbe
$\CG^\Conf_q$ along $\opi$ acquires a canonical trivialization. Let
$$\nabla^-_{q,Z},\,\, \on{IC}^-_{q,Z},\,\, \Delta^-_{q,Z}$$
be the *-, GM- and !-extensions, respectively, of $\IC_{\oCZ}$, but viewed as a $\CG^\Conf_q$-twisted sheaf. 

\medskip

There is a canonical map
$$\chi:\CZ^-\to \BG_a,$$
given by the residue construction. Denote
$$\on{Gauss}^-_{q,*}:=\nabla^-_{q,Z}\overset{*}\otimes \chi^*(\on{exp}),\,\,
\on{Gauss}^-_{q,!*}:=\on{IC}^-_{q,Z}\overset{*}\otimes \chi^*(\on{exp}),\,\,
\on{Gauss}^-_{q,!}:=\Delta^-_{q,Z}\overset{*}\otimes \chi^*(\on{exp}).$$

All three are $\CG^\Conf_q$-twisted perverse sheaves on $\CZ^-$. Finally, set
$$\Omega^{\on{DK}}_{q,\Whit}:=\pi^-_*(\on{Gauss}^-{q,*}),\,\,
\Omega^{\on{sml}}_{q,\Whit}:=\pi^-_*(\on{Gauss}^-_{q,!*}),\,\,
\Omega^{\on{Lus}}_{q,\Whit}:=\pi^-_*(\on{Gauss}^-_{q,!}).$$

We note, however, that according to a cleanness result given by \thmref{t:acyclicity}, the maps
$$\pi^-_!(\on{Gauss}^-_{q,*})\to \pi^-_*(\on{Gauss}^-_{q,*}),$$
$$\pi^-_!(\on{Gauss}^-_{q,!*})\to \pi^-_*(\on{Gauss}^-_{q,!*})$$
and
$$\pi^-_!(\on{Gauss}^-_{q,!})\to \pi^-_*(\on{Gauss}^-_{q,!})$$
are isomorphisms.

\medskip

This shows that the objects \eqref{e:Omega Whit intro} are all \emph{perverse} and we have
$$\BD_\Conf^{\on{Verdier}}(\Omega^{\on{DK}}_{q,\Whit})\simeq \Omega^{\on{Lus}}_{q^{-1},\Whit} \text{ and }
\BD_\Conf^{\on{Verdier}}(\Omega^{\on{sml}}_{q,\Whit})\simeq \Omega^{\on{sml}}_{q^{-1},\Whit}.$$

\sssec{}

A rather simple calculation (see \propref{p:Whit on open}) shows that 
$$\Omega^{\on{DK}}_{q,\Whit}|_{\oConf},\,\, \Omega^{\on{sml}}_{q,\Whit}|_{\oConf},\,\, \Omega^{\on{Lus}}_{q,\Whit}|_{\oConf}$$
are all isomorphic to $\oOmega_q$.

\medskip

Our two main results, Theorems \ref{t:main 1} and \ref{t:main 2} assert that, under a certain restriction on $q$
(namely, $q$ should \emph{avoid small torsion}, see \secref{sss:small root} for what this means), these identifications extend (automatically, uniquely) to
isomorphisms
$$\Omega^{\on{DK}}_{q,\Whit}\simeq \Omega^{\on{DK}}_{q},\,\, 
\Omega^{\on{Lus}}_{q,\Whit}\simeq \Omega^{\on{Lus}}_{q}$$
and
\begin{equation} \label{e:Omega small intro}
\Omega^{\on{sml}}_{q,\Whit}\simeq \Omega^{\on{sml}}_{q},
\end{equation} 
respectively. 

\medskip

We note that the isomorphism \eqref{e:Omega small intro} has been already established in \cite{Lys2} under
somewhat more restrictive conditions on $q$.

\sssec{}

We will outline the strategy of the proof of Theorems \ref{t:main 1} and \ref{t:main 2} in some detail in \secref{ss:strategy} below.
Here we remark that, as we explain in Sects. \ref{ss:subtop} and \secref{ss:formal param Whit}, the statements
of both these theorems are equivalent to one cohomological estimate.

\medskip

Namely, let $\lambda\in \Lambda^{\on{neg}}-0$ be \emph{not} one of the negative simple roots. Consider the
intersection
$$S^0\cap S^{-,\lambda}\subset \Gr_G,$$
where 
$$S^0=N\ppart\cdot 1 \text{ and } S^{-,\lambda}=N^-\ppart\cdot t^\lambda$$
are the positive and negative semi-infinite orbits on the affine Grassmannian.

\medskip

The datum of $q$ defines a Kummer sheaf on $S^0\cap S^{-,\lambda}$, to be denoted $\Psi_{q,\lambda}$,
see \secref{sss:Psi lambda}\footnote{In fact, if $q$ is of the form $\zeta\cdot q^{\on{min}}_\BZ$, where 
$q^{\on{min}}_\BZ$ is the minimal form on $\Lambda$ and $\zeta\in \fZ$, then $\Psi_{q,\lambda}$ is the pullback of the
Kummer local system $\Psi_\zeta$ on $\BG_m$ along a canonically defined invertible function $f_\lambda$ on 
$S^0\cap S^{-,\lambda}$, see \secref{sss:det funct}.}.

\medskip

Consider the compactly supported cohomology
\begin{equation} \label{e:cohomology Gauss}
H^i_c(S^0\cap S^{-,\lambda},\Psi_{q,\lambda}\overset{*}\otimes \chi^*(\on{exp})).
\end{equation}

We have
$$\dim(S^0\cap S^{-,\lambda})=-\langle \lambda,\check\rho\rangle,$$ and it is fairly easy to see that the top cohomology in 
\eqref{e:cohomology Gauss}, i.e., one for $i=-2\langle \lambda,\check\rho\rangle$, vanishes.

\medskip

Now, the cohomology estimate mentioned above (one which is equivalent to Theorems \ref{t:main 1} and \ref{t:main 2})
says that the cohomology \eqref{e:cohomology Gauss} vanishes also in the \emph{sub-top} degree 
$$i=-2\langle \lambda,\check\rho\rangle-1,$$
for $q$ that avoids small torsion.

\medskip

This is our \thmref{t:main 3} (this is also \cite[Theorem 1.1.5]{Lys2} for a more restrictive hypothesis on $q$). 

\medskip

Our proof of Theorems \ref{t:main 1}, \ref{t:main 2} and \ref{t:main 3} makes use of quantum groups
(this is while the proof in \cite{Lys2} is a direct geometric argument, but one that has to exclude more cases). 

\medskip

So in a certain sense, we have a somewhat crazy story: in order to prove the vanishing of the cohomology
of a local system on a variety in the sub-top degree (say for $\ell$-adic sheaves over a ground field of positive characteristic),
we use the structure theory of quantum groups (the quantum group in question is the quantization of the Langlands
dual group $\cG$). 

\ssec{Strategy of proof: quantum groups creep in} \label{ss:strategy}

We first explain how to prove Theorems \ref{t:main 1} and \ref{t:main 2} when the ground field $k$ has characteristic $0$. 
In \secref{sss:crystal game} we explain how we deduce from this the general case.

\sssec{}

The case when $q$ takes non-torsion values is easy to settle (in fact, in this case \thmref{t:main 3} is an easy calculation,
which was performed already in \cite{Ga3}). So we will assume that $q$ takes values that are torsion in $\fZ$.

\medskip

The proof of Theorems \ref{t:main 1} and \ref{t:main 2} proceeds by studying an additional piece of structure 
on the factorization algebras $\Omega^{\on{Lus}}_{q,\Whit}$ (resp., $\Omega^{\on{Lus}}_{q}$), namely
Lusztig's quantum Frobenius. 

\medskip

Let $\Lambda^\sharp$ be the ``quantum Frobenius lattice" (see \secref{sss:sharp}), and let
$\Conf^\sharp$ be the corresponding configuration space. We can view $\Conf^\sharp$ as a
semi-group acting on $\Conf$. In this case we can talk about factorization algebras in
$\on{Perv}(\Conf^\sharp)$ acting on factorization algebras in $\Shv_{\CG^\Conf_q}(\Conf)$,
see \secref{sss:act}.

\medskip

Following a recipe of \cite{BG2}, we construct a factorization algebra $\Omega^{\on{cl},\sharp}$ in 
$\on{Perv}(\Conf^\sharp)$.

\medskip

By the quantum Frobenius structure on $\Omega^{\on{Lus}}_{q,\Whit}$ and $\Omega^{\on{Lus}}_{q}$
we mean an action of $\Omega^{\on{cl},\sharp}$ on these algebras so that we have the isomorphisms
\begin{equation} \label{e:quant Frob Whit}
\on{Bar}(\Omega^{\on{Lus}}_{q,\Whit})\simeq \Omega^{\on{sml}}_{q,\Whit}
\end{equation}
and 
\begin{equation} \label{e:quant Frob q}
\on{Bar}(\Omega^{\on{Lus}}_{q})\simeq \Omega^{\on{sml}}_{q},
\end{equation}
respectively. 

\medskip

Once we know that there \emph{exists} quantum Frobenius structures on both $\Omega^{\on{Lus}}_{q,\Whit}$ and $\Omega^{\on{Lus}}_{q}$,
the proof of \thmref{t:main 2} is achieved by a Jordan-Holder series argument (see \secref{sss:proof of Lus and sml via Whit}).  

\medskip

Now, the quantum Frobenius structure on $\Omega^{\on{Lus}}_{q,\Whit}$ is constructed geometrically, 
using a metaplectic version of the construction from \cite{BG2}, see \secref{ss:Whit quant Frob}. 

\medskip

However, when it comes to $\Omega^{\on{Lus}}_{q}$, the way we define it does not a priori supply sufficient information 
to construct a quantum Frobenius structure on it. And it is here that quantum groups come in.

\sssec{}

When considering $\Omega^{\on{Lus}}_{q}$, by Lefschetz principle, we may assume that the ground field is actually $\BC$, 
and the sheaf theory in question is that of constructible sheaves in classical topology.

\medskip

Let $\Rep_q(\cT)$ be the category of representations of the quantum torus associated with $q$. I.e., this is a braided monoidal category,
which as a monoidal category is isomorphic to $\Rep(\cT)$, but the braiding is specified by $q$. Thus, it makes sense to talk about
Hopf algebras in $\Rep_q(\cT)$.

\medskip

In \secref{ss:Hopf to Fact} we recall a construction that attaches to a Hopf algebra $H$ in $\Rep_q(\cT)$
(whose augmentation ideal is supported on coweights $\lambda\in \Lambda^{\on{neg}}-0$ with each coweight
component finite-dimensional) a factorization algebra $\Omega_H$ in $\Shv_{\CG^\Conf_q}(\Conf)$. 

\medskip

Applying this construction to Lusztig's, DeConcini-Kac and small versions of the (positive part of the) quantum group, we obtain 
factorization algebras
\begin{equation} \label{e:Omega quant intro}
\Omega^{\on{Lus}}_{q,\on{Quant}},\,\, \Omega^{\on{DK}}_{q,\on{Quant}},\,\, \Omega^{\on{sml}}_{q,\on{Quant}},
\end{equation} 
respectively. 

\medskip

It follows by definition that we have a canonical identification
\begin{equation} \label{e:quant on open}
\Omega^{?}_{q,\on{Quant}}|_{\oConf}\simeq \oOmega_q
\end{equation} 
for all three versions. 
\medskip

The t-exactness property of the construction $H\rightsquigarrow \Omega_H$ shows that the identification \eqref{e:quant on open}
extends (automatically, uniquely) to an isomorphism
$$\Omega^{\on{sml}}_{q,\on{Quant}}\simeq \Omega^{\on{sml}}_q.$$

Further, in \thmref{t:quant and abs} we show that the identification \eqref{e:quant on open} extends to the entire $\Conf$:
\begin{equation} \label{e:quant DK intro}
\Omega^{\on{DK}}_{q,\on{Quant}}\simeq \Omega^{\on{DK}}_q
\end{equation} 
(this is not difficult to prove). Finally, by Verdier duality, from \eqref{e:quant DK intro} we obtain:
$$\Omega^{\on{Lus}}_{q,\on{Quant}}\simeq \Omega^{\on{Lus}}_q.$$

\medskip

Thus, we have identified the pair of factorization algebras ($\Omega^{\on{Lus}}_q$, $\Omega^{\on{sml}}_q$)
with ($\Omega^{\on{Lus}}_{q,\on{Quant}}$, $\Omega^{\on{sml}}_{q,\on{Quant}}$). Hence, in order to 
construct the quantum Frobenius structure on $\Omega^{\on{Lus}}_q$, it is enough to do so for 
$\Omega^{\on{Lus}}_{q,\on{Quant}}$. The latter is given by Lusztig's Frobenius for quantum groups\footnote{It is here that we use the assumption that $q$ avoids small torsion.}.

\begin{rem}
We should remark that the factorization algebras \eqref{e:Omega quant intro} are not new objects in mathematics.
For example, $\Omega^{\on{sml}}_{q,\on{Quant}}$ was the key object of study of the book \cite{BFS} (building on earlier
works of V.~Schechtman and A.~Varchenko). In fact the contents of \secref{ss:Hopf to Fact} summarize the relationship
between $\Omega^{\on{sml}}_{q,\on{Quant}}$ and $u_q(\cN)$
from {\it loc.cit.}
\end{rem} 

\sssec{} \label{sss:crystal game}

Let us now explain how to deduce Theorems \ref{t:main 1} and \ref{t:main 2} over an arbitrary ground field from the case 
of the ground field of characteristic $0$. 

\medskip

As we explained above, Theorems \ref{t:main 1} and \ref{t:main 2} are equivalent to \thmref{t:main 3}, which asserts that
the sub-top cohomology in \eqref{e:cohomology Gauss} vanishes. 

\medskip

With no restriction of generality we can assume that $G$ is simple, and hence $q$ is of the form
$\zeta\cdot q^{\on{min}}_\BZ$, where $q^{\on{min}}_\BZ$ is the minimal integer-valued quadratic form on $\Lambda$
and $\zeta$ is a torsion element in $\fZ$. 

\medskip

In \secref{s:subtop} we prove \thmref{t:indep of char}, which says that the validity of \thmref{t:main 3} for a given $\on{ord}(\zeta)$ 
is independent of the ground field. 

\medskip

To do so, we analyze irreducible components of $Z$ of $S^0\cap S^{-,\lambda}$ for which
$$H^i_c(Z,,\Psi_{q,\lambda}\overset{*}\otimes \chi^*(\on{exp}))$$
may potentially be non-zero for $i=-2\langle \lambda,\check\rho\rangle-1$. 

\medskip

We first single out components that we declare to be ``under scrutiny", then among those some that we call ``suspicious"
(these two conditions do not depend on $\on{ord}(\zeta)$), and finally those that we declare as ``indicted". 
We then proceed to conviction and show that the presence of an indicted component indeed brings about the failure of \thmref{t:main 3} 
(for a given $\on{ord}(\zeta)$).  

\medskip

Thus, \thmref{t:indep of char} amounts to saying that for a given $\on{ord}(\zeta)$, the presence of indicted
irreducible components is independent of the ground field.  We prove this by showing that the property of being
indicted can be expressed purely in terms of the structure of Kashiwara's crystal 
on the set 
$$\sB(\lambda):=\underset{\lambda'}\cup\, \on{Irred}(S^{\lambda'}\cap S^{-,\lambda}).$$

We conclude the proof by noticing that the crystal $\sB(\lambda)$ is independent
of the characteristic of the ground field, due to Kashiwara's uniqueness theorem. 

\begin{rem}
Let us emphasize that the idea to use the structure of Kashiwara's crystal on $\sB(\lambda)$ for the
proof of \thmref{t:main 3} originates in \cite{Lys2}. 

\medskip

In fact, we follow the argument of {\it loc.cit.}, but go one step further: in \cite{Lys2} one looks at the
irreducible components that we call suspicious and rules them out for the specified values of $\on{ord}(\zeta)$.
We go from ``suspicious" and ``indicted" and that puts a stricter bound on $\on{ord}(\zeta)$.

\end{rem}

\ssec{Organization of the paper}

The actual order of exposition in the paper is somewhat different from how we have presented the contents
in \secref{ss:what we do}. We now briefly outline how the paper is organized section-by-section. 

\sssec{}

In \secref{s:Omega} we introduce the setting of factorization algebras in (gerbe-twisted) sheaves on the configuration space
and define the factorization algebras $\Omega_q^{\on{Lus}}$, $\Omega_q^{\on{sml}}$ and $\Omega_q^{\on{DK}}$.

\sssec{}

In \secref{s:properties} we state a number of theorems pertaining to the inductive construction of $\Omega^{\on{DK}}_q$
by extending it across the main diagonal. The factorization algebra $\Omega^{\on{Lus}}_q$ will enjoy a Verdier dual 
behavior.

\sssec{}

In \secref{s:quantum} we introduce the setting of Hopf algebras in the braided monoidal category $\Rep_q(\cT)$,
and the relationship between these Hopf algebras and factorization algebras on $\Conf$. 
Starting from the
positive part of the (various versions of the) quantum group, we construct the factorization algebras 
$\Omega_{q,\on{Quant}}^{\on{Lus}}$, $\Omega_{q,\on{Quant}}^{\on{sml}}$ and $\Omega_{q,\on{Quant}}^{\on{DK}}$.

\medskip 

We establish the isomorphisms $\Omega_{q,\on{Quant}}^{?}\simeq \Omega_{q}^{?}$. We then use 
properties of quantum groups to prove the theorems announced in \secref{s:properties}.

\sssec{}

In \secref{s:Frobenius} we introduce the quantum Frobenius.

\medskip

We introduce the quantum Frobenius lattice $\Lambda^\sharp$ (along with the corresponding configuration space)
$\Conf^\sharp$, and define what it means for a factorization algebra in $\Shv(\Conf^\sharp)$ to act on a factorization
algebra in $\Shv_{\CG^\Conf_q}(\Conf)$.

\medskip

We then use the quantum Frobenius for $U_q^{\on{Lus}}(\cN)$ to construct an action of a certain canonically defined factorization
algebra $\Omega^{\on{cl},\sharp}$ in $\Shv(\Conf^\sharp)$ on $\Omega_{q,\on{Quant}}^{\on{Lus}}$. From here we deduce 
the existence of a (canonically defined) action of $\Shv(\Conf^\sharp)$ on $\Omega_q^{\on{Lus}}$. 

\sssec{}

In \secref{s:Zastava} we recall the definition of the (several versions of the) Zastava space. We introduce the (twisted)
perverse sheaves 
$$\nabla^-_{q,Z},\,\, \Delta^-_{q,Z},\,\, \IC^-_{q,Z}$$
on $\CZ^-$.

\medskip

After tensoring by the pullback of the exponential/Artin-Schreier character sheaf, we obtain the perverse sheaves
\begin{equation} \label{e:Gauss intro}
\on{Gauss}^-_{q,*},\,\, \on{Gauss}^-_{q,!}, \on{Gauss}^-_{q,!*}
\end{equation}
on $\CZ^-$, see Remark \ref{r:Gauss} for the explanation of the name ``Gauss".

\medskip

One of the key technical results here is \thmref{t:acyclicity}, which says that the objects \eqref{e:Gauss intro} extend
\emph{cleanly} along the open embedding
$$\CZ^-\hookrightarrow \ol\CZ,$$
where $\ol\CZ$ is the compactified Zastava space.

\medskip

Finally, in Sects. \ref{ss:Whit}-\ref{ss:Gauss via semiinf} we present a more conceptual point of view 
on the construction of the objects \eqref{e:Gauss intro}. Namely, we show how they arise from factorization
algebras in Whittaker and semi-infinite categories on $\Gr_G$. 

\sssec{}

In \secref{s:Whit} we finally formulate and prove the main results of this paper.

\medskip

We define the factorization algebras
$$\Omega^{\on{DK}}_{q,\Whit},\,\, \Omega^{\on{Lus}}_{q,\Whit},\,\, \Omega^{\on{sml}}_{q,\Whit}$$
by taking the direct image(s) of the objects \eqref{e:Gauss intro} along the projection
$$\pi^-:\CZ^-\to \Conf.$$

We show that their restrictions to $\oConf$ identify with $\oOmega_q$, and we state our main Theorems
\ref{t:main 1} and \ref{t:main 2}, which say that the above identification over $\oConf$ extends to 
isomorphisms
$$\Omega^?_{q,\Whit}\simeq \Omega^?_q$$
for $?=\on{DK,Lus,sml}$. 

\medskip

We then show that Theorems \ref{t:main 1} and \ref{t:main 2} are both equivalent to \thmref{t:main 3},
which states the sub-top cohomology in \eqref{e:cohomology Gauss} vanishes. 

\medskip

Finally, we prove Theorems \ref{t:main 1} and \ref{t:main 2} (over a ground field of characteristic zero) 
using the quantum Frobenius and a certain trick involving Jordan-Holder series. 

\sssec{}

Finally, in \secref{s:subtop} we deduce the validity of \thmref{t:main 3} (and hence also that of Theorems \ref{t:main 1} and \ref{t:main 2})
over an arbitrary ground field from the case when the ground field has characteristic $0$.

\medskip

The game here is to express the validity of \thmref{t:main 3} in terms of the structure of Kashiwara's crystal on the set
of irreducible components of intersections of semi-infinite orbits. 

\sssec{}  

In \secref{s:h bar} we explain how to start with a sheaf theory $Y\mapsto \Shv(Y)$ with a field of coefficients $\sfe$,
and obtain from it a sheaf theory $Y\mapsto \Shv_\hbart(Y)$ over $\sfe\hbart$. 

\medskip

This story is parallel (but simpler) to how one gets $\BZ_\ell$-adic sheaves from $\BZ/\ell^n\BZ$-sheaves, which was systematically 
developed in \cite[Sect. 2.3]{GaLu}. 

\ssec{Notations and conventions}

By and large, this paper follows the notational conventions from \cite{GLys1,GLys2}. 

\sssec{}

In this paper our algebraic geometry takes place over an algebraically closed field, denoted $k$. 

\medskip

We will only need the classical (i.e., non-derived) algebraic geometry. We let $\on{Sch}_{\on{ft}}^{\on{aff}}$ denote the category
of affine schemes of finite type over $k$. We let $\on{PreStk}_{\on{lft}}$ denote the category of prestacks locally of finite type over $k$,
i.e., the category of all functors
$$(\on{Sch}_{\on{ft}}^{\on{aff}})^{\on{op}}\to \infty\on{-Groupoids}.$$

This category contains all schemes, ind-schemes, algebraic stacks, etc. (all locally of finite type). 

\medskip

That said, except for \secref{sss:gerbe twist} below, for the purposes of this paper we will only need groupoids 
with values in \emph{discrete} $\infty\on{-Groupoids}$, i.e., sets, and the most general prestacks that we will consider are ind-schemes. 

\medskip

We let $X$ be a smooth connected curve over $k$.

\sssec{}

We let $G$ be a semi-simple \emph{simply-connected} group\footnote{Although the FLE can be formulated for an arbitrary reductive $G$,
the current paper is dealing with the vacuum objects, namely, the corresponding factorization algebras, which only depend on the root system
of $G$. Hence for the purposes of this paper, it is enough to consider the case of $G$ simply-connected. This assumption has the additional 
advantage of streamlining to discussion of the geometric metaplectic data/factorization gerbes on the configuration spaces: for $G$ simply-connected
these gerbes are \emph{uniquely} recovered from the datum of a quadratic form $q$.} over $k$. 

\medskip

We let $\Lambda$ denote its coweight (=coroot lattice). Let $\check\Lambda$ denote the dual lattice (i.e., the weight 
lattice of $G$). Let $\check\rho\in  \check\Lambda$ denote the half-sum of the positive roots. 

\medskip

Let $I$ denote the Dynkin diagram of $G$. For $i\in I$ we let $\alpha_i\in \Lambda$ denote the corresponding
simple coroot. We let $\Lambda^{\on{pos}}$ (resp., $\Lambda^{\on{neg}}$) spanned by the elements 
$\alpha_i$ (resp., $-\alpha_i$). 

\medskip

For an element $\lambda\in \Lambda^{\on{pos}}$ (resp., $\Lambda^{\on{neg}}$) we will denote by $|\lambda|$ 
its length, i.e., $\langle \lambda,\check\rho\rangle$ (resp., $-\langle \lambda,\check\rho\rangle$).

\sssec{}

We will work with DG categories over a field of coefficients $\sfe$, assumed of characteristic $0$. When we say
``category" we will implicitly mean ``DG category".

\medskip

In this paper, unless specified otherwise, we will work with \emph{small} (non ind-complete) DG categories. 

\medskip

Given a DG category $\bC$, we can talk about t-structures on it. Given a t-structure, we will denote by $\bC^\heartsuit$
its heart. 

\sssec{Sheaf theories}  \label{sss:sh th}

We will work with one of the sheaf theories 
\begin{equation} \label{e:sheaf theory}
Y\mapsto \Shv(Y), \quad (\on{Sch}_{\on{ft}}^{\on{aff}})^{\on{op}}\to \on{DGCat}_\sfe,
\end{equation} 
(see \cite[Sect. 0.8.8]{GLys2}) from the following list:

\begin{itemize}

\item When $k=\BC$ we can take $\Shv(-)$ to be constructible sheaves in the classical topology
with coefficients in a field $\sfe$ of characteristic $0$;

\item Over an arbitrary ground field, we can take $\Shv(-)$ to be constructible $\ol\BQ_\ell$-adic
sheaves (in this case the field $\sfe$ of coefficients is $\ol\BQ_\ell$);

\item Over a ground field $k$ of characteristic $0$, we can take $\Shv(-)$ to be the category of holonomic 
D-modules (in this case the field $\sfe$ of coefficients is $k$). 

\end{itemize}

We extend the assignment \eqref{e:sheaf theory} to a functor
$$(\on{PreStk}_{\on{lft}}^{\on{aff}})^{\on{op}}\to \on{DGCat}_\sfe$$
by the procedure of \cite[Sect. 0.8.9]{GLys2}. 

\sssec{}  \label{sss:gerbe twist}

Let $1\on{-LS}$ be the commutative group object in $\on{PreStk}_{\on{lft}}$ that assigns to
$Y\in \on{Sch}_{\on{ft}}^{\on{aff}}$ the category of 1-dimensional local systems on $Y$.

\medskip

Let $\on{Ge}$ be the Zariski sheafification $B_{\on{Zar}}(1\on{-LS})$ of $B(1\on{-LS})$. For a prestack $\CY$, by  
$\sfe^\times$-gerbe on $\CY$ we will mean a map $\CY\to \on{Ge}$. 

\medskip

Given a $\sfe^\times$-gerbe $\CG$ on $\CY$, we can form a twisted category $\Shv_\CG(\CY)$, see
\cite[Sect. 1.7]{GLys1}.

\medskip

If $Y$ is a scheme, the category $\Shv_\CG(\CY)$ is equipped with a t-structure and we can consider
the abelian category $\Perv_\CG(\CY)$. 

\sssec{}  \label{sss:K}

In each of the above three examples of sheaf theories, we let $\fZ$ be the following abelian group:

 \begin{itemize}

\item For sheaves in the classical topology, we let $\fZ:=\sfe^\times$;

\item For $\ol\BQ_\ell$-adic sheaves, we let $\fZ:=\ol\BZ_\ell^\times$; 

\item For D-modules, we let $\fZ:=k/\BZ$.

\end{itemize}

In each of the above cases, an element $\zeta\in \fZ$ defines a Kummer character sheaf on $\BG_m$,
to be denoted $\Psi_\zeta$.

\sssec{}  \label{sss:K gerbe}

Furthermore, let $\CL$ be a line bundle on $\CY$ and let $\zeta$ be an element in $\fZ$. From this pair 
we obtain a canonically defined $\sfe^\times$-gerbe, to be denoted $\CL^\zeta$.  

\medskip

Indeed, by definition, it is sufficient to perform this construction for $\CY=Y\in \on{Sch}_{\on{ft}}^{\on{aff}}$.
In this case, the sought-for map
$$Y\to B_{\on{Zar}}(1\on{-LS})$$
is such that for every $U\subset Y$ and a trivialization of $\CL|_U$, the corresponding map
\begin{equation} \label{e:triv U}
U\to Y\to B_{\on{Zar}}(1\on{-LS})
\end{equation}
acquires a trivialization, and a change of the trivialization of $\CL|_U$ by $f_U:U\to \BG_m$ results 
in the change of the trivialization of \eqref{e:triv U} by 
$$U \overset{f_U}\to \BG_m \overset{\Psi_\zeta}\to 1\on{-LS}.$$

\ssec{Acknowledgements}

The author would like to thank P.~Etingof, I.~Heckenberger, D.~Nakano and I.~Angiono for their help
with quantum groups at small roots of unity. 

\medskip

The author is grateful to S.~Lysenko for his collaboration on the FLE and for sharing ideas
related to this work.

\medskip

The author's thinking about the subject of this paper has its origins in the book \cite{BFS};
he is grateful to its authors for inspiration and discussions we have had over many years. 

\medskip

A major part of this work was written while the author was holding the I.~M.~Gelfand chair at IHES.

\medskip

The author's research is supported by NSF grant DMS-1707662. 

\section{Factorization algebras attached to quadratic forms}  \label{s:Omega}

In this section we will introduce a quantum parameter $q$, recall the construction of the 
\emph{factorization gerbe} $\CG^\Conf_q$ on the configuration space $\Conf$,
and define the factorization algebras 
$$\Omega^{\on{sml}}_q,\,\, \Omega^{\on{DK}}_q,\,\, \Omega^{\on{Lus}}_q,$$
which are $\CG^\Conf_q$-twisted perverse sheaves on $\Conf$, equipped with a factorization structure. 

\ssec{The parameters}

\sssec{}  \label{sss:q}

We start with a datum of a quadratic form $q$ on $\Lambda$ with values in $\fZ$ (see \secref{sss:K} for what the symbol $\fZ$ stands for),
which is $W$-invariant and \emph{restricted}. The latter means the following:

\medskip

Let 
$$b:\Lambda\otimes \Lambda\to  \fZ$$
be the associated symmetric bilinear form, i.e.,
$$b(\lambda,\mu)=q(\lambda+\mu)-q(\lambda)-q(\mu).$$

We say that $q$ is restricted if $b$ satisfies the following identity
\begin{equation} \label{e:restr}
b(\alpha,\lambda)=\langle \lambda,\check\alpha\rangle \cdot q(\alpha)
\end{equation}
for every coroot $\alpha$ and $\lambda\in \Lambda$. 

\begin{rem}  \label{r:2 restr}

It is easy to see that the identity 
$$2b(\alpha,\lambda)=2\langle \lambda,\check\alpha\rangle \cdot q(\alpha)$$
is a formal consequence of $W$-invariance. But formula \eqref{e:restr} is not;
it is easy to produce a counterexample for the group $SL_2\times \BG_m$. 

\end{rem}

\sssec{}

It is shown in \cite[Sect. 3.2.4]{GLys1} that any restricted $q$ can be written as 
$$\Lambda \overset{q_\BZ}\to \BZ \overset{\zeta}\to \fZ$$ 
for some element $\zeta\in \fZ$ and an \emph{integer-valued} $W$-invariant quadratic form
$q_\BZ$ on $\Lambda$.  Note that it follows from Remark \ref{r:2 restr} that $q_\BZ$ 
automatically satisfies \eqref{e:restr}.

\medskip

In particular, for every simple factor of our root system we have
\begin{equation} \label{e:long and short}
q(\alpha_l)=d\cdot q(\alpha_s),
\end{equation}
where $\alpha_l$ (resp., $\alpha_s$) is any long (resp., short) coroot, and $d$ is the lacing number (the ratio of the squares of the lengths). 

\sssec{} \label{sss:root of unity}

We shall say that $q$ is \emph{torsion-valued} if the values of $q$ in $\fZ$ are torsion. 

\medskip

We shall say that $q$ is \emph{non-torsion valued} if for all coroots $\alpha$, the elements $q(\alpha)$ are non-torsion.
By \eqref{e:long and short}, this condition is enough to check for one coroot in each simple factor.

\medskip

Note that when our sheaf theory is that of sheaves in the classical topology and so $\fZ$ is the multiplicative group of the field of 
coefficients which we denote $\sfe$, ``torsion" means being a root of unity in $\sfe^\times$. So in the usual terminology
of quantum groups, ``torsion vs. non-torsion" means ``roots of unity vs. non-roots of unity". 

\medskip

We shall say that $q$ is \emph{non-degenerate} if $q(\alpha)\neq  0$ for all coroots $\alpha$. By \eqref{e:long and short}, 
this condition is enough to check for the long coroot in each simple factor of the root system. 

\medskip

The condition of non-degeneracy means that we are avoiding what is called ``the quasi-classical case" in \cite[Sect. 33.2]{Lus}
along any of the roots. 

\sssec{}  \label{sss:small root}

We shall say that $q$ \emph{avoids small torsion} if for every simple factor in our root system and (any) long coroot $\alpha_l$ in it we have
$$\on{ord}(q(\alpha_l))\geq d+1,$$
where $d=1,2,3$ is the lacing number of that simple factor. This is equivalent to \cite[Condition 35.1.2(a)]{Lus}.

\medskip

Note that if $d=1$, the above assumption is equivalent to $q$ being non-degenerate. For $d=2$, the assumption is equivalent to
$\on{ord}(\alpha_l)\neq 2$. For $d=3$, the assumption is equivalent to $\on{ord}(\alpha_l)\neq 2,3$. 

\medskip

The key results of this paper (Theorems \ref{t:main 1} and  \ref{t:main 2}) will use this assumption. On the quantum group
side, this assumption ensures a relationship between $U_q^{\on{Lus}}(\cN)$
and $u_q(\cN)$ via the quantum Frobenius, i.e., that \eqref{e:SES quant} is a short exact 
sequence of Hopf algebras.

\begin{rem} \label{r:small torsion bad}
From a certain point of view, when $q$ does \emph{not} avoid small torsion, the factorization algebras we will construct, 
namely, $\Omega_q^{\on{Lus}}$, $\Omega_q^{\on{DK}}$, $\Omega_q^{\on{sml}}$ are not quite the right objects
to consider. 

\medskip

The more relevant ones are the objects $\Omega_{q,\Whit}^{\on{Lus}}$, $\Omega_{q,\Whit}^{\on{DK}}$, 
$\Omega_{q,\Whit}^{\on{sml}}$, constructed in \secref{s:Whit}\footnote{For the explanation why these objects are
better behaved, see Remark \ref{r:Whit good}.}.  And these objects are indeed different (for example, for $G=G_2$
and $\on{ord}(q(\alpha_s))=\on{ord}(q(\alpha_l))=2$, see Remark \ref{r:criminal}). 

\medskip 

From this point of view, the quantum algebras $U_q^{\on{Lus}}(\cN)$,
$U_q^{\on{DK}}(\cN)$, $\fu_q(\cN)$ are not quite the right objects either (at these very small roots of unity). 

\end{rem}

\sssec{}  

Note that the pairing 
$$b:\Lambda\otimes \Lambda\to \fZ$$
canonically extends to a pairing
$$b:\Lambda\otimes \Lambda_{\on{ad}}\to \fZ,$$
where $\Lambda_{\on{ad}}$ is the coroot lattice of the adjoint quotient of $G$.

\medskip

Namely, we set
\begin{equation} \label{e:b extend}
b(\alpha_i,\mu):=\langle \mu, \check\alpha_i \rangle\cdot q(\alpha_i),\quad i\in I.
\end{equation} 

Alternatively, let $n$ be an integer so that $n\cdot \Lambda_{\on{ad}}\subset \Lambda$. Let 
us write $q$ as $q_\BZ\cdot \zeta$, and choose an $n$-th root $\zeta^{\frac{1}{n}}$ of $\zeta$. 
Then we have
$$b(\lambda,\mu)=b_\BZ(\lambda,n\cdot \mu)\cdot \zeta^{\frac{1}{n}}.$$
By \eqref{e:restr}, this gives the same value as \eqref{e:b extend} on the simple coroots. Furthermore, this implies
that the formula 
$$b(\alpha,\lambda):=\langle \lambda,\check\alpha \rangle\cdot q(\alpha)$$
holds for all coroots. 

\sssec{}  \label{sss:w rho}

In view of the above, we have a well-defined homomorphism 
$$\lambda\in \Lambda \mapsto b(\lambda,\rho)\in \fZ,$$
where $\rho$ is half-sum of positive coroots. Namely,  
\begin{equation} \label{e:b rho}
b(\alpha_i,\rho):=q(\alpha_i). 
\end{equation} 

\medskip

The following is simple combinatorial statement:
\begin{lem} \label{l:w rho}  
For $\lambda\in \Lambda^{\on{neg}}$, consider the element
\begin{equation} \label{e:value}
q(\lambda)+b(\lambda,\rho)\in \fZ.
\end{equation}

The element \eqref{e:value} vanishes for all $\lambda$ of the form
$w(\rho)-\rho$, $w\in W$.

\end{lem}
    
\ssec{The configuration space}  \label{ss:conf}

\sssec{}

Let $\Conf$ denote the configuration space of points of $x$ weighted by elements of $\Lambda^{\on{neg}}-0$, i.e., expressions of the form
\begin{equation} \label{e:point of conf}
\Sigma\, \lambda_n\cdot x_n, \quad \lambda_n\in \Lambda^{\on{neg}}-0, \quad x_{n'}\neq x_{n''}.
\end{equation} 

\medskip

We have
$$\Conf=\underset{\lambda\in \Lambda^{\on{neg}}-0}\sqcup\, \Conf^\lambda,$$
according to the total degree.  

\medskip

Each connected component $\Conf^\lambda$, also denoted $X^\lambda$, is a partially symmetrized power of $X$. Namely, for
\begin{equation} \label{e:lambda neg}
\lambda=\underset{i\in I}\Sigma\cdot (-n_i)\cdot \alpha_i
\end{equation} 
(where we recall that $I$ is the set of vertices of the Dynkin diagram), 
we have
\begin{equation} \label{e:lambda conf}
\Conf^\lambda:=X^\lambda\simeq \underset{i\in I}\Pi\, X^{(n_i)}.
\end{equation} 

\sssec{}

We let 
$$\oConf \overset{j}\hookrightarrow \Conf$$
denote the open locus corresponding to the condition that in \eqref{e:point of conf} all $\lambda_n$ are negative simple roots. 

\medskip

In other words, each connected component $\oConf^\lambda$ is the complement of the diagonal divisor in \eqref{e:lambda conf}.

\sssec{}

In what follows we will denote by 
$$(\Conf\times \Conf)_{\on{disj}}\subset \Conf\times \Conf$$
the disjoint locus, i.e., the open subset consisting of those pairs 
$$(\Sigma\, \lambda_n\cdot x_n,\Sigma\, \lambda'_{n'}\cdot x'_{n'}),$$
for which all $x_n$ and $x'_{n'}$ are pairwise distinct.  

\medskip

Similarly, we will use the notation  
$$(\Conf^{\times n})_{\on{disj}}\subset \Conf^{\times n}$$
for the $n$-fold product.

\sssec{}

We let
$$\on{add}:\Conf\times \Conf\to \Conf$$
denote the addition map, and let $\on{add}_n$ be its $n$-fold version. 

\medskip

Note that $\on{add}$ (resp., $\on{add}_n$) is \'etale when restricted to $(\Conf\times \Conf)_{\on{disj}}$
(resp., $(\Conf^{\times n})_{\on{disj}}$). 

\sssec{}  \label{sss:Delta lambda}

For an element $\lambda\in \Lambda^{\on{neg}}-0$, let $\Delta_\lambda$ denote the 
corresponding main diagonal
\begin{equation} \label{e:Delta lambda}
X\to \Conf^\lambda\subset \Conf, \quad x\mapsto \lambda\cdot x.
\end{equation} 

For a fixed point $x\in X$, we will use the notation $\iota_\lambda$ for the corresponding 
point $\lambda\cdot x\in \Conf^\lambda\subset \Conf$. 

\ssec{Gerbes on the configuration space}  \label{ss:gerbe}

\sssec{} \label{sss:rigid new}

According to \cite[Sect. 4.5.2]{GLys1}, to $q$ one canonically attaches a \emph{geometric metaplectic data}
$\CG_q^T$ for the torus $T$. It is uniquely characterized by the following two requirements:

\medskip

\noindent--The associated quadratic form (see \cite[Sects. 4.2.1 and 4.2.8]{GLys1}) equals $q$;

\smallskip

\noindent--For $\lambda\in \Lambda$, the gerbe $\CG_q^\lambda$ on $X$ (see \cite[Sect. 4.2.1]{GLys1}) is trivialized for $\lambda$
being each of the negative simple roots.

\medskip

In what follows, for a fixed point $x\in X$ we let
$\CG^\lambda_{q,x}$ denote the fiber of $\CG^\lambda_q$ at $x$; this is a plain $\sfe^\times$-gerbe. 

\sssec{}  \label{sss:factor gerbe}

By \cite[Sect. 4.6.5]{GLys2}, to $\CG_q^T$
we attach a \emph{factorization gerbe}, denoted $\CG_q^\Conf$ on $\Conf$.

\medskip

We recall that the datum of factorization amounts to an isomorphism
\begin{equation} \label{e:factor gerbe}
\CG_q^\Conf\boxtimes \CG_q^\Conf|_{(\Conf\times \Conf)_{\on{disj}}} \simeq \on{add}^*(\CG_q^\Conf)|_{(\Conf\times \Conf)_{\on{disj}}}.
\end{equation}

\medskip

The isomorphism \eqref{e:factor gerbe} is endowed with a datum of commutativity and associativity for iterated isomorphisms
$$(\CG_q^\Conf)^{\boxtimes n}|_{(\Conf^{\times n})_{\on{disj}}}\simeq \on{add}_n^*(\CG_q^\Conf)|_{(\Conf^{\times n})_{\on{disj}}}, \quad n\in \BN.$$

\medskip

Let us spell out the construction $\CG_q^T\rightsquigarrow \CG_q^\Conf$ explicitly, in terms of \cite[Sect. 4.2]{GLys1}.

\sssec{}

According to {\it loc. cit.}, the data of $\CG_q^T$
attaches to a finite set $J$ and a map $\lambda^J: J\to \Lambda$, a gerbe $\CG_q^{\lambda^J}$ on $X^J$.

\medskip

For $J=\{*\}$ and $\lambda^J$ corresponding to $\lambda\in \Lambda$, we recover the gerbe $\CG_q^\lambda$ on $X$, mentioned 
above. 

\medskip

For $J=\{1,2\}$ and $\lambda^J$ given by a pair of elements $\lambda_1,\lambda_2\in \Lambda$, the factorization structure
on $\CG^T_q$ defines an isomorphism 
$$\CG_q^{\lambda_1,\lambda_2}|_{X\times X-\Delta}=\CG_q^{\lambda_1}\boxtimes \CG_q^{\lambda_2}|_{X\times X-\Delta},$$
which extends to an isomorphism
\begin{equation} \label{e:on sq}
\CG_q^{\lambda_1,\lambda_2}\simeq (\CG_q^{\lambda_1}\boxtimes \CG_q^{\lambda_2})\otimes \CO(-\Delta)^{b(\lambda_1,\lambda_2)},
\end{equation} 
where:

\begin{itemize}

\item  for a line bundle $\CL$ on a scheme $\CY$ and an element $\zeta\in \fZ$, we denote by $\CL^\zeta$ the corresponding 
$\sfe^\times$-gerbe on $\CY$, see \secref{sss:K gerbe};

\item $b$ is the symmetric bilinear form attached to $q$ (in fact, \eqref{e:on sq} shows how to recover $b$ starting from $\CG_q^T$). 

\end{itemize}

\medskip

The assignment $(J,\lambda^J)\rightsquigarrow \CG^{\lambda^J}$ is functorial
in $(J,\lambda^J)$. In particular, $\CG_q^{\lambda^J}$ is equivariant with respect to the group $\on{Aut}(J,\Lambda^J)$.

\medskip

The above structure of equivariance has the following property. Let again $J=\{1,2\}$, and $\lambda^J$ be the constant map 
with value $\lambda$. Then the induced structure of equivariance on 
$$\CG_q^{\lambda,\lambda}|_\Delta$$
with respect to $\BZ/2\BZ$ is equipped with a trivialization (the latter makes sense as $\BZ/2\BZ$ acts trivially
on the underlying scheme $\Delta$). 

\medskip

It is explained in \cite[Sect. 4.2.8]{GLys1}, this datum of trivialization 
exhibits $b(\lambda,\lambda)$ as $2\times$ some other element of $\fZ$; this other element is the value $q(\lambda)$ of $q$
on $\lambda$. 

\medskip

Equivalently, this datum of trivialization is equivalent to descending the gerbe $\CG_q^{\lambda,\lambda}$ on $X^2$
to a gerbe on $X^{(2)}$, which identifies with
\begin{equation} \label{e:descend to sq}
(\CG_q^\lambda)^{(2)}\otimes \CO(-\Delta')^{q(\lambda)},
\end{equation} 
where:

\begin{itemize}

\item $\CG^{(n)}$ is the gerbe on $X^{(n)}$ given by the $n$th symmetric power\footnote{I.e, if we trivialize $\CG$ Zariski-locally on $X$, the gerbe
$\CG^{(n)}$ also acquires a trivialization; a change of the trivialization of $\CG$ by a $1$-dimensional local system $E$ on $X$ results in the change
of trivialization of $\CG^{(n)}$ by $E^{(n)}$, where the latter is a well-defined $1$-dimensional local system on $X^{(n)}$.} 
of a given gerbe $\CG$ on $X$;

\item $\Delta'$ is the diagonal divisor in $X^{(2)}$. 

\end{itemize} 

\medskip

Let now $(J,\lambda^J)$ be arbitrary. Let $X^{(J)}$ be the partially symmetrized power of the curve equal to the (GIT) quotient of $X^J$ by $\on{Aut}(J,\Lambda^J)$.
From the above datum of trivialization of the equivariance structure on the diagonals, we obtain that the gerbe $\CG_q^{\lambda^J}$ canonically descends to a gerbe
$\CG_q^{(\lambda^J)}$ on $X^{(J)}$.

\sssec{}

We are now ready to describe $\CG_q^\Conf$ explicitly in terms of $\CG_q^T$. 

\medskip

For $\lambda\in \Lambda^{\on{neg}}$ written as \eqref{e:lambda neg}, set
$$J=\underset{i\in I}\sqcup\, \{1,...,n_i\},$$
and let $\lambda^J$ be such that it sends each $\{1,...,n_i\}$ to $-\alpha_i$. Note that we have a natural identification
$$\Conf^\lambda=X^{(J)}.$$

Then with respect to this identification, we have
$$\CG_q^\Conf|_{\Conf^\lambda}=\CG_q^{(\lambda^J)}.$$

\sssec{}  \label{sss:gerbe expl new}

Here is an even more explicit description of $\CG_q^\Conf$ in terms of $q$. 

\medskip

Recall (see \secref{sss:rigid new}) that each $\CG_q^{-\alpha_i}$ is equipped with a trivialization. 

\medskip

This trivialization uniquely extends to a trivialization of the restriction 
of $\CG_q^\Conf$ to 
$$\oConf \subset  \Conf$$
in a way compatible with factorization. 

\medskip

Using \eqref{e:on sq} and \eqref{e:descend to sq}, we obtain that for $\lambda\in \Lambda^{\on{neg}}$
written as \eqref{e:lambda neg} (so that $\Conf^\lambda$ is written as in \eqref{e:lambda conf}), we have: 

\begin{equation} \label{e:formula for gerbe}
\CG_q^\Conf|_{\Conf^\lambda}=\left(\underset{i}\otimes\, \CO(-\Delta'_i)^{q(\alpha_i)}\right)
\bigotimes  \left(\underset{i\neq j}\otimes \CO(-\Delta_{i,j})^{b(\alpha_i,\alpha_j)}\right),
\end{equation}
where:

\begin{itemize}

\item $\Delta'_i$ denotes the diagonal divisor in $X^{(n_i)}$;

\item The product in the second factor goes over the set of \emph{unordered} pairs of distinct simple coroots,
and $\Delta_{i,j}$ denotes the incidence divisor in $X^{(n_i)}\times X^{(n_i)}$;

\end{itemize}  
 
\sssec{}

Restricting the RHS of formula \eqref{e:formula for gerbe} to the main diagonal, 
we obtain that the gerbe $\CG^\lambda_q$ on $X$ identifies with
\begin{equation} \label{e:gerbe on lambda}
\omega^{q(\lambda)+b(\lambda,\rho)},
\end{equation} 
see \secref{sss:w rho} for the meaning of $b(\lambda,\rho)$. 

\medskip

In particular, from \lemref{l:w rho} we obtain that $\CG^\lambda_q$ is \emph{canonically trivial}
for $\lambda$ of the form $w(\rho)-\rho$. 

\sssec{}  \label{sss:two curves}

Let $f:X_1\to X_2$ be an \'etale morphism of curves, and consider the corresponding map
of the configuration spaces
\begin{equation} \label{e:conf two curves}
f_\Conf:\Conf_1\to \Conf_2. 
\end{equation}

Let $\Conf_{1,et}\subset \Conf_1$ be the locus on which $f_\Conf$ is \'etale. Note that this locus always contains
the main diagonal as well as the complement to the \emph{diagonal divisor}. 

\medskip

Let $(\CG_q^\Conf)_1$ (resp., $(\CG_q^\Conf)_2$) denote the corresponding gerbe on $\Conf_1$
(resp., $\Conf_2$).  It follows from formula \eqref{e:formula for gerbe} that the tautological isomorphism
$$(\CG_q^\Conf)_1|_{\oConf_1}\simeq f_\Conf^*((\CG_q^\Conf)_2)|_{\oConf_1}$$
extends (automatically uniquely) to an isomorphism
\begin{equation} \label{e:gerbe two curves}
(\CG_q^\Conf)_1|_{\Conf_{1,et}}\simeq f_\Conf^*((\CG_q^\Conf)_2)|_{\Conf_{1,et}}.
\end{equation}

\sssec{}

Assume for a second that $X=\BA^1$, with a chosen coordinate $t$. This choice of coordinate 
makes all the divisors $\Delta_i$ and $\Delta_{i,j}$ principal, i.e., it gives
rise to a trivialization of the line bundles $\CO(-\Delta_i)$ and $\CO(-\Delta_{i,j})$, appearing in formula
\eqref{e:formula for gerbe}. In particular, we obtain that a choice of coordinate defines a trivialization of 
the gerbe $\CG_q^\Conf$.

\medskip

Note, however, that this trivialization is \emph{incompatible} with the canonical trivialization of 
$\CG_q^\Conf|_{\oConf}$. The discrepancy is given by a local system equal to the tensor product 
$$\left(\underset{i}\otimes\, f_i^*(\Psi_{q(\alpha_i)})\right) \bigotimes 
\left(\underset{i,j}\otimes\,  f_{i,j}^*(\Psi_{b(\alpha_i,\alpha_j)})\right),$$
where:

\begin{itemize}

\item For $\zeta\in \fZ$, we denote by $\Psi_\zeta$ the corresponding Kummer local system on $\BG_m$
(see \secref{sss:K}); 

\item $f_i$ (resp., $f_{i,j}$) is the (invertible) function on $X^\lambda$ given by the generator of the 
ideal of $\Delta_i$ (resp., $\Delta_{i,j}$). 

\end{itemize}

\ssec{Factorization algebras}

\sssec{}

The object of study in this paper is \emph{factorization algebras} in the category of 
\emph{gerbe-twisted perverse sheaves}\footnote{See \secref{sss:gerbe twist} for what this means.} 
on $\Conf$ (resp., $\oConf$).

\sssec{}

In the untwisted case, such a factorization algebra is an object $\CA\in \on{Perv}(\Conf)$,
equipped with an isomorphism
\begin{equation}  \label{e:factor algebra}
\CA\boxtimes \CA|_{(\Conf\times \Conf)_{\on{disj}}} \simeq \on{add}^!(\CA)|_{(\Conf\times \Conf)_{\on{disj}}},
\end{equation}
which is commutative and associative in a natural sense. 

\medskip

Note that in the above formula 
$\on{add}^!(\CA)|_{(\Conf\times \Conf)_{\on{disj}}}$ is perverse, since $\on{add}|_{(\Conf\times \Conf)_{\on{disj}}}$
is \'etale.

\sssec{}

Let now $\CG^\Conf$ be a gerbe over $\Conf$, equipped with a factorization structure (see \secref{sss:factor gerbe} for what this means).
Then, due to the factorization structure on $\CG$ we can talk about objects of $\on{Perv}_\CG(\Conf)$, 
equipped with a factorization structure. 

\medskip

Indeed, in this
case the left-hand side of \eqref{e:factor algebra} is an object of
$$\on{Perv}_{\CG\boxtimes \CG|_{(\Conf\times \Conf)_{\on{disj}}} }((\Conf\times \Conf)_{\on{disj}}),$$
and the right-hand side is an object of
$$\on{Perv}_{\on{add}^*(\CG)|_{(\Conf\times \Conf)_{\on{disj}}} }((\Conf\times \Conf)_{\on{disj}}),$$
but the two gerbes are identified due to \eqref{e:factor gerbe}. 

\sssec{}

The same discussion applies when we replace $\Conf$ by $\oConf$. 

\ssec{The factorization algebra on the open part} 

\sssec{}

Consider first the untwisted category $\on{Perv}(\oConf)$. We let $\oOmega$ be the factorization algebra in it, 
uniquely determined by the condition that 
$$\oOmega|_{\Conf^{-\alpha_i}}\simeq \sfe_X[1].$$

\sssec{}

Explicitly, recall that 
$$\oConf=\underset{\lambda}\sqcup\, \oConf^\lambda=\underset{\lambda}\sqcup\, (\overset{\circ}X{}^\lambda-\Delta),$$
where $\Delta$ is the diagonal divisor in each $\overset{\circ}X{}^\lambda$.  I.e., for $\lambda$ written as \eqref{e:lambda neg}, we have:
$$\Delta=(\underset{i}\cup\, \Delta'_i)\bigcup\, (\underset{i,j}\cup\, \Delta_{i,j}).$$

Then
$$\oOmega|_{\oConf^\lambda}$$ identifies with the restriction of
$$\underset{i}\boxtimes\, \on{sign}^{n_i}[n_i],$$
where $\on{sign}^n$ denotes the \emph{sign local system} on $X^{(n)}-\Delta$ (the cohomological shift by $[n]$ makes
$\on{sign}^n[n]$ into a perverse sheaf). 

\sssec{}  \label{sss:Omega o}

Recall now that the restriction of the factorization gerbe $\CG_q^\Conf$ to $\oConf$ admits a canonical
trivialization. Hence, we can identify $\Perv_{\CG_q^\Conf}(\oConf)$ with $\Perv(\oConf)$. 

\medskip 

Let $\oOmega_q\in \Perv_{\CG_q^\Conf}(\oConf)$ be the image of $\oOmega$ under the equivalence
$$\Perv_{\CG_q^\Conf}(\oConf)\simeq \Perv(\oConf).$$

\ssec{Construction of the ``small" factorization algebra $\Omega_q^{\on{sml}}$}

When discussing the various versions of the quantum factorization algebra $\Omega_q$,  
we will assume that $q$ is non-degenerate (see \secref{sss:root of unity} for what this means)\footnote{See 
also Remark \ref{r:small torsion bad}, according to which these factorization algebras
may not ultimately be the right objects to consider unless we require a slightly stronger condition, namely, 
that $q$ avoid small torsion (see \secref{sss:small root} for what this means).}.

\medskip

This assumption will be in effect for the duration of this subsection and the next one.

\sssec{}

We define the factorization algebra $\Omega_q^{\on{sml}}\in \Perv_{\CG_q^\Conf}(\Conf)$
$$\Omega_q^{\on{sml}}:=j_{!*}(\oOmega_q).$$

We emphasize that in the above formula, the Goresky-MacPherson extension $j_{!*}(-)$ is understood as a functor
$$\Perv_{\CG_q^\Conf}(\oConf)\to \Perv_{\CG_q^\Conf}(\Conf).$$

\sssec{}

Verdier duality defines a contravariant equivalence
$$\BD^{\on{Verdier}}:\Perv_{\CG_q^\Conf}(\Conf) \simeq \Perv_{\CG_{q^{-1}}^\Conf}(\Conf),$$
where where we note that $\CG_{q^{-1}}^\Conf$ is the factorization gerbe inverse to 
$\CG_q^\Conf$.

\medskip

By construction, we have
$$\BD^{\on{Verdier}}(\Omega_q^{\on{sml}})\simeq \Omega_{q^{-1}}^{\on{sml}}.$$

\sssec{}

Let $f:X_1\to X_2$ be as in \secref{sss:two curves}. Let $\Omega_{q,1}^{\on{sml}}$ (resp., $\Omega^{\on{sml}}_{q,2}$) be the
corresponding factorization algebra on $\Conf_1$ (resp., $\Conf_2$). We obtain that the tautological isomorphism
$$\oOmega_{q,1}^{\on{sml}}\simeq f_\Conf^*(\oOmega_{q,2}^{\on{sml}})$$
extends (automatically uniquely) to an isomorphism
\begin{equation} \label{e:Omega small two curves}
\Omega_{q,1}^{\on{sml}}|_{\Conf_{1,et}} \simeq f_\Conf^*(\Omega_{q,2}^{\on{sml}})|_{\Conf_{1,et}}.
\end{equation}

\ssec{Construction of the DeConcini-Kac and Lusztig versions}  \label{ss:constr by deformation}

We remind that we retain the assumption that $q$ is non-degenerate. 

\sssec{}

Recall the notation $\fZ$ (see \secref{sss:K}), and let $\fZ_\hbar$ be its variant
with a formal parameter, see \secref{sss:K h bar}.

\medskip

Let $q$ be as in \secref{sss:q}. We define a 
form $q_\hbar$ with values in $\fZ_\hbar$ as follows:

\medskip

Let $q^{\on{min}}_\BZ$ be the minimal integer-valued W-invariant quadratic form on $\Lambda$ (i.e., one that
takes value $1$ on the short coroot in every simple factor of our root system).  Set $$q_\hbar=q+q^{\on{min}}_\BZ\cdot \zeta_\hbar;$$ 
where $\zeta_\hbar\in \fZ_\hbar$ 
is the element defined as follows: 

\begin{itemize}

\item For sheaves in the classical topology $\zeta_\hbar= \exp(\hbar)\in (\sfe\hbart)^\times$;

\item For $\ell$-adic sheaves, $\zeta_\hbar=1+\hbar\in (\BZ_\ell[\![\hbar]\!])^\times$;

\item For D-modules $\zeta=\hbar\in k[\![\hbar]\!]/\BZ$.

\end{itemize} 

Let $\CG^\Conf_{q_\hbar}$ denote the corresponding factorization gerbe on $\Conf$
(see \secref{ss:h bar gerbes}). Explicitly, it can be defined by
the same formula as in \eqref{e:formula for gerbe}, up to replacing $q$ by $q_\hbar$. 

\sssec{}

We will denote the corresponding category of twisted sheaves 
by $\Shv_{\CG^\Conf_{q_{\hbart}}}(\Conf)$, and its localization with respect to $\hbar$ by 
$\Shv_{\CG^\Conf_{q_{\hbarl}}}(\Conf)$, see \secref{sss:twisted shvs h}. 

\medskip

We will also consider the ind-completions
$$\Shv^{\on{Ind}}_{\CG^\Conf_{q_{\hbart}}}(\Conf) \text{ and } \Shv^{\on{Ind}}_{\CG^\Conf_{q_{\hbarl}}}(\Conf)$$
of $\Shv_{\CG^\Conf_{q_{\hbart}}}(\Conf)$ and $\Shv_{\CG^\Conf_{q_{\hbarl}}}(\Conf)$, respectively. 

\medskip

We will identify $\Shv^{\on{Ind}}_{\CG^\Conf_{q_{\hbarl}}}(\Conf)$ with a full subcategory of 
$\Shv^{\on{Ind}}_{\CG^\Conf_{q_{\hbart}}}(\Conf)$ consisting of $\hbar$-local objects. 

%

%

\medskip

We let
$$\on{Perv}(-)\subset \Shv(-) \text{ and } \on{Perv}^{\on{Ind}}(-)\subset \Shv^{\on{Ind}}(-)$$
denote the corresponding abelian subcategories.  

\medskip

We will identify $\on{Perv}^{\on{Ind}}_{\CG^\Conf_{q_{\hbarl}}}(\Conf)$ with a full subcategory of 
$\on{Perv}^{\on{Ind}}_{\CG^\Conf_{q_{\hbart}}}(\Conf)$, consisting of objects on which $\hbar$ is invertible. 

\sssec{}

Let $\oOmega_{q_\hbart}$ be the factorization algebra in 
$$\on{Perv}_{\CG^\Conf_{q_{\hbart}}}(\oConf),$$
(the latter category identifies with the untwisted $\on{Perv}_{\hbart}(\oConf)$), 
defined as in \secref{sss:Omega o}.

\medskip

Let  
$$\oOmega_{q_\hbarl}\in \on{Perv}_{\CG^\Conf_{q_{\hbarl}}}(\oConf)$$ be the image of $\oOmega_{q_\hbart}$ 
under the localization functor 
$$\on{Perv}_{\CG^\Conf_{q_{\hbart}}}(\oConf)\to \on{Perv}_{\CG^\Conf_{q_{\hbarl}}}(\oConf).$$

\sssec{}

Set
$$\Omega_{q_\hbarl}:=j_{!*}(\oOmega_{q_\hbarl})\in  \on{Perv}_{\CG^\Conf_{q_{\hbarl}}}(\Conf).$$

\begin{rem}

In many respects, $\Omega_{q_\hbarl}$ behaves similarly to the algebra $\Omega^{\on{sml}}_q$ for $q$
which is \emph{non-torsion valued}. 

%

\end{rem} 

\sssec{}  \label{sss:Omega hbar}

Note that the functor
$$j_*:\Shv_{\CG^\Conf_{q_{\hbart}}}(\oConf)\to \Shv_{\CG^\Conf_{q_{\hbart}}}(\Conf)$$
is t-exact. Indeed, the right t-exactness follows from the right t-exactness of the functor
$$j_*:\Shv_{\CG_q^\Conf}(\oConf)\to \Shv_{\CG_q^\Conf}(\Conf)$$
(the morphism $j$ is affine), while the left t-exactness follows from the fact that left adjoint of $j_*$,
i.e., $j^*$ is right t-exact. 

\medskip

Consider the object
$$j_*(\oOmega_{q_\hbart})\in \on{Perv}_{\CG^\Conf_{q_{\hbart}}}(\Conf)\subset \on{Perv}^{\on{Ind}}_{\CG^\Conf_{q_{\hbart}}}(\Conf).$$

Thinking of $\on{Perv}^{\on{Ind}}_{\CG^\Conf_{q_{\hbarl}}}(\Conf)$ as a full subcategory of $\on{Perv}^{\on{Ind}}_{\CG^\Conf_{q_{\hbart}}}(\Conf)$,
we can consider also the objects 
$$\Omega_{q_\hbarl} \text{ and } j_*(\oOmega_{q_\hbarl}),$$
thought of objects of $\on{Perv}^{\on{Ind}}_{\CG^\Conf_{q_{\hbart}}}(\Conf)$.

\medskip

We have the \emph{injective} maps in $\on{Perv}^{\on{Ind}}_{\CG^\Conf_{q_{\hbart}}}(\Conf)$
$$j_*(\oOmega_{q_\hbart}) \to j_*(\oOmega_{q_\hbarl}) \leftarrow \Omega_{q_\hbarl}.$$

\sssec{}

We define the object 
$$\Omega^{\on{DK}}_{q_\hbart}\in \on{Perv}^{\on{Ind}}_{\CG^\Conf_{q_{\hbart}}}(\Conf)$$
as the intersection
$$j_*(\oOmega_{q_\hbart}) \cap \Omega_{q_\hbarl}\subset j_*(\oOmega_{q_\hbarl}).$$

By \propref{p:t h bis}, $\Omega^{\on{DK}}_{q_\hbart}$ is actually an object of $\on{Perv}_{\CG^\Conf_{q_{\hbart}}}(\Conf)$. 

\medskip

By construction,
\begin{equation} \label{e:Omega h open}
j^*(\Omega^{\on{DK}}_{q_\hbart})\simeq \oOmega_{q_\hbart}.
\end{equation}

Furthermore, the image of $\Omega^{\on{DK}}_{q_\hbart}$ in $\on{Perv}_{\CG^\Conf_{q_{\hbarl}}}(\Conf)$
identifies with $\Omega_{q_\hbarl}$. 

\sssec{}

We define the factorization algebra 
$$\Omega_q^{\on{DK}}\in \on{Perv}_{\CG_q^\Conf}(\Conf)$$ 
as 
$$\Omega^{\on{DK}}_{q_\hbart}/\hbar\simeq \Omega^{\on{DK}}_{q_\hbart}\underset{\sfe\hbart}\otimes \sfe$$
(note that $\Omega^{\on{DK}}_{q_\hbart}$ is torsion-free (equivalently, flat) over $\sfe\hbart$). 
 
\medskip

Finally, we define 
the factorization algebra 
$$\Omega_q^{\on{Lus}}\in  \on{Perv}_{\CG_q^\Conf}(\Conf)$$ 
to be the Verdier dual of $\Omega_{q^{-1}}^{\on{DK}}$. 

\section{Properties of the DeConcini-Kac version}  \label{s:properties}

In this section we will formulate a series of theorems pertaining to the behavior of $\Omega^{\on{DK}}_q$.
These theorems describe (with an increasing degree of precision) how to construct $\Omega^{\on{DK}}_q$
inductively, starting from the open locus $\oConf$, and extending across the diagonals. 

\medskip

Throughout this section we will be assuming that $q$ is non-degenerate (see \secref{sss:root of unity}).

\ssec{First properties of $\Omega^{\on{DK}}_q$}

\sssec{}  \label{sss:maps Omega ext}

By \eqref{e:Omega h open}, we have
$$j^*(\Omega_q^{\on{DK}})\simeq \oOmega_q.$$

By adjunction, we obtain a map
\begin{equation} \label{e:Omega to open open}
\Omega_q^{\on{DK}} \to j_*(\oOmega_q).
\end{equation}

\begin{prop}  \label{p:DK injects}
The map \eqref{e:Omega to open open} is injective.
\end{prop}
 
\begin{proof}

The assertion is equivalent to the fact that the quotient 
$$j_*(\oOmega_{q_\hbart})/\Omega^{\on{DK}}_{q_\hbart}$$
is $\hbar$-torsion free. However, the above quotient embeds into
$$j_*(\oOmega_{q_\hbarl})/\Omega_{q_\hbarl},$$
on which $\hbar$ is invertible, and the assertion follows.

\end{proof} 

\begin{rem}  \label{r:transfer to A1}
When the ground field $k$ has characteristic zero, 
in Sects. \ref{ss:descr DK}-\ref{ss:descr DK bis}, we will give an explicit description of $\Omega_q^{\on{DK}}$ as a sub-object of
$j_*(\oOmega_q)$.

%
%

\end{rem} 

\sssec{}

From \propref{p:DK injects} we obtain that there exists an injective map
\begin{equation} \label{e:small to DK}
\Omega_q^{\on{sml}}\hookrightarrow \Omega_q^{\on{DK}},
\end{equation}
extending the identification of the restrictions of both sides to $\oConf$ with $\oOmega$. 

\medskip

We claim:

\begin{thm} \label{t:DK vs small gen} 
Let $q$ be non-torsion valued. Then the map \eqref{e:small to DK}
is an isomorphism.
\end{thm} 

When the ground field $k$ has characteristic $0$, this theorem will be proved in \secref{ss:proof of prop}. 
A proof that works for any ground field will be given in \secref{sss:proof of gen via Whit}.

\ssec{Properties of $\Omega_q^{\on{DK}}$ in characteristic $0$} \label{ss:descr DK}

From now on, until the end of this section we will assume that the ground field $k$ has characteristic $0$. 

\medskip

We will state a number of properties of $\Omega_q^{\on{DK}}$, which will be proved 
in \secref{ss:proof of prop}. The factorization algebra $\Omega_q^{\on{Lus}}$ will
enjoy Verdier-dual properties of those of $\Omega_q^{\on{DK}}$.

\sssec{} \label{sss:maps Omega}

For $\lambda\in \Lambda^{\on{neg}}$, let $\jmath_\lambda$ denote the open embedding complementary 
to the main diagonal
$$\Delta_\lambda:X\to X^\lambda.$$

\medskip

Note that by \propref{p:DK injects}, the tautological map
\begin{equation} \label{e:Omega to open}
\Omega_q^{\on{DK}}|_{\Conf^\lambda}\to (\jmath_\lambda)_*\circ (\jmath_\lambda)^*(\Omega_q^{\on{DK}})
\end{equation} 
induces an \emph{injection}
\begin{equation} \label{e:Omega to open H0}
\Omega_q^{\on{DK}}|_{\Conf^\lambda} \hookrightarrow H^0\left((\jmath_\lambda)_*\circ (\jmath_\lambda)^*(\Omega_q^{\on{DK}})\right),
\end{equation} 
where $H^0$ refers to the perverse t-structure (as it does throughout the paper). Indeed, the composition of \eqref{e:Omega to open H0}
with the (injective) map
$$H^0\left((\jmath_\lambda)_*\circ (\jmath_\lambda)^*(\Omega_q^{\on{DK}})\right)\to  j_* \circ j^*(\Omega_q^{\on{DK}})=j_*(\oOmega_q)$$
is the map from \propref{p:DK injects}. 

\medskip

This injectivity property is equivalent to the fact that $(\Delta_\lambda)^!(\Omega_q^{\on{DK}})$ lives in (perverse) cohomological
degrees $\geq 1$ for $|\lambda|>1$. 

\medskip

As a consequence of the injectivity of \eqref{e:Omega to open H0}, we obtain a map
\begin{equation} \label{e:GM to Omega}
(\jmath_\lambda)_{!*}\circ (\jmath_\lambda)^*(\Omega_q^{\on{DK}})\to \Omega_q^{\on{DK}}|_{\Conf^\lambda},
\end{equation} 
which is automatically injective. 

\sssec{}

We claim:

\begin{thm} \label{t:prop DK char 0 prim}  Let $\lambda=w(\rho)-\rho$ with $\ell(w)=2$.  Then
$(\Delta_\lambda)^*(\Omega_q^{\on{DK}})$
lives in (perverse) cohomological degrees $\leq -1$, i.e.,  $H^0((\Delta_\lambda)^*(\Omega_q^{\on{DK}}))=0$.
\end{thm}

%
%

This theorem will be proved in \secref{ss:proof of prop}.

\sssec{}

Let us rephrase \thmref{t:prop DK char 0 prim} in terms of the properties of the map \eqref{e:GM to Omega}: 

\begin{cor} \label{c:prop DK char 0 prim} Let $\lambda=w(\rho)-\rho$ with $\ell(w)=2$. Then the
map \eqref{e:GM to Omega} is an isomorphism.
\end{cor}

%


\ssec{Inductive construction of $\Omega^{\on{DK}}_q$}

\sssec{}  \label{sss:dual Cox}

From now on, until the end of this section we will impose the following additional assumption on $q$: 

\medskip

\noindent(*) \hskip2cm {\it For every positive coroot $\alpha$ we have the inequality $\on{ord}(q(\alpha))\geq \langle \rho,\check\alpha \rangle=|\check\alpha|$.}

\medskip

Note that assumption (*) reads as ``the order of $q(\alpha)$ is large enough". We can rewrite it also as follows: for every simple factor
$$\on{ord}(q(\alpha_{0,s}))\geq |(\alpha_{0,s})^\vee|=:h-1 \text{ and } \on{ord}(q(\alpha_{0,l}))\geq |(\alpha_{0,l})^\vee|=:h^\vee-1,$$
where $\alpha_{0,s}$ and $\alpha_{0,l}$ ate the longest short and long coroots; $h$ is the Coxeter number of $G$, and 
$h^\vee$ is dual Coxeter 
number of the Langlands dual $\cG$ of $G$

\begin{rem} \label{r:Geo}

In what follows we will assume that the result of \cite[Theorem 6.4.1]{Geo}, computing the cohomology of
the De Concini-Kac quantum group, holds under these assumptions, see Equation \eqref{e:Geo}. 

\medskip

In {\it loc.cit.}, this was established under the assumption that $\on{ord}(q(\alpha_{0,s}))$ is odd (and is not divisible by 3 if $G$ 
contains a factor isomorphic to $G_2$). In a recent communication, D.~Nakano communicated to the author a proof in the case
when $\on{ord}(q(\alpha_{0,s}))$ is even, under a slightly stronger assumption on $\on{ord}(q(\alpha_{0,s}))$.

\medskip

However, based on the Kazhdan-Lusztig equivalence between quantum groups and the affine algebra, and based on the 
range of validity of the Kashiwara-Tanisaki localization theorem at the negative level, we believe that (*) is sufficient for the validity of
\eqref{e:Geo}. 

\end{rem}

\sssec{}

We claim: 

\begin{thm} \label{t:prop DK char 0 sec}  \hfill

\smallskip

\noindent{\em(a)}
Let $\lambda$ be \emph{not} of the form $w(\rho)-\rho$ with $\ell(w)=2$.
Then the object $(\Delta_\lambda)^!(\Omega_q^{\on{DK}})$ lives in (perverse) cohomological
degrees $\geq 2$.

\smallskip

\noindent{\em(b)}
Let $\lambda$ be of the form $w(\rho)-\rho$ with $\ell(w)=2$. Then $H^1((\Delta_\lambda)^!(\Omega_q^{\on{DK}}))$
is isomorphic to $\sfe_{X}[1]$. 

\end{thm} 

This theorem will be proved in \secref{ss:proof of prop}.

\begin{rem}
It is likely that for the validity of \thmref{t:prop DK char 0 sec}, assumption (*) is an overkill. What we actually need is that the map \eqref{e:mod Serre}
be an isomorphism, see Remark \ref{r:when KD is KD}.
\end{rem}  

\sssec{}

Let us rephrase \thmref{t:prop DK char 0 sec} in terms of the maps from \secref{sss:maps Omega}:

\begin{cor} \label{c:prop DK char 0 sec}  \hfill

\smallskip

\noindent{\em(a)} Let $\lambda$ be \emph{not} of the form $w(\rho)-\rho$ with $\ell(w)=2$. Then 
the map \eqref{e:Omega to open H0} is an isomorphism. 

\smallskip

\noindent{\em(b)} 
Let $\lambda$ be of the form $w(\rho)-\rho$ with $\ell(w)=2$. Then 
the cone of the map \eqref{e:Omega to open} is isomorphic to $(\Delta_\lambda)_!(\sfe_{X}[1])$;
in particular, the object $(\jmath_\lambda)_*\circ (\jmath_\lambda)^*(\Omega_q^{\on{DK}})$ is perverse. 

\end{cor} 

\sssec{}  \label{sss:Omega inductive}
Note that, given the factorization property of $\Omega_q^{\on{DK}}$, 
Corollaries \ref{c:prop DK char 0 prim} and \ref{c:prop DK char 0 sec}(a) supply
a recipe of how to construct $\Omega_q^{\on{DK}}$ inductively on $|\lambda|$: 

\medskip

The base of induction is when $\lambda$ is a negative simple coroot, in which case $\Omega_q^{\on{DK}}|_{X^\lambda}=\sfe_X[1]$.
Let now $|\lambda|>1$. By induction and factorization we can assume that we have already constructed 
$\Omega_q^{\on{DK}}|_{X^\lambda-\Delta_\lambda(X)}$. To recover $\Omega_q^{\on{DK}}|_{X^\lambda}$ we proceed as follows: 

\medskip

\begin{itemize}

\item If $\lambda=w(\rho)-\rho$ with $\ell(w)=2$, apply $(\jmath_\lambda)_{!*}(-)$;

\item If $\lambda$ is not of this form, apply $H^0((\jmath_\lambda)_*(-))$. 

\end{itemize} 

\ssec{Further properties of $\Omega_q^{\on{DK}}$ in characteristic $0$}  \label{ss:descr DK bis}

We continue to assume that the inequality (*) holds\footnote{Unlike the case of \thmref{t:prop DK char 0 sec},
it is likely that for the validity of \thmref{t:prop DK char 0 non-deg} below, the full strength of (*) is needed.}. 

\sssec{}

We claim: 

\begin{thm}  \label{t:prop DK char 0 non-deg} \hfill

\smallskip

\noindent{\em(a)} If $\lambda$ is \emph{not} of the form $w(\rho)-\rho$ for $w\in W$, then 
$(\Delta_\lambda)^!(\Omega_q^{\on{DK}})=0$.

%
%
 \smallskip

\noindent{\em(b)} If $\lambda$ is the form $w(\rho)-\rho$ for $w\in W$ with $\ell(w)\geq 3$, then
$(\Delta_\lambda)^!(\Omega_q^{\on{DK}})$ is isomorphic to $\sfe_X[-\ell(w)+1]$, in particular, it is 
concentrated in cohomological degree $\ell(w)-2$. 

\end{thm} 

This theorem will be proved in \secref{ss:proof of prop}. 

\sssec{}

Let us rephrase \thmref{t:prop DK char 0 non-deg}(a) in terms of properties of the map \eqref{e:Omega to open}: 

\begin{cor}  \label{c:prop DK char 0 non-deg}
If $\lambda$ is \emph{not} of the form $w(\rho)-\rho$ for $w\in W$, then the map
\eqref{e:Omega to open} is an isomorphism; in particular, $(\jmath_\lambda)_*\circ (\jmath_\lambda)^*(\Omega_q^{\on{DK}})$
is perverse.
\end{cor}

\sssec{}

Assume for a moment that $q$ is non-torsion-valued, so that by \thmref{t:DK vs small gen}, the map
$$\Omega^{\on{sml}}\to \Omega^{\on{DK}}$$
is an isomorphism, and by duality, the map
$$ \Omega^{\on{Lus}}\to \Omega^{\on{sml}}$$
is also an isomorphism.

\medskip

From \corref{c:prop DK char 0 non-deg} we obtain:

\begin{cor} \label{c:fibers sm non coroot}
Let $q$ be non-torsion-valued. Then for any $\lambda\in \Lambda^{\on{neg}}-0$ \emph{not} of the form $w(\rho)-\rho$, both restrictions
$\Delta_\lambda^!(\Omega_q^{\on{sml}})$ and $\Delta_\lambda^*(\Omega_q^{\on{sml}})$ vanish.
\end{cor} 

%

%

\begin{rem}
Note that in the case of non-torsion valued $q$, \thmref{t:DK vs small gen} implies that the map \eqref{e:GM to Omega} is an
 isomorphism for all $\lambda$. 

\medskip

In particular, in this case, for
$\lambda$ \emph{not} of the form $w(\rho)-\rho$ with $w\in W$, all the maps in
\begin{multline*} 
\Omega^{\on{sml}}|_{\Conf^\lambda} \simeq
\Omega^{\on{DK}}|_{\Conf^\lambda} \simeq (\jmath_\lambda)_{!*}\circ (\jmath_\lambda)^*(\Omega_q^{\on{DK}})\to \Omega_q^{\on{DK}}|_{\Conf^\lambda}\to \\
\to H^0\left((\jmath_\lambda)_*\circ (\jmath_\lambda)^*(\Omega_q^{\on{DK}})\right)\to 
(\jmath_\lambda)_*\circ (\jmath_\lambda)^*(\Omega_q^{\on{DK}})
\end{multline*} 
are isomorphisms. Indeed, the fact that the composite arrow
is an isomorphism is equivalent to the vanishing of $(\Delta_\lambda)^!(\Omega^{\on{sml}})$, which is the content of 
\corref{c:fibers sm non coroot}. 
\end{rem} 

\ssec{A sharper estimate}

In this subsection we will assume that  the inequality in (*) is sharp.

\sssec{}

We claim: 

\begin{thm} \label{t:prop DK char 0 very non-deg}
For $\lambda$ of the form $w(\rho)-\rho$ for $w\in W$ with $\ell(w)\geq 3$, the object 
$(\Delta_\lambda)^*(\Omega_q^{\on{DK}})$ is concentrated in cohomological degrees $\leq -1$. 
\end{thm}

This theorem will be proved in  \secref{ss:proof of prop}. 

\sssec{}

Let us rephrase \thmref{t:prop DK char 0 very non-deg} in terms of properties of the map \eqref{e:GM to Omega}:

\begin{cor} \label{c:prop DK char 0 very non-deg}
For $\lambda$ of the form $w(\rho)-\rho$ for $w\in W$ with $\ell(w)\geq 3$, the map
\eqref{e:GM to Omega} is an isomorphism.
\end{cor}

\begin{rem}
Thus, we obtain that if the inequality in (*) is sharp, for $\lambda$ of the form $w(\rho)-\rho$ for $w\in W$ with $\ell(w)\geq 3$,
both maps
$$(\jmath_\lambda)_{!*}\circ (\jmath_\lambda)^*(\Omega_q^{\on{DK}})\to \Omega_q^{\on{DK}}|_{\Conf^\lambda}\to
H^0\left((\jmath_\lambda)_*\circ (\jmath_\lambda)^*(\Omega_q^{\on{DK}})\right)$$
are isomorphisms. 
\end{rem}

\sssec{}

The above Theorems \ref{t:prop DK char 0 non-deg} and \ref{t:prop DK char 0 very non-deg} give an even more detailed
recipe for the inductive construction of $\Omega_q^{\on{DK}}$ (cf. \secref{sss:Omega inductive}). By induction and factorization 
we can assume that we have already constructed 
$\Omega_q^{\on{DK}}|_{X^\lambda-\Delta_\lambda(X)}$. To recover $\Omega_q^{\on{DK}}|_{X^\lambda}$ we proceed as follows: 

\medskip

\begin{itemize}

\item If $\lambda=w(\rho)-\rho$ with $\ell(w)=2$, apply $(\jmath_\lambda)_{!*}(-)$;

\item If $\lambda=w(\rho)-\rho$ with $\ell(w)\geq 3$, apply $H^0((\jmath_\lambda)_{*}(-))$,
which is the same as $(\jmath_\lambda)_{!*}(-)$;

\item If $\lambda$ is not of the form $w(\rho)-\rho$, apply $H^0((\jmath_\lambda)_{*}(-))$,
which is the same as $(\jmath_\lambda)_{*}(-)$ (if $q$ is non-torsion valued, this is also the same as 
$(\jmath_\lambda)_{!*}(-)$).

\end{itemize}

\section{Factorization algebras via quantum groups}  \label{s:quantum}

In this section we will work with constructible sheaves in the classical topology with coefficients
in a field $\sfe$ (assumed algebraically closed and of characteristic $0$). 

\medskip

We will describe a relationship between the (three versions of the) factorization algebra $\Omega_q$ 
and the corresponding versions of the quantum group. 

\ssec{Quantum Hopf algebras} \label{ss:Hopf}
  
\sssec{}

Let $\Vect^\Lambda$ denote the category of $\Lambda$-graded vector spaces,
which is the same as the category $\Rep(\cT)$, of algebraic representations of the torus $\cT$.

\medskip

For $\lambda\in \Lambda$ we denote by $\sfe^\lambda$ the corresponding object of $\Vect^\Lambda$:
this is a copy of $\sfe$ placed in coweight component $\lambda$. 

\medskip

For an on object $V\in \Vect^\Lambda$ and $\lambda\in \Lambda$, we will let $(V)^\lambda$ denote its
$\lambda$-coweight component, i.e.,
$$(V)^\lambda=\CHom_{\Vect^\Lambda}(\sfe^\lambda,V).$$

\sssec{}

Choose a bilinear form 
$$b':\Lambda\otimes \Lambda\to \sfe^\times$$
so that $q(\lambda)=b'(\lambda,\lambda)$.

\medskip

Such a form $b'$ defines a braiding on 
$$\Rep(\cT)\simeq \Vect^\Lambda,$$
viewed as a monoidal category. Explicitly, the braiding automorphism
$$\sfe^{\lambda+\mu}=\sfe^\lambda\otimes \sfe^\mu \overset{R}\to \sfe^\mu\otimes \sfe^\lambda=
\sfe^{\mu+\lambda}$$ 
is given by multiplication by $b'(\lambda,\mu)$. 

\medskip

Denote the resulting braided monoidal category by $\Rep_q(\cT)$. 

\begin{rem}
It follows from \cite[Sects. C.2.3]{GLys1} 
that $\Rep_q(\cT)$, viewed as a braided monoidal category, depends only on $q$ (and not the choice of $b'$), up to a 
\emph{canonical} equivalence. A choice of $b'$ serves to identify the underlying monoidal category with $\Rep(\cT)$.
\end{rem} 

\sssec{}  \label{sss:conditions on Hopf}

We will consider a certain type of Hopf algebras in $\Rep_q(\cT)$. Namely, we will impose the following conditions
on our Hopf algebra (denoted $H$): 

\begin{itemize} 

\item All weight components $H^\lambda$ lie in the heart of the t-structure and 
are finite-dimensional;

\item If the weight component $H^\lambda$ is non-zero, then $\lambda\in \Lambda^{\on{pos}}$;

\item The unit map $1:\sfe\to H^0$ is an isomorphism.

\end{itemize} 

\sssec{} \label{sss:duality Hopf}

We note that component-wise duality assigns to such a Hopf algebra $H$ its dual, to be denoted $H^\vee$, which
is naturally a Hopf algebra in $\Rep_{q^{-1}}(\cT)$;
$$(H^\vee)^\lambda=(H^\lambda)^\vee.$$

\ssec{The free, cofree and small versions of the positive part of the quantum group}  \label{ss:quantum groups}

Throughout this subsection we will be assuming that
$q$ is non-degenerate (see \secref{sss:root of unity} for what this means)\footnote{See, however, Remark \ref{r:small torsion bad}.}. 

\medskip

We will consider some particular Hopf algebras in $\Rep_q(\cT)$, attached to our root system. 

\sssec{}

The first one is the free algebra,
denoted $U^{\on{free}}_q(\cN)$. As an associative algebra, it is the free associative algebra on the 
object
$$\underset{i}\oplus\, \sfe^{\alpha_i}\in \Rep_q(\cT).$$

In other words, it is defined by the universal property that 
\begin{equation} \label{e:univ pr free}
\Hom_{\on{AssAlg}(\Rep_q(\cT))}(U^{\on{free}}_q(\cN),A)=\Hom_{\Rep_q(\cT)}(\underset{i}\oplus\, \sfe^{\alpha_i},A)\simeq \underset{i}\Pi\, (A)^{\alpha_i}.
\end{equation} 

In particular, we have the canonical maps
\begin{equation} \label{e:Chev gen}
e_i:\sfe^{\alpha_i}\to U^{\on{free}}_q(\cN),
\end{equation}
which we will refer to as the ``generators". 

\medskip

The Hopf algebra structure on $U^{\on{free}}_q(\cN)$ is determined by the requirement that the co-multiplication map
$$\on{comult}:U^{\on{free}}_q(\cN)\to U^{\on{free}}_q(\cN)\otimes U^{\on{free}}_q(\cN),$$
viewed as a map in $\on{AssAlg}(\Rep_q(\cT))$, corresponds under \eqref{e:univ pr free} to the maps
$$e_i\otimes 1+1\otimes e_i: \sfe^{\alpha_i}\to U^{\on{free}}_q(\cN)\otimes U^{\on{free}}_q(\cN),$$
where $1$ stands for the unit map $\sfe\to U^{\on{free}}_q(\cN)$.

\sssec{}

We set $U^{\on{cofree}}_q(\cN)$ to be dual (in the sense of \secref{sss:duality Hopf}) of the Hopf algebra 
$U^{\on{free}}_{q^{-1}}(\cN)$ in $\Rep_{q^{-1}}(\cT)$. 

\medskip

As a co-associative co-algebra it has the universal property that
\begin{equation} \label{e:univ pr cofree}
\Hom_{\on{CoAssCoAlg}(\Rep_q(\cT))}(A,U^{\on{cofree}}_q(\cN))=\Hom_{\Rep_q(\cT)}(A,\sfe^{\alpha_i})\simeq \underset{i}\Pi\, 
((A)^{\alpha_i})^\vee.
\end{equation} 

In particular, we have the canonical maps
\begin{equation} \label{e:cogen}
e_i^\vee: U^{\on{cofree}}_q(\cN)\to \sfe^{\alpha_i}.
\end{equation} 

\medskip

The Hopf algebra structure on $U^{\on{cofree}}_q(\cN)$ is determined by the requirement that the multiplication map
$$\on{mult}:U^{\on{cofree}}_q(\cN)\otimes U^{\on{cofree}}_q(\cN)\to U^{\on{cofree}}_q(\cN)$$
corresponds under \eqref{e:univ pr cofree} to the maps
$$e^\vee_i\otimes 1^\vee+1^\vee\otimes e^\vee_i:U^{\on{cofree}}_q(\cN)\otimes U^{\on{cofree}}_q(\cN)\to \sfe^{\alpha_i},$$
where $1^\vee$ denotes the co-unit map 
$$U^{\on{cofree}}_q(\cN)\to \sfe.$$

\sssec{}

It is easy to see that the map $e_i$ of \eqref{e:Chev gen} defines an isomorphism of $\alpha_i$-coweight 
components. We let $e_i^\vee$ denote the resulting map
\begin{equation} \label{e:co-gen}
U^{\on{free}}_q(\cN)\to \sfe^{\alpha_i},
\end{equation}
the projection onto the $\alpha_i$-coweight component.

\medskip

We have a commutative diagram
$$
\CD
U^{\on{free}}_q(\cN)\otimes U^{\on{free}}_q(\cN)  @>{e^\vee_i\otimes 1^\vee+1^\vee\otimes e^\vee_i}>>  \sfe^{\alpha_i}  \\
@VVV   @VV{\on{id}}V \\
U^{\on{free}}_q(\cN)   @>{e^\vee_i}>>  \sfe^{\alpha_i},
\endCD
$$
where $1^\vee$ denotes the co-unit map $U^{\on{free}}_q(\cN)\to \sfe$.

\medskip

Similarly, it is easy to see that the map $e_i^\vee$ of \eqref{e:cogen} defines isomorphisms
on $\alpha_i$-coweight components. 
In particular, we have the canonical map
$$e_i:\sfe^{\alpha_i}\to U^{\on{cofree}}_q(\cN),$$
that makes the diagrams
\begin{equation} \label{e:gen for cofree}
\CD
\sfe^{\alpha_i}    @>{e_i}>>   U^{\on{cofree}}_q(\cN)) \\
@V{\on{id}}VV   @VV{\on{comult}}V   \\
\sfe^{\alpha_i}  @>{e_i\otimes 1+1\otimes e_i}>>  U^{\on{cofree}}_q(\cN)\otimes U^{\on{cofree}}_q(\cN)
\endCD
\end{equation} 
commute. 

\sssec{}

We now claim that there exists a canonically defined map of Hopf algebras
\begin{equation} \label{e:free to cofree}
U^{\on{free}}_q(\cN)\to U^{\on{cofree}}_q(\cN).
\end{equation} 

As a map of associative algebras, \eqref{e:free to cofree} is determined by the requirement that the diagrams
$$
\CD
\sfe^{\alpha_i}    @>{e_i}>>   U^{\on{free}}_q(\cN) \\
@VVV   @VVV    \\
\sfe^{\alpha_i}    @>{e_i}>>   U^{\on{cofree}}_q(\cN)
\endCD
$$
commute. 

\medskip

The compatibility with the coalgebra structure is insured by the diagrams \eqref{e:gen for cofree}. 

\sssec{}

It follows that in terms of \eqref{e:univ pr cofree}, the map \eqref{e:free to cofree} corresponds to the maps  $e_i^\vee$
of \eqref{e:co-gen}. 

\medskip

We obtain that the duality functor of \secref{sss:duality Hopf} sends the map \eqref{e:free to cofree} to the map
$$U^{\on{free}}_{q^{-1}}(\cN)\to U^{\on{cofree}}_{q^{-1}}(\cN)$$
in $\Rep_{q^{-1}}(\cT)$. 

\sssec{}

Let $u_q(\cN)$ be the Hopf algebra in $\Rep_q(\cT)$ equal to the image of the map \eqref{e:free to cofree}.

\medskip

This is the small version of (the positive part of) the quantum group.

\sssec{} \label{sss:Serre}

Recall that for a pair of vertices of the Dynkin diagram $i,j$ one can attach a canonical element
$$S^{i,j}_q\in U^{\on{free}}_q(\cN)$$
called the \emph{quantum Serre relation}, whose coweight is
$$\alpha_i+(1-\langle \alpha_i,\check\alpha_j\rangle)\in \Lambda^{\on{pos}}.$$

We normalize $S^{i,j}_q$ so that it has the form
$$e_i^{1-\langle \alpha_i,\check\alpha_j\rangle}\cdot e_j+...+e_j\cdot e_i^{1-\langle \alpha_i,\check\alpha_j\rangle},$$

\medskip

The following is established in \cite{Ro}:

\begin{lem} \label{l:quantum Serre}
Let $q$ be \emph{non-torsion valued}. Then $u_q(\cN)$ is the quotient of $U^{\on{free}}_q(\cN)$ by the two-sided 
ideal $\CI_q$ generated by the elements $S^{i,j}_q$ for all pairs of vertices $i,j\in I$.
\end{lem}

\ssec{The DeConcini-Kac and Luztig's versions of the positive part of the quantum group}  \label{ss:quantum groups bis}

In order to define the Hopf algebras $U_q^{\on{DK}}(\cN)$ and $U_q^{\on{Lus}}(\cN)$, we will need to consider the deformation
of $q$ as in \secref{ss:constr by deformation}. 

\sssec{}

Recall the quadratic form 
$$q_\hbar:\Lambda \to (\sfe\hbart)^\times.$$

We introduce the category $\Rep_{q_\hbart}(\cT)$ and its localization $\Rep_{q_\hbarl}(\cT)$ by the recipe
of \secref{sss:intr hbar}, applied to the \emph{small} category 
$$\Rep(\cT)^{\on{loc.fin.dim.}}= (\Vect^{\on{fin.dim.}})^\Lambda$$
where the superscript ``loc.fin.dim." means that we are considering representations with each weight
component finite-dimensional.

\medskip

The procedure of \secref{sss:t on hbar} endows these categories with a t-structure,
so that the analog of \propref{p:t h} holds. 

\medskip

Note that as abstract DG categories, we have:
$$\Rep_{q_\hbart}(\cT) \simeq ((\sfe\hbart\mod)^{\on{f.g.}})^\Lambda \text{ and }
\Rep_{q_\hbarl}(\cT) \simeq ((\Vect_{\sfe\hbarl})^{\on{fin.dim.}})^\Lambda,$$
respectively, equipped with their natural t-structures. 

\medskip

The role of $q_\hbar$ is that it defines a braided structure on these categories. In particular, we obtain 
$$\Rep_{q_\hbarl}(\cT)\simeq \Rep_{q'}(\cT)^{\on{loc.fin.dim.}},$$
where the latter is the variant of the category $\Rep_{q'}(\cT)$
over the field $\sfe\hbarl$ with $q'=q_\hbar$. 

\sssec{}

Consider the resulting Hopf algebras in $\Rep_{q_\hbarl}(\cT)$:
$$U^{\on{free}}_{q_\hbarl}(\cN) \text{ and } U^{\on{cofree}}_{q_\hbarl}(\cN)$$
and the map 
\begin{equation} \label{e:free to cofree laurent}
U^{\on{free}}_{q_\hbarl}(\cN) \to U^{\on{cofree}}_{q_\hbarl}(\cN).
\end{equation}

Set $U_{q_\hbarl}(\cN)$ be the image of the map \eqref{e:free to cofree laurent}. As in \lemref{l:quantum Serre}, 
the kernel $\CI_{q_\hbarl}$ of the projection
$$U^{\on{free}}_{q_\hbarl}(\cN) \to U_{q_\hbarl}(\cN)$$
is generated by the corresponding elements
$$S^{i,j}_{q_\hbarl}\in U^{\on{free}}_{q_\hbarl}(\cN).$$

\sssec{}

Consider now the Hopf algebra $U^{\on{free}}_{q_\hbart}(\cN)$ in $\Rep_{q_\hbart}(\cT)$, and set
$$\CI_{q_\hbart}:=U^{\on{free}}_{q_\hbart}(\cN)\cap \CI_{q_\hbarl},$$
where the intersection is taking place in $U^{\on{free}}_{q_\hbarl}(\cN)$.

\medskip

Then $\CI_{q_\hbart}\subset U^{\on{free}}_{q_\hbart}(\cN)$ is a Hopf ideal, and we define
$$U^{\on{DK}}_{q_\hbart}(\cN):=U^{\on{free}}_{q_\hbart}(\cN)/\CI_{q_\hbart}.$$

By construction, $U^{\on{DK}}_{q_\hbart}(\cN)$ is torsion-free, and hence flat, as a module over $\sfe\hbart$.
In particular, for every $\lambda\in \Lambda$, the corresponding coweight component
$$(U^{\on{DK}}_{q_\hbart}(\cN))^\lambda$$
is a free module of finite rank over $\sfe\hbart$. Moreover this 
rank equals 
$$\dim_{\sfe\hbarl}\left((U^{\on{DK}}_{q_\hbarl}(\cN))^\lambda\right)=\dim_{\sfe}\left(U(\cn)^\lambda\right),$$
where $U(\cn)$ is the usual universal enveloping algebra of $\cn$. 

\sssec{}

The map 
\begin{equation} \label{e:free to cofree taylor}
U^{\on{free}}_{q_\hbart}(\cN) \to U^{\on{cofree}}_{q_\hbart}(\cN)
\end{equation}
factors as
\begin{equation} \label{e:free to cofree taylor DK}
U^{\on{free}}_{q_\hbart}(\cN) \twoheadrightarrow U^{\on{DK}}_{q_\hbart}(\cN)\to U^{\on{cofree}}_{q_\hbart}(\cN).
\end{equation}

Note that the map
$$U^{\on{DK}}_{q_\hbart}(\cN)\to U^{\on{cofree}}_{q_\hbart}(\cN)$$
is injective, but its cokernel has $\hbar$-torsion, so it ceases to be injective mod $\hbar$. 

\begin{rem} \label{r:when KD is KD taylor}
By construction, the elements $S^{i,j}_{q_\hbarl}$ actually belong to $U^{\on{free}}_{q_\hbart}(\cN)$. When viewed 
in this capacity, we will denote them by $S^{i,j}_{q_\hbart}$, and they belong to the ideal $\CI_{q_\hbart}$.

\medskip

In \secref{sss:proof KD is KD} we will show that under the assumption that $q$ satisfies (*) (see \secref{sss:dual Cox})
the elements $S^{i,j}_{q_\hbart}$ generate $\CI_{q_\hbart}$ as a two-sided ideal.

\medskip

Hence, in this case, we can explicitly write $U^{\on{DK}}_{q_\hbart}(\cN)$ as the quotient of
$U^{\on{free}}_{q_\hbart}(\cN)$ by the elements $S^{i,j}_{q_\hbart}$.

\end{rem}

\sssec{}

Define 
$$U^{\on{DK}}_q(\cN):=U^{\on{DK}}_{q_\hbart}(\cN)/\hbar\simeq U^{\on{DK}}_{q_\hbart}(\cN)\underset{\sfe\hbart}\otimes \sfe,$$
the latter isomorphism due to the fact that $U^{\on{DK}}_{q_\hbart}(\cN)$ is $\sfe\hbart$-flat. 

\medskip

The factorization \eqref{e:free to cofree taylor DK} implies that the maps
\begin{equation} \label{e:string Hopf short}
U^{\on{free}}_q(\cN)\twoheadrightarrow  u_q(\cN) \hookrightarrow U^{\on{cofree}}_q(\cN)
\end{equation}
factor as
\begin{equation} \label{e:string Hopf left}
U^{\on{free}}_q(\cN)\twoheadrightarrow  U^{\on{DK}}_q(\cN) \twoheadrightarrow
u_q(\cN) \hookrightarrow U^{\on{cofree}}_q(\cN).
\end{equation}

\begin{rem}  \label{r:when KD is KD}
From Remark \ref{r:when KD is KD taylor} we obtain that under the assumption that $q$ satisfies (*), 
$U^{\on{DK}}_q(\cN)$ equals the quotient of $U^{\on{free}}_q(\cN)$ by the two-sided ideal generated 
by the elements $S^{i,j}_q$.

\medskip

In general, for a given value of $q$, the map
\begin{equation} \label{e:mod Serre}
U^{\on{free}}_q(\cN)/\langle S^{i,j}_q\rangle\to U^{\on{DK}}_q(\cN)
\end{equation} 
is an isomorphism \emph{if and and only if} $U^{\on{free}}_q(\cN)/\langle S^{i,j}_q\rangle$ has the right
size, i.e., if
$$\dim\left((U^{\on{free}}_q(\cN)/\langle S^{i,j}_q \rangle)^\lambda\right)=\dim\left(U(\cn)^\lambda\right), \quad \lambda \in \Lambda^{\on{pos}}.$$

\end{rem}

\sssec{}

As a (trivial) particular case, from \lemref{l:quantum Serre} we obtain:

\begin{cor} \label{c:quantum non-coroot}
Let $q$ be non-torsion valued. Then the map
$$U^{\on{DK}}_q(\cN) \twoheadrightarrow u_q(\cN)$$
is an isomorphism.
\end{cor}

\sssec{}

We define the Hopf algebra $U_q^{\on{Lus}}(\cN)$ as the dual (in the sense of \secref{sss:duality Hopf}) of 
the Hopf algebra $U_{q^{-1}}^{\on{DK}}(\cN)$ in $\Rep_{q^{-1}}(\cT)$. 

\medskip

By duality, from \eqref{e:string Hopf left} we obtain that the maps in \eqref{e:string Hopf short} factor also as
$$U^{\on{free}}_q(\cN)\twoheadrightarrow  u_q(\cN) \hookrightarrow U_q^{\on{Lus}}(\cN)  \hookrightarrow
U^{\on{cofree}}_q(\cN).$$

To summarize we have the following string of maps of Hopf algebras
\begin{equation} \label{e:string Hopf}
U^{\on{free}}_q(\cN)\twoheadrightarrow U_q^{\on{DK}}(\cN)\twoheadrightarrow u_q(\cN)
\hookrightarrow U_q^{\on{Lus}}(\cN) \hookrightarrow U^{\on{cofree}}_q(\cN).
\end{equation} 

\sssec{}

Finally, from \corref{c:quantum non-coroot} we obtain:

\begin{cor} \label{c:quantum non-coroot bis}
Let $q$ be \emph{non-torsion valued}. Then both maps 
$$U_q^{\on{DK}}(\cN)\twoheadrightarrow u_q(\cN) \text{ and } u_q(\cN)\hookrightarrow U^{\on{Lus}}_q(\cN)$$
are isomorphisms. 
\end{cor}

\ssec{From Hopf algebras to factorization algebras}  \label{ss:Hopf to Fact}

In this subsection we summarize the construction from \cite[Sect. 29.5]{GLys2}. 

\sssec{}

Let $X$ be $\BA^1$ with a chosen coordinate. Let $\CG_q^\Conf$ be the $\sfe^\times$-gerbe attached to the 
quadratic form $q:\Lambda\to \sfe^\times$. 

\medskip

Note that, according to formula \eqref{e:gerbe on lambda}, our choice of the coordinate on $\BA^1$ gives rise to
a trivialization of the gerbe $\CG^\lambda_q$ for every $\lambda\in \Lambda^{\on{neg}}-0$. 

\sssec{}

The construction of \cite[Sect. 29.5]{GLys2} defines a \emph{contravariant} equivalence between the category of Hopf algebras in $\Rep_q(\cT)$,
satisfying the conditions of \secref{sss:conditions on Hopf} and that of factorization algebras in $\on{Perv}_{\CG_q^\Conf}(\Conf)$
\begin{equation} \label{e:from Hopf to fact}
H \mapsto \Omega_H,
\end{equation} 
which has the following properties.

\sssec{}  \label{sss:fibers and cofibers of Omega}

For $\lambda\in \Lambda^{\on{neg}}-0$, let $\iota_\lambda$ denote the embedding of the point $\lambda\cdot 0$ into 
$\Conf^\lambda$, cf. \secref{sss:Delta lambda}.
(Note that the above trivialization of the gerbe $\CG^\lambda_q$ allows to view fibers and cofibers of
objects of $\Shv_{\CG_q^\Conf}(\Conf)$ at $\lambda\cdot 0\in \Conf^\lambda$ as objects of $\Vect$.)

\medskip

Let $H\mod(\Rep_q(\cT))$ denote the category of left $H$-modules in $\Rep_q(\cT)$, where $H$ is viewed as an associative algebra.
Note that the operation of tensor product on the right defines an action of $\Rep_q(\cT)$ on $H\mod(\Rep_q(\cT))$, 
$$\CM,\sfe^\lambda \mapsto \CM\otimes \sfe^\lambda.$$

\medskip

For $\CM_1,\CM_2\in H\mod$, let
$\uHom_H(\CM_1,\CM_2)$ denote the $\Lambda$-graded vector space whose $\lambda$-component is 
$$\CHom_{H\mod(\Rep_q(\cT))}(\CM_1\otimes \sfe^\lambda,\CM_2).$$

\sssec{}

The first property of the functor \eqref{e:from Hopf to fact} says that
\begin{equation} \label{e:!-fibers of Omega}
\iota_\lambda^!(\Omega_H)\simeq (\uHom_H(\sfe,\sfe))^\lambda.
\end{equation}

Note that the RHS in \eqref{e:!-fibers of Omega} identifies canonically with the $\lambda$-component of the 
cochain complex of $H$,
$$\on{C}^\cdot(H),$$
where $H$ is viewed as a plain $\Lambda^{\on{pos}}$-graded associative augmented algebra via the (monoidal!) functor 
$$\Rep_q(\cT)\to \Rep(\cT)=\Vect^\Lambda \to \Vect.$$ 

\sssec{}  \label{sss:Omega and duality}

The second property of the functor \eqref{e:from Hopf to fact}
states that we have a canonical isomorphism (functorial in $H$):
\begin{equation} \label{e:duality omega}
\BD^{\on{Verdier}}(\Omega_H)\simeq \Omega_{H^\vee},
\end{equation}
where both sides are viewed as factorization algebras in $\on{Perv}_{\CG_{q^{-1}}^\Lambda}(\Conf)$. 

\sssec{}

Combining with \eqref{e:!-fibers of Omega} we obtain that the following expression for the *-fibers of $\Omega(H)$:
\begin{equation} \label{e:*-fibers of Omega}
\iota_\lambda^*(\Omega_H)\simeq \left((\uHom_{H^\vee}(\sfe,\sfe))^\lambda\right)^\vee \simeq
\left((\on{C}^\cdot(H^\vee))^\lambda\right)^\vee\simeq (\on{C}_\cdot(H^\vee))^{-\lambda},
\end{equation}
where $\on{C}_\cdot(-)$ stands for the chain complex of a given associative augmented algebra. 

\sssec{}

For a given $\lambda\in \Lambda^{\on{neg}}-0$, let
$$(\iota_\lambda)^{!*}:\Shv_{\CG_q^\Conf}(\Conf)\to \Vect$$
be the functor of \emph{hyperbolic restriction} along $\iota_\lambda$. This functor is a feature of the topological
setting, and it is defined as the composite of the following two functors:

\medskip

The first functor is  *-restriction along
$$\Conf_\BR \hookrightarrow  \Conf,$$
where $\Conf_\BR$ is the \emph{topological submanifold} that corresponds to \emph{real configurations}, i.e., we take points
$$\underset{k}\Sigma\, \lambda_k\cdot x_k\in \Conf,$$
where all $x_k$ belong to $\BR\subset \BC=\BA^1$. 

\medskip

The second functor is !-restriction along 
$$\on{pt} \to \Conf_\BR,$$
corresponding to the point $\lambda\cdot 0$.

\medskip

Note that the functor $(\iota_\lambda)^{!*}$, when applied to the full subcategory of $\Shv_{\CG}(\Conf^\lambda)$, consisting of objects 
constructible with respect to the diagonal stratification, is \emph{t-exact} and \emph{conservative}.

\sssec{}

The third property of the functor \eqref{e:from Hopf to fact} states that we have a canonical isomorphism
\begin{equation} \label{e:hyprbolic of Omega}
(\iota_\lambda)^{!*}(\Omega_H)\simeq (H^{-\lambda})^\vee.
\end{equation} 

\sssec{}

The isomorphism \eqref{e:hyprbolic of Omega}, combined with the properties of $(\iota_\lambda)^{!*}$ implies:

\begin{cor} \label{c:inj/surj Omega}
If a map of Hopf algebras $H_1\to H_2$ is injective/surjective, then the corresponding map in $\on{Perv}_{\CG_q^\Conf}(\Conf)$
$$\Omega_{H_2}\to \Omega_{H_1}$$
is surjective/injective.
\end{cor} 

\ssec{Factorization algebras attached to quantum groups}

In this subsection we will apply the functor  \eqref{e:from Hopf to fact} to the Hopf algebras associated with quantum groups
from Sects. \ref{ss:quantum groups} and \ref{ss:quantum groups bis}.

\medskip

In particular, we will again impose the assumption that $q$ avoids small torsion. 

\sssec{}

Denote
$$\Omega^{\on{free}}_{q,\Quant}:=\Omega_{U^{\on{free}}_q(\cN)},\,\, \Omega^{\on{cofree}}_{q,\Quant}:=\Omega_{U^{\on{cofree}}_q(\cN)},$$
$$\Omega^{\on{sml}}_{q,\Quant}:=\Omega_{u_q(\cN)},\,\, \Omega^{\on{DK}}_{q,\Quant}:=\Omega_{U^{\on{DK}}_q(\cN)},\,\,
\Omega^{\on{Lus}}_{q,\Quant}:=\Omega_{U^{\on{Lus}}_q(\cN)}.$$

\medskip

First, we claim:

\begin{prop}  \label{p:identify Omega free}
There exists a canonical identification
\begin{equation} \label{e:identify Omega free}
\Omega^{\on{free}}_{q,\Quant}\simeq j_*(\oOmega_q)
\end{equation} 
as factorization algebras. 
\end{prop} 

\begin{proof} 

First we establish an identification
$$j^*(\Omega^{\on{free}}_{q,\Quant})\simeq \oOmega_q.$$

By factorization, it is enough to establish the isomorphisms
$$\Omega^{\on{free}}_{q,\Quant}|_{\Conf^{-\alpha_i}}\simeq \sfe_{\BA^1}[1]$$
for individual simple coroots. By translation invariance, it suffices to establish the isomorphisms
$$(\iota_{-\alpha_i})^!(\Omega^{\on{free}}_{q,\Quant})\simeq \sfe[-1].$$

However, this follows from formula \eqref{e:!-fibers of Omega}: indeed, we have
\begin{equation} \label{e:cohomology of free}
\on{C}^\cdot(U_q^{\on{free}}(\cN)) \simeq \underset{i}\oplus\, \sfe^{-\alpha_i}[-1].
\end{equation}

To prove \eqref{e:identify Omega free}, it remains to show that the !-restriction of 
$\Omega^{\on{free}}_{q,\Quant}$ to the complement of $\oConf$ in $\Conf$ vanishes. 
By induction and factorization, it suffices to show that for $\lambda$ which is not one
of the negative simple coroots, we have 
$$(\Delta_\lambda)^!(\Omega^{\on{free}}_{q,\Quant})=0.$$

Again, by translation invariance, this is equivalent to showing that
$$(\iota_\lambda)^!(\Omega^{\on{free}}_{q,\Quant})=0.$$

However, this again follows from \eqref{e:!-fibers of Omega} and \eqref{e:cohomology of free}. 

\end{proof}

Combining \propref{p:identify Omega free} with \eqref{e:duality omega}, we obtain: 

\begin{cor} \label{c:identify Omega cofree}
There exists a canonical identification as factorization algebras
\begin{equation} \label{e:identify Omega cofree}
\Omega^{\on{cofree}}_{q,\Quant}\simeq j_!(\oOmega_q)
\end{equation}
\end{cor} 

Combining this with \corref{c:inj/surj Omega}, we obtain: 

\begin{cor} \label{c:Omega DK and Lus}
There exist canonically defined injective maps 
$$\Omega^{\on{sml}}_{q,\Quant}\hookrightarrow  \Omega^{\on{DK}}_{q,\Quant}\hookrightarrow  j_*(\oOmega_q)$$
and canonically defined surjective maps
$$j_!(\oOmega_q)\twoheadrightarrow  \Omega^{\on{Lus}}_{q,\Quant}\twoheadrightarrow  \Omega^{\on{sm}}_{q,\Quant}.$$
All of these maps respect the factorization algebra structure. 
\end{cor}

\sssec{}

Note that on the one hand, we can consider the tautological map
\begin{equation} \label{e:from ! to *}
j_!(\oOmega_q)\to  j_*(\oOmega_q).
\end{equation}

On the other hand, consider the map
\begin{equation} \label{e:from ! to * quant}
\Omega^{\on{cofree}}_{q,\Quant}\to \Omega^{\on{free}}_{q,\Quant},
\end{equation} 
obtained from the map \eqref{e:free to cofree} by functoriality. We claim: 

\begin{lem} \label{l:from ! to *}
The diagram
$$
\CD
j_!(\oOmega_q)  @>{\text{\eqref{e:from ! to *}}}>>   j_*(\oOmega_q)  \\
@V{\text{\eqref{e:identify Omega cofree}}}VV   @VV{\text{\eqref{e:identify Omega free}}}V  \\
\Omega^{\on{cofree}}_{q,\Quant}  @>{\text{\eqref{e:from ! to * quant}}}>>   \Omega^{\on{free}}_{q,\Quant}
\endCD
$$
commutes \emph{up to a non-zero constant} over each $\Conf^\lambda$.
\end{lem} 

\begin{proof}

By adjunction, we need to show that the diagram
$$
\CD 
\oOmega_q  @>{\on{id}}>>   \oOmega_q  \\ 
@V{\sim}VV   @VV{\sim}V   \\
j^*(\Omega^{\on{cofree}}_{q,\Quant})  @>>>  j^*(\Omega^{\on{free}}_{q,\Quant})
\endCD
$$
commutes \emph{up to a non-zero constant} over each $\Conf^\lambda$. This is evident
since both sides are irreducible. 

%
%

\end{proof} 

\begin{rem}
One can show that the constant in \lemref{l:from ! to *} equals $1$. 
\end{rem} 

Invoking \corref{c:inj/surj Omega} again, we obtain:

\begin{cor} \label{c:identify Omega small}
There exists an identification
$$\Omega^{\on{sml}}_{q,\Quant}\simeq j_{!*}(\oOmega_q)=:\Omega^{\on{sml}}_q.$$
\end{cor}

Finally, we obtain: 

\begin{cor} 
There exists a string of maps of factorization algebras 
\begin{equation} \label{e:five fact alg}
j_!(\oOmega_q)\twoheadrightarrow  \Omega^{\on{Lus}}_{q,\Quant} \twoheadrightarrow \Omega^{\on{sml}}_{q,\Quant}
\hookrightarrow \Omega^{\on{DK}}_{q,\Quant}\hookrightarrow  j_*(\oOmega_q),
\end{equation} 
where all the maps become isomorphisms when restricted to $\oConf$. 
\end{cor} 

Note that the maps in \eqref{e:five fact alg} are counterparts of the maps in \eqref{e:string Hopf}.

\sssec{}

Let us note the following corollary of \corref{c:quantum non-coroot bis}:

\begin{cor} \label{c:quantum non-coroot fact}
Assume that $q$ is not a root of unity. Then the maps
$$\Omega^{\on{Lus}}_{q,\Quant} \twoheadrightarrow \Omega^{\on{sml}}_{q,\Quant}
\hookrightarrow \Omega^{\on{DK}}_{q,\Quant}$$
are isomorphisms.
\end{cor} 

\ssec{The first comparison: $\Omega^{\on{KD}}_{q,\Quant}$ vs abstract $\Omega^{\on{KD}}_q$}

\sssec{}

Our current goal is to prove the following result:

\begin{thm} \label{t:quant and abs} 
The embeddings
$$\Omega^{\on{DK}}_{q,\Quant}\hookrightarrow j_*(\oOmega_q)\hookleftarrow \Omega_q^{\on{DK}}$$
have equal images.
\end{thm}

The rest of this subsection is devoted to the proof of this theorem.

%

%

\sssec{}

We will consider the class of Hopf algebras in $\Rep_{q_\hbart}(\cT)$ as in \secref{sss:conditions on Hopf}, where the first
condition is replaced by the following one:

\begin{itemize}

\item Each coweight component $H^\lambda$ lies in the heart of $(\sfe\hbart\mod)^{\on{f.g.}}$ and is $\hbar$-flat.

\end{itemize}

The equivalence \eqref{e:from Hopf to fact} extends to a (contravariant) equivalence between the category of such 
Hopf algebras and that of factorization algebras in $\Perv_{\CG_{q_{\hbart}}}(\Conf)$ that are $\hbar$-flat
(see \secref{sss:hflat} for what the latter means). 

\begin{rem}
The flatness condition in the statement of the equivalence is due to the fact that our assignment $H\mapsto \Omega_H$
is contravariant. If we worked with the covariant version, we could have omitted the
flatness condition on both sides. 
\end{rem}

\sssec{}

Consider the Hopf algebras
$$U_{q_\hbart}^{\on{free}}(\cN)\twoheadrightarrow U_{q_\hbart}^{\on{DK}}(\cN)$$
as well as 
$$U_{q_\hbarl}^{\on{free}}(\cN)\twoheadrightarrow U_{q_\hbarl}^{\on{DK}}(\cN) \hookrightarrow U_{q_\hbarl}^{\on{cofree}}(\cN).$$

Consider the corresponding factorization algebras
$$\Omega^{\on{DK}}_{q_\hbart,\Quant}\hookrightarrow \Omega^{\on{free}}_{q_\hbart,\Quant}\simeq j_*(\oOmega_{q_\hbart})$$
and 
$$j_!(\oOmega_{q_\hbarl})\overset{\sim}\to 
\Omega^{\on{cofree}}_{q_\hbarl,\Quant} \twoheadrightarrow \Omega^{\on{DK}}_{q_\hbarl,\Quant}\hookrightarrow 
\Omega^{\on{free}}_{q_\hbarl,\Quant}\simeq j_*(\oOmega_{q_\hbarl}).$$

\medskip

In particular, we obtain an isomorphism
$$\Omega^{\on{DK}}_{q_\hbarl,\Quant}\simeq \Omega_{q_\hbarl},$$
(see \secref{sss:Omega hbar}, where the latter is defined). 

\sssec{}

The resulting maps
$$\Omega^{\on{DK}}_{q_\hbart,\Quant} \hookrightarrow j_*(\oOmega_{q_\hbart})$$
and
$$\Omega^{\on{DK}}_{q_\hbart,\Quant} \hookrightarrow \Omega^{\on{DK}}_{q_\hbarl,\Quant} \simeq 
\Omega_{q_\hbarl}$$
define a map
\begin{equation} \label{e:DK to abs hbar}
\Omega^{\on{DK}}_{q_\hbart,\Quant}\hookrightarrow \Omega^{\on{DK}}_{q_\hbart}.
\end{equation} 

We claim that the map \eqref{e:DK to abs hbar} is an isomorphism. Once this is proved, \thmref{t:quant and abs} would
follow by reduction mod $\hbar$. 

\sssec{}

We know that it \eqref{e:DK to abs hbar} is an isomorphism after
inverting $\hbar$. Hence, its cokernel is $\hbar$-torsion. Therefore, if this cokernel were non-zero,
the map
$$\Omega^{\on{DK}}_{q_\hbart,\Quant}/\hbar \to \Omega^{\on{DK}}_{q_\hbart}/\hbar$$
would \emph{not} be injective. However, this would be a contradiction as the latter map is a map
$$\Omega^{\on{DK}}_{q,\Quant}\to \Omega^{\on{DK}}_q,$$
whose composition with the map 
$$\Omega^{\on{DK}}_q\to j_*(\oOmega_q)$$
is the embedding 
$$\Omega^{\on{DK}}_{q,\on{Quant}}\to j_*(\oOmega_q).$$

\qed[\thmref{t:quant and abs}]

\ssec{Proof of the properties of $\Omega^{\on{KD}}_q$ in characteristic $0$} \label{ss:proof of prop}

Having established the isomorphism $\Omega^{\on{DK}}\simeq \Omega^{\on{DK}}_{q,\Quant}$, we can now prove 
Theorems \ref{t:DK vs small gen}, \ref{t:prop DK char 0 prim}, \ref{t:prop DK char 0 sec}, 
\ref{t:prop DK char 0 non-deg} and \ref{t:prop DK char 0 very non-deg}. 

\sssec{}

By the \'etale invariance of the construction (see Remark \ref{r:transfer to A1}), we can assume that $X=\BA^1$. 

\medskip

By Lefschetz principle and Riemann-Hilbert,
it is enough to do so in the context of the Betti sheaf theory. So we will prove the corresponding assertions for 
$\Omega^{\on{DK}}_{q,\Quant}$.

\sssec{}

First off, the assertion of \thmref{t:DK vs small gen} is immediate from 
$$\Omega^{\on{DK}}_{q,\Quant}\overset{\text{\thmref{t:quant and abs}}}
\simeq \Omega_q^{\on{DK}} \text{ and } \Omega^{\on{sml}}_{q,\Quant}\overset{\text{\corref{c:identify Omega small}}}\simeq \Omega_q^{\on{sml}},$$
combined with \corref{c:quantum non-coroot fact}.


\sssec{}

Consider the objects
$$(\Delta_\lambda)^!(\Omega^{\on{DK}}_{q,\Quant}) \text{ and } (\Delta_\lambda)^*(\Omega^{\on{DK}}_{q,\Quant})$$
in $\Shv_{\CG}(\BA^1)$. They are equivariant with respect to the action of $\BA^1$ on itself
by translations. So it is enough to prove the corresponding assertions for the !- and *- fibers,
respectively, of these objects at the point 
$$\lambda\cdot 0\overset{\iota_\lambda}\hookrightarrow X^\lambda.$$

\medskip

We will show:

\medskip

\noindent{(a)} If $q$ satisfies (*) then 
$$(\iota_\lambda)^!(\Omega^{\on{DK}}_{q,\Quant})\simeq
\begin{cases}
&\sfe[-\ell(w)], \quad \lambda=w(\rho)-\rho\\
&0, \quad \text{otherwise}.
\end{cases}
$$

\medskip
 
\noindent{(b)} If $\lambda=w(\rho)-\rho$ with $\ell(w)=2$, then $(\iota_\lambda)^*(\Omega^{\on{DK}})$ lives in cohomological degrees $\leq-2$, i.e., 
$H^{-1}((\iota_\lambda)^*(\Omega^{\on{DK}}))=0$.  

\medskip
 
\noindent{(b')} If $q$ satisfies a sharp inequality (*) and $\lambda=w(\rho)-\rho$ with $\ell(w)\geq 3$, then 
$(\iota_\lambda)^*(\Omega^{\on{DK}})$ lives in cohomological degrees $\leq -2$, i.e., 
$H^{-1}((\iota_\lambda)^*(\Omega^{\on{DK}}))=0$.  

\sssec{}

We note that for $\Omega^{\on{DK}}_q$ replaced by $\Omega^{\on{DK}}_{q,\Quant}$, the assertion of
\thmref{t:prop DK char 0 prim} is equivalent to property (b), the assertions of 
\thmref{t:prop DK char 0 sec} and \thmref{t:prop DK char 0 non-deg} combined are equivalent to 
property (a) and the assertion of \thmref{t:prop DK char 0 very non-deg} is equivalent to property (b'). 

\medskip

The rest of this subsection is devoted to the proof of properties (a), (b) and (b') above. 

\sssec{}

By \eqref{e:!-fibers of Omega}, we have
$$(\iota_\lambda)^!(\Omega^{\on{DK}}_{q,\Quant})=\on{C}^\cdot(U^{\on{DK}}_q(\cN))^\lambda.$$

We now recall the following result of \cite[Theorem 6.4.1]{Geo}, valid for $q$ satisfying (*):
\begin{equation} \label{e:Geo}
\on{C}^\cdot(U^{\on{DK}}_q(\cN))\simeq \underset{w\in W}\oplus \sfe^{w(\rho)-\rho}[-\ell(w)],
\end{equation} 
see Remark \ref{r:Geo}. This proves (a). 

\sssec{}  \label{sss:proof KD is KD}

We will now make a digression and show that if $q$ satisfies (*), then the map \eqref{e:mod Serre} is an isomorphism.

\medskip

Let $\CI_q=\CI_{q_\hbart}/\hbar$ be the kernel of the map $U^{\on{free}}_q(\cN)\to U^{\on{DK}}_q(\cN)$. 
By the spectral sequence, 
$$H^2(U^{\on{DK}}_q(\cN))\simeq \Hom_{U^{\on{free}}_q(\cN)\on{-bimod}}(\CI_q,\sfe),$$
where we note that the right-hand in the above formula is the dual space of the (two-sided) quotient $\ol\CI_q$ of $\CI_q$ by the augmentation ideal of $U^{\on{free}}_q(\cN)$. 

\medskip

Combined with \eqref{e:Geo}, we obtain that if $q$ satisfies (*), then $\ol\CI_q$ is spanned by 
elements with coweights 
$$\rho-w(\rho), \quad \ell(w)=2,$$
i.e.,
\begin{equation} \label{e:Serre weight}
\lambda_{i,j}:=-(s_i\cdot s_j(\rho)-\rho)=\alpha_j+(1-\langle \alpha_j,\check\alpha_i\rangle)\cdot \alpha_i,\quad i,j\in I
\end{equation} 
and each such coweight space is one-dimensional. 

\medskip

Recall the Serre relation elements $S^{i,j}_q\in (\CI_q)^{\lambda_{i,j}}$, see \secref{sss:Serre}. It is easy to see that the image $\ol{S}^{i,j}_q$
of $S^{i,j}_q$ under
$$\CI_q\twoheadrightarrow \ol\CI_q$$ is non-zero. Hence, we obtain that the elements $\ol{S}^{i,j}_q$
span $\ol\CI_q$. 

\medskip

Hence, we obtain that if $q$ satisfies (*), the elements $S^{i,j}_q$ generate $\CI_q$ as a two-sided ideal. 

\sssec{}

Let us prove (b) and (b'). By \eqref{e:duality omega} and \eqref{e:!-fibers of Omega}, the fiber 
$(\iota_\lambda)^*(\Omega_q^{\on{DK}})$ identifies with
$$\on{C}_\cdot(U^{\on{Lus}}_{q^{-1}}(\cN))^\lambda.$$

We need to show that the latter has no cohomology in degree $-1$ in coweights of the form $\rho-w(\rho)$. 

\medskip

The vector space $H_1(U^{\on{Lus}}_{q^{-1}}(\cN))$ is the quotient of the augmentation ideal of 
$U^{\on{Lus}}_{q^{-1}}(\cN)$ by its square, and thus is spanned by the generators of $U^{\on{Lus}}_{q^{-1}}(\cN)$. 

\medskip

There two kinds of generators: ones corresponding to simple coroots, and ones corresponding to all coroots
(the latter are the ``divided powers" generators).
The generator of the first kind corresponding to a simple coroot $\alpha_i$ has coweight $\alpha_i$. The generator 
of the second kind corresponding to a coroot $\alpha$ has coweight
$\on{ord}(q(\alpha))\cdot \alpha$.

\medskip

Let first $\ell(w)=2$, so $\rho-w(\rho)$ is of the form \eqref{e:Serre weight}. In this case, the assertion follows from the fact 
that none of the coweights of the form $\alpha_i$ or $\on{ord}(q(\alpha))\cdot \alpha$ have the form \eqref{e:Serre weight}. 
This proves (b). 

\medskip

Let now $\ell(w)\geq 3$. For the generators of the first kind, we cannot have $\rho-w(\rho)=\alpha_i$. 
For the generators of the second kind, 
assume that 
$$\rho-w(\rho)=\on{ord}(q(\alpha))\cdot \alpha$$
for some coroot $\alpha$.

\medskip

Evaluating $q_\BZ(-)+b_\BZ(-,\rho)$ on both sides, we obtain
$$\langle \rho,\check\alpha \rangle=\on{ord}(q(\alpha)),$$
which is impossible if (*) is sharp. This proves (b'). 

\section{The quantum Frobenius}  \label{s:Frobenius}

In this section we will be assuming that $q$ is torsion-valued, but that it
avoids small torsion (see \secref{sss:small root} for what this means)\footnote{The 
latter condition is needed to ensure a relationship between $U_q^{\on{Lus}}(\cN)$
and $u_q(\cN)$ via the quantum Frobenius, i.e., that \eqref{e:SES quant} is a short exact 
sequence of Hopf algebras.}.

\medskip

We will study the manifestation of Lusztig's quantum Frobenius via the factorization algebras
$\Omega^{\on{Lus}}_q$ and $\Omega^{\on{sml}}_q$.

\ssec{The quantum Frobenius lattice}

\sssec{}  \label{sss:sharp}

Let $\Lambda^\sharp\subset \Lambda$ be the sublattice generated by the elements
$$\alpha^\sharp_i=\on{ord}(q(\alpha_i))\cdot \alpha_i.$$

Due to the condition that $q\in \on{Quad}(\Lambda,\fZ)^W_{\on{restr}}$, we have
\begin{equation} \label{e:sharp orth}
b(\gamma,\lambda)=0 \text{ for all } \gamma\in \Lambda^\sharp,\lambda\in \Lambda.
\end{equation} 

Moreover, the sublattice $\Lambda^\sharp$ is $W$-invariant, and hence contains all the elements
$$\alpha^\sharp:=\on{ord}(q(\alpha))\cdot \alpha.$$

\medskip

Finally, it is known that $(\Lambda^\sharp,\{\alpha^\sharp\})$ is the root system of a simple group of adjoint type, to be
denoted $G^\sharp$ (over the field of coefficients $\sfe$), with $\alpha^\sharp_i$ being the simple roots (see \cite[Sect. 2.2.4]{Lus}). 

\sssec{}

Let $\Conf^\sharp$ be the configuration space corresponding to the lattice $\Lambda^\sharp$. 
Let $\Frob_{q,\Conf}$ denote the tautological closed embedding
$$\Conf^\sharp\to \Conf.$$

\medskip

The functoriality of the construction
$$q\mapsto \CG^T_q\mapsto \CG^\Conf_q$$
with respect to the lattice implies that the pullback of the gerbe $\CG^\Conf_q$
along $\Frob_{q,\Conf}$ admits a canonical trivialization (as a factorization gerbe). 

\sssec{}

Recall that $\on{add}$ denotes the addition operation on divisors
$$\on{add}:\Conf\times \Conf\to \Conf;$$
it makes $\Conf$ into a commutative semi-group.

\medskip

The operation 
$$\CF_1, \CF_2\in \Shv(\Conf)\, \mapsto\, \CF_1\star \CF_2:= \on{add}_!(\CF_1\boxtimes \CF_2)\in \Shv(\Conf)$$
makes $\Shv(\Conf)$ into a (symmetric) monoidal category. We will refer to this (symmetric) monoidal structure as \emph{convolution};
this is in order to distinguish it from the pointwise $\sotimes$ symmetric monoidal structure. 

\medskip

In practice we will apply this to the lattice $\Lambda^\sharp$, rather than $\Lambda$. 

\sssec{}

Let $\on{act}$ denote the addition map
$$\on{act}:\Conf^\sharp\times \Conf\to \Conf;$$
it makes $\Conf$ into a $\Conf^\sharp$-module.

\medskip

Note that due to factorization and the trivialization of $\CG^\Conf_q|_{\Conf^\sharp}$, we have a canonical
isomorphism
\begin{equation} \label{e:gerbe equiv disj}
\on{act}^*(\CG^\Conf_q)|_{(\Conf^\sharp\times \Conf)_{\on{disj}}}\simeq \on{pr}^*(\CG^\Conf_q)|_{(\Conf^\sharp\times \Conf)_{\on{disj}}},
\end{equation} 
where $\on{pr}$ is the projection on the $\Conf$ factor. 

\medskip

However, due to \eqref{e:sharp orth}, the identification \eqref{e:gerbe equiv disj} extends (automatically, uniquely) to an identification
\begin{equation} \label{e:gerbe equiv}
\on{act}^*(\CG^\Conf_q) \simeq \on{pr}^*(\CG^\Conf_q)
\end{equation} 
over all of $\Conf^\sharp\times \Conf$, see \eqref{e:on sq}. 
This identification is (automatically) compatible with the semi-group structure on $\Conf^\sharp$.

\medskip

Thus, we obtain that $\CG^\Conf_q$ is equivariant with respect to the action of $\Conf^\sharp$ on $\Conf$, in a way compatible
with factorization. 

\medskip

We obtain that the operation
$$\CF^\sharp\in \Shv(\Conf^\sharp),\,\CF\in \Shv_{\CG^\Conf_q}(\Conf)\, \mapsto\, \CF^\sharp\star \CF:=
\on{act}_!(\CF^\sharp\boxtimes \CF)\in  \Shv_{\CG^\Conf_q}(\Conf)$$
defines an action of the monoidal category $\Shv(\Conf^\sharp)$ on $\Shv_{\CG^\Conf_q}(\Conf)$. We will refer to this action also as 
\emph{convolution}. 

\ssec{Convolution factorization algebras}

\sssec{}  \label{sss:com fact}

Let $\CA$ be a factorization algebra in $\Shv(\Conf)$. The structure of 
(commutative) semi-group on $\Conf$ allows us to talk about an associative (resp., commutative) algebra structure on $\CA$:

\medskip

By definition, such a structure consists of a map
\begin{equation} \label{e:assoc fact}
\CA\star \CA\to \CA,
\end{equation} 
compatible with factorization and the associativity (resp., and the commutativity) law on $\Conf$. In particular, such an $\CA$ 
is an associative (resp., commutative) algebra object in $\Shv(\Conf)$ with respect to convolution.  

\begin{rem} \label{r:add unit}

Note that if $\CA_1$ and $\CA_2$ are factorization algebras on $\Conf$, then their convolution $\CA_1\star \CA_2$
will \emph{not} be a factorization algebra. Rather,
\begin{equation} \label{e:insert unit to conv}
\CA_1\oplus (\CA_1\star \CA_2) \oplus \CA_2
\end{equation} 
will be a factorization algebra. 

\medskip

This can be rephrased as follows. Let $\Conf^+$ be the disjoint union $\Conf\sqcup \on{pt}$. Then the structure of commutative semi-group on $\Conf$
extends to a structure of commutative monoid on $\Conf^+$ with $\on{pt}$ being the unit. This extends the structure of \emph{non-unital}
symmetric monoidal category on $\Shv(\Conf)$ to a unital one on $\Shv(\Conf^+)$. 

\medskip

Given a factorization algebra $\CA$ on $\Conf$, we extend it to $\CA^+\in \Shv(\Conf^+)$ by setting $\CA^+|_{\on{pt}}=\sfe$. With these conventions,
for a pair of factorization algebras $\CA_1$ and $\CA_2$ on $\Conf$, 
$$\CA_1^+\star\CA_2^+\in \Shv(\Conf^+)$$
\emph{is} a factorization algebra, and its restriction back to $\Conf$ equals \eqref{e:insert unit to conv}. 

\medskip

This procedure mimics the following: let $\CC$ be a unital symmetric monoidal category. For an algebra $\CA\in \CC$,
let $\CA^+:=\CA\oplus \one_\CC$ be the corresponding unital augmented algebra. Then we have
$$\CA_1^+\otimes \CA_2^+\simeq \left(\CA_1\oplus (\CA_1\otimes \CA_2)\oplus \CA_2\right) \oplus \one_\CC.$$

\end{rem} 

\sssec{}

Note that the map \eqref{e:assoc fact} by adjunction corresponds to a map
\begin{equation} \label{e:assoc fact adj}
\CA\boxtimes \CA\to \on{add}^!(\CA).
\end{equation} 

The compatibility of the algebra structure and factorization implies, in particular, that the restriction
of the map \eqref{e:assoc fact adj} to $(\Conf\times \Conf)_{\on{disj}}$ coincides with the
factorization isomorphism
\begin{equation} \label{e:fact isom}
\CA\boxtimes \CA|_{(\Conf\times \Conf)_{\on{disj}}}\simeq  \on{add}^!(\CA)|_{(\Conf\times \Conf)_{\on{disj}}}.
\end{equation} 

\sssec{} \label{sss:!-ass}
Set 
\begin{equation} \label{e:restr to diag}
\CA_X:=\underset{\gamma}\oplus\, (\Delta_\gamma)^!(\CA)\in \Shv(X)^\Lambda.
\end{equation} 

By !-restricting \eqref{e:assoc fact} to the main diagonals, we obtain that a structure on $\CA$ of associative 
(resp., commutative) algebra with respect to $\star$ induces a structure of associative (resp., commutative) algebra
on $\CA_X$, with respect to the symmetric monoidal structure on $\Shv(X)^\Lambda$, 
induced by the $\sotimes$ (symmetric) monoidal structure on $\Shv(X)$ and the addition
operation on $\Lambda$. 

\begin{rem}
The algebra $\CA_X$ is non-unital. The corresponding unital algebra
$$\CA_X^+:=\CA_X\oplus \omega_X$$
can be obtained from $\CA^+$ (see Remark \ref{r:add unit}) also by pullback, where the corresponding map
$$\Delta_0:X \to \Conf^+$$
is
$$X \to \on{pt}\hookrightarrow \Conf^+.$$
\end{rem}

\sssec{} \label{sss:action is a cond}

For $x\in X$, let $\Conf_x\subset \Conf$ be the closed subset equal to the union of the images of the maps
$$\Conf \simeq \on{pt}\times \Conf \overset{\iota_{-\alpha_i}\times \on{id}}\longrightarrow \Conf\times \Conf
\overset{\on{add}}\to \Conf.$$

I.e., $\Conf_x$ is the closed subset consisting of those divisors that have a non-trivial contribution at $x$. 

\medskip

Let $\CF_1$ and $\CF_2$ be a pair of perverse sheaves on $\Conf$ that do not have sub-objects supported at 
subsets of the form $\Conf_x$ (for the same $x$ appearing on both sides). In this case the map 
$$(j_{\on{disj}})_!\circ (j_{\on{disj}})^*(\CF_1\boxtimes \CF_2)\to \CF_1\boxtimes \CF_2 $$
is \emph{surjective} as a map of perverse sheaves on $\Conf\times \Conf$. (Here $j_{\on{disj}}$ denotes
the open embedding $(\Conf\times \Conf)_{\on{disj}}\hookrightarrow \Conf\times \Conf$.)

\medskip

Let $\CA$ be a factorization algebra in $\Perv(\Conf)$, such that $\CA_X$ does not have quotient objects supported at closed points
of $X$. This implies that $\CA$ does not have quotient objects supported on subsets of the form $\Conf_x$. 

\medskip

We obtain that in this case, \emph{if} the factorization isomorphism \eqref{e:fact isom}, viewed as a map
$$\CA\boxtimes \CA|_{(\Conf\times \Conf)_{\on{disj}}}\to \on{add}^!(\CA)|_{(\Conf\times \Conf)_{\on{disj}}},$$
extends to all of $\Conf\times \Conf$,
\emph{then} it does so uniquely.

\medskip

Thus, we obtain that under the above condition, being an associative algebra within 
factorization algebras is \emph{a condition and not an additional structure}. Furthermore, in this case, the associative
algebra structure is automatically commutative. 

\sssec{}  \label{sss:act}

Given a factorization algebra $\CA^\sharp\in \Shv(\Conf^\sharp)$, equipped with an associative algebra structure,
and a factorization algebra $\CA\in \Shv_{\CG^\Conf_q}(\Conf)$, one can talk about an action of $\CA^\sharp$ on $\CA$:

\medskip

By an action we will mean the datum of a map of factorization algebras
\begin{equation} \label{e:mod embed first}
(\Frob_{q,\Conf})_!(\CA^\sharp)\to \CA
\end{equation} 
and a map
\begin{equation} \label{e:assoc fact act}
\CA^\sharp\star \CA\to \CA,
\end{equation} 
that make the following diagram commute:
\begin{equation} \label{e:action compat}
\CD
\CA^\sharp \star  (\Frob_{q,\Conf})_!(\CA^\sharp) @>{\sim}>>  (\Frob_{q,\Conf})_!(\CA^\sharp\star \CA^\sharp) @>>> (\Frob_{q,\Conf})_!(\CA^\sharp)  \\
@VVV & &   @VVV  \\
\CA^\sharp\star \CA & @>>> & \CA,
\endCD
\end{equation} 
in a way compatible with factorization and the associativity law. 

\begin{rem}
The fact that we have to specify the map \eqref{e:mod embed first} in addition to \eqref{e:assoc fact act} is the reflection
of the non-unital nature of our symmetric monoidal categories, see Remark \ref{r:add unit}.

\medskip

These two pieces of data can be combined into one:
$$(\CA^\sharp)^+\star \CA^+\to \CA ^+,$$
subject to the associativity law. 

\end{rem}


\sssec{}

%
The map \eqref{e:assoc fact act} can be expressed by adjunction as a map 
\begin{equation} \label{e:assoc fact act adj adj}
(\Frob_{q,\Conf})_!(\CA^\sharp)\boxtimes \CA \to \on{add}^!(\CA)
\end{equation} 
on $\Conf\times \Conf$.

\medskip

Again, the compatibility with factorization implies that the restriction of the restriction of the map 
\eqref{e:assoc fact act adj adj} to $(\Conf\times \Conf)_{\on{disj}}$ indentifies with
\begin{equation} \label{e:fact isom mod}
(\Frob_{q,\Conf})_!(\CA^\sharp)\boxtimes \CA |_{(\Conf\times \Conf)_{\on{disj}}}\overset{\text{\eqref{e:mod embed first}}\boxtimes \on{id}}\longrightarrow
\CA \boxtimes \CA |_{(\Conf\times \Conf)_{\on{disj}}}\simeq \on{add}^!(\CA)|_{(\Conf\times \Conf)_{\on{disj}}}.
\end{equation} 

\sssec{}

Let us be given an action of $\CA^\sharp$ on $\CA$, and consider $\CA^\sharp_X$ and $\CA_X$
(defined as in \eqref{e:restr to diag}).

\medskip

We obtain a map 
$$\CA^\sharp_X\to \CA_X$$
and an action of $\CA^\sharp_X$, viewed as an algebra object in the (symmetric) monoidal category
$\Shv(X)^{\Lambda^\sharp}$, 
on $\CA_X$ (that are compatible via an analog of the diagram \eqref{e:action compat}), where we view  
\begin{equation} \label{e:categ on X}
\Shv_{\CG^\lambda_q}(X)^\Lambda:=\underset{\lambda}\oplus\, \Shv_{\CG^\lambda_q}(X)
\end{equation} 
as an $\Shv(X)^{\Lambda^\sharp}$-module category.

\begin{rem}

This data can be equivalently expressed as an action of $(\CA^\sharp_X)^+$ on $\CA_X^+$. 

\end{rem}

\sssec{}

Finally, let assume that both $\CA^\sharp$ and $\CA$ are perverse and such that $\CA^\sharp_X$ and $\CA_X$ 
do not have quotient objects supported at closed points of $x$. 

\medskip

Then as in \secref{sss:action is a cond}, we obtain that \emph{given a datum of
\eqref{e:mod embed first}}, an action of $\CA^\sharp$ on $\CA$ is not an additional structure, but rather a condition. Namely, this condition
says that the map \eqref{e:fact isom mod} extends to all of $\Conf\times \Conf$ (if it does, this extension is unique).

%
%
%
%
%
%
%

\ssec{The classical factorization algebra $\Omega^{\sharp,\on{cl}}$}

In this subsection we will construct a certain factorization algebra, denoted $\Omega^{\on{cl}}$ in $\Perv(\Conf)$.
In practice, we will apply this construction to the lattice $\Lambda^\sharp$ rather than $\Lambda$; in this case, the resulting
factorization algebra will be denoted by $\Omega^{\sharp,\on{cl}}$.

\sssec{}

Recall (following, e.g., \cite[Sect. 2.3]{Ga2}) that to a commutative algebra object $A$ in the (non-unital) symmetric monoidal category
$\Vect^{\Lambda^{\on{neg}}-0}$, one canonically attaches a factorization algebra $\on{Fact}(A)$ on $\Conf$, equipped with a structure
of commutative algebra in the sense of \secref{sss:com fact}.

\medskip

We have:
\begin{equation} \label{e:!-fibers of Fact}
\on{Fact}(A)_X\simeq \omega_X\otimes A,
\end{equation}
as commutative algebras in $\Shv(X)^\Lambda$, see \secref{sss:!-ass} for the notation.

\medskip

In fact, the assignment
\begin{equation} \label{e:com Fact}
A \mapsto \on{Fact}(A)
\end{equation}
is the composition of the pullback functor
$$\on{ComAlg}(\Vect^{\Lambda^{\on{neg}}-0})\to \on{ComAlg}(\Shv(X)^{\Lambda^{\on{neg}}-0}),\quad A\mapsto \omega_X\otimes A$$
followed by an equivalence of categories 
\begin{equation} \label{e:com Fact gen}
\on{ComAlg}(\Shv(X)^{\Lambda^{\on{neg}}-0})\to \on{ComAlg}(\on{FactAlg}(\Shv(\Conf))),
\end{equation}
where the functor inverse to \eqref{e:com Fact gen} is given by $\CA\mapsto \CA_X$. 

\begin{rem}

When the sheaf theory we are working with is that of sheaves in the classical topology, 
the functor \eqref{e:com Fact} is a special case of the functor \eqref{e:from Hopf to fact}
for $q=1$.

\medskip

Namely, the two procedures are related via the Koszul duality equivalence between the category of
\emph{co-commutative} Hopf algebras as in \secref{sss:conditions on Hopf} and 
$\on{ComAlg}((\Vect^{\Lambda^{\on{neg}}-0})^{\on{loc.fin.dim.}})$: 

\medskip

For a Hopf algebra $H$, the associative algebra
$$\on{C}^\cdot(H)$$
has a natural structure of $\BE_2$-algebra, and if $H$ was cocommutative, this $\BE_2$-algebra structure 
naturally upgrades to an $\BE_\infty$-structure. 

\end{rem}

\sssec{}

We will apply the above construction to (the augmentation ideal of)
$$A:=\on{C}^\cdot(\cn).$$

Denote the resulting factorization algebra by $\Omega^{\on{cl}}$. 

\begin{rem}

As was mentioned above,  in practice we will apply this construction
to $\Lambda^\sharp$ rather than $\Lambda$, and consider the factorization algebra
$$\Omega^{\sharp,\on{cl}}:=\on{Fact}(\on{C}^\cdot(\fn^\sharp))\in \Shv(\Conf^\sharp),$$
where $\fn^\sharp$ is the maximal unipotent in the group $G^\sharp$ from \secref{sss:sharp}.

\end{rem}

\sssec{}

Since $(A)^{-\alpha_i}=\sfe[-1]$, from \eqref{e:!-fibers of Fact} we obtain that
$$\Omega^{\on{cl}}|_{\Conf^{-\alpha_i}}=\sfe_X[1].$$

Hence, we obtain a canonical identification of factorization algebras over $\oConf$:
$$j^*(\Omega^{\on{cl}})\simeq \oOmega.$$

\medskip

Consider the resulting map
\begin{equation} \label{e:omega cl to open}
\Omega^{\on{cl}}\to j_*(\oOmega).
\end{equation}

The following is established in \cite[Sect. 3 and Lemma 4.8]{BG2}:

\begin{thm} \label{t:BG1}
The factorization algebra $\Omega^{\on{cl}}$ is perverse, and the map \eqref{e:omega cl to open} is injective.
\end{thm}

In fact, the proof in \secref{ss:proof of prop} applies to $\Omega^{\on{cl}}$, giving rise to  the following explicit 
inductive description of $\Omega^{\on{cl}}$ as a sub-object of $j_*(\oOmega)$:

\begin{thm} \label{t:BG2}
For $\lambda\in \Lambda^{\on{neg}}$, the maps 
$$(\jmath_\lambda)_{!*}\circ (\jmath_\lambda)^*(\Omega^{\on{cl}})\to 
\Omega^{\on{cl}}|_{\Conf^\lambda}\to H^0\left((\jmath_\lambda)_*\circ (\jmath_\lambda)^*(\Omega^{\on{cl}})\right)\to
(\jmath_\lambda)_*\circ (\jmath_\lambda)^*(\Omega^{\on{cl}})$$
have the following proprties:

\smallskip

\noindent{\em(a)}
For $\lambda=w(\rho)-\rho$ with $\ell(w)\geq 2$, the map $(\jmath_\lambda)_{!*}\circ (\jmath_\lambda)^*(\Omega^{\on{cl}})\to 
\Omega^{\on{cl}}|_{\Conf^\lambda}$ is an isomorphism.

\medskip

\noindent{\em(b)} 
For $\lambda=w(\rho)-\rho$ with $\ell(w)\geq 2$, the cone of the map
$$\Omega^{\on{cl}}|_{\Conf^\lambda}\to (\jmath_\lambda)_*\circ (\jmath_\lambda)^*(\Omega^{\on{cl}})$$
identifies with $(\Delta_\lambda)_*(\sfe_X[-\ell(w)+1])$. In particular:

\smallskip

\noindent{\em(i)} If $\ell(w)=2$, the object $(\jmath_\lambda)_*\circ (\jmath_\lambda)^*(\Omega^{\on{cl}})$
is perverse; 

\smallskip

\noindent{\em(ii)} If $\ell(w)\geq 3$, the map
$$\Omega^{\on{cl}}|_{\Conf^\lambda}\to H^0\left((\jmath_\lambda)_*\circ (\jmath_\lambda)^*(\Omega^{\on{cl}})\right)$$
is an isomorphism. 

\smallskip

\noindent{\em(c)}
For $\lambda$ \emph{not} of the form $w(\rho)-\rho$, the maps 
$$\Omega^{\on{cl}}|_{\Conf^\lambda}\to H^0\left((\jmath_\lambda)_*\circ (\jmath_\lambda)^*(\Omega^{\on{cl}})\right)\to
(\jmath_\lambda)_*\circ (\jmath_\lambda)^*(\Omega^{\on{cl}})$$
are isomorphisms.

\end{thm}

\begin{cor}
The map 
$$\Omega^{\on{cl}}|_{\Conf^\lambda}\to H^0\left((\jmath_\lambda)_*\circ (\jmath_\lambda)^*(\Omega^{\on{cl}})\right)$$
is an isomorphism for all $\lambda$ \emph{not} of the form $w(\rho)-\rho$ with $\ell(w)=2$. 
\end{cor} 

\begin{rem}
Note that points (a) and (b) of \thmref{t:BG2} for $\lambda=w(\rho)-\rho$ with $\ell(w)\geq 3$ mean that 
in this case, the maps $$(\jmath_\lambda)_{!*}\circ (\jmath_\lambda)^*(\Omega^{\on{cl}})\to 
\Omega^{\on{cl}}|_{\Conf^\lambda}\to H^0\left((\jmath_\lambda)_*\circ (\jmath_\lambda)^*(\Omega^{\on{cl}})\right)$$
are both isomorphisms.
\end{rem}

%
%
%
%

\begin{rem}
One can use the proof of \thmref{t:quant and abs} to
show that the factorization algebra $\Omega^{\on{cl}}$ is in fact a special case of $\Omega^{\on{DK}}_q$
(constructed by deforming the parameter in \secref{ss:constr by deformation}) for $q=1$. 
\end{rem}

\ssec{Relationship of $\Omega_q^{\on{Lus}}$ with $\Omega_q^{\on{sml}}$ via
$\Omega^{\on{cl}}$}

\sssec{}

We will now state a key theorem that expresses the relationship of $\Omega_q^{\on{Lus}}$ and $\Omega_q^{\on{sml}}$
via $\Omega_q^{\sharp,\on{cl}}$. (We emphasize that for the validity of this theorem, the assumption that $q$ avoid small torsion is important.)

\begin{thm} \label{t:Lus and sml}  \hfill

\smallskip

\noindent{\em(a)} For every vertex of the Dynkin diagram, we have
$$H^0\left((\Delta_{-\alpha_i^\sharp})^!(\Omega_q^{\on{Lus}})\right)\simeq \sfe_X[1].$$
In particular, we have a have an isomorphism of factorization algebras in $\Perv(\oConf^\sharp)$:
\begin{equation} \label{e:Frob restr open}
(j^\sharp)^*\left(H^0\left((\Frob_{q,\Conf})^!(\Omega_q^{\on{Lus}})\right)\right)\simeq \oOmega^\sharp.
\end{equation}

\smallskip

\noindent{\em(b)} The map 
\begin{equation} \label{e:Frob restr open *}
H^0\left((\Frob_{q,\Conf})^!(\Omega_q^{\on{Lus}})\right) \to (j^\sharp)_*(\oOmega^\sharp)
\end{equation}
is injective. 

\smallskip

\noindent{\em(c)} The image of the map \eqref{e:Frob restr open *} equals that of $\Omega^{\sharp,\on{cl}}$. 
In particular, we have an isomorphism of factorization algebras 
\begin{equation} \label{e:Frob restr is cl}
\Omega^{\sharp,\on{cl}}\simeq H^0\left((\Frob_{q,\Conf})^!(\Omega_q^{\on{Lus}})\right).
\end{equation}

\smallskip

\noindent{\em(d)} 
The map 
$$(\Frob_{q,\Conf})_!(\Omega^{\sharp,\on{cl}})\to \Omega_q^{\on{Lus}},$$
resulting from \eqref{e:Frob restr is cl} extends to an action of 
$\Omega^{\sharp,\on{cl}}$ on $\Omega_q^{\on{Lus}}$.

\smallskip

\noindent{\em(e)} 
The projection $\Omega_q^{\on{Lus}}\to \Omega_q^{\on{sml}}$ is compatible with 
$\Omega^{\sharp,\on{cl}}$-actions, where the action on $\Omega_q^{\on{sml}}$ is the trivial one.

\smallskip

\noindent{\em(f)}
The resulting map
\begin{equation} \label{e:coinv abs}
\left(\on{Cone}((\Omega^{\sharp,\on{cl}})_X\to (\Omega_q^{\on{Lus}})_X)\right)\underset{(\Omega^{\sharp,\on{cl}})_X}\otimes \omega_X\to 
(\Omega_q^{\on{sml}})_X
\end{equation}
is an isomorphism.

\end{thm}

\begin{rem}  \label{r:aug}
The reason that in point (f) we have the somewhat bizarre looking 
\begin{equation} \label{e:bizarre}
\left(\on{Cone}((\Omega^{\sharp,\on{cl}})_X\to (\Omega_q^{\on{Lus}})_X)\right)\underset{(\Omega^{\sharp,\on{cl}})_X}\otimes \omega_X
\end{equation}
instead of just 
$$\left((\Omega^{\sharp,\on{Lus}})_X\right)\underset{(\Omega^{\sharp,\on{cl}})_X}\otimes \omega_X$$
is that we are working in the non-unital setting. E.g., our operation of coinvariants with the cone is designed so that it would
send $\Omega^{\sharp,\on{cl}}$ itself to $0$. 

\medskip

We could equivalently rewrite \eqref{e:coinv abs} as an isomorphism
\begin{equation} \label{e:less bizarre}
(\Omega_q^{\on{Lus}})_X^+\underset{(\Omega^{\sharp,\on{cl}}_X)^+}\otimes \omega_X\simeq (\Omega_q^{\on{sml}})_X^+.
\end{equation}

\end{rem} 

\begin{rem}  \label{r:bar}
Given an action of $\CA^\sharp\in \Shv(\Conf^\sharp)$ on $\CA\in \Shv_{\CG^\Conf_q}(\Conf)$, we can form the object
$$\on{Bar}(\CA^\sharp,\CA):=\on{Bar}((\CA^\sharp)^+,\CA^+) \in \Shv_{\CG^\Conf_q}(\Conf),$$
which also carries a structure of factorization algebra.

\medskip

We have
$$\on{Bar}(\CA^\sharp,\CA)_X\simeq \on{Bar}(\CA^\sharp_X,\CA_X):=\on{Bar}((\CA^\sharp_X)^+,\CA^+_X)\simeq 
\left(\on{Cone}(\CA^\sharp_X\to \CA_X)\right)\underset{\CA^\sharp_X}\otimes \omega_X.$$

\medskip

So, by factorization, \thmref{t:Lus and sml} is equivalent to the statement that the map
$$\Omega_q^{\on{Lus}}\to \Omega_q^{\on{sml}}$$
induces an isomorphism
$$\on{Bar}(\Omega^{\sharp,\on{cl}},\Omega_q^{\on{Lus}})\simeq \Omega_q^{\on{sml}}.$$
\end{rem}

%

\sssec{}

\thmref{t:Lus and sml} will be proved in the case of a ground field of characteristic $0$ in the rest of this section. The general case will be treated
in \secref{sss:proof of Lus and sml via Whit}.

\sssec{}  

In the case of when the ground field has characteristic $0$, as in \secref{ss:proof of prop}, we reduce the assertion 
to the case when $X=\BA^1$
and the sheaf theory being that of constructible sheaves in the classical topology. Consider the corresponding factorization algebra 
$\Omega^{\on{Lus}}_{q,\Quant}$, which we already know is isomorphic to $\Omega^{\on{Lus}}_q$.

\sssec{}  \label{sss:Frob quant}

In \secref{ss:Frob for fact} we will show that there exists a canonically defined action of $\Omega^{\sharp,\on{cl}}$ on
$\Omega^{\on{Lus}}_{q,\Quant}$ such that the map
$$\Omega_{q,\Quant}^{\on{Lus}}\to \Omega_{q,\Quant}^{\on{sml}}$$
is compatible with the trivial $\Omega^{\sharp,\on{cl}}$-action on $\Omega_{q,\Quant}^{\on{sml}}$
and induces an identification
\begin{equation} \label{e:coinv Quant}
\left((\on{Cone}(\Omega^{\sharp,\on{cl}})_X\to (\Omega_{q,\Quant}^{\on{Lus}})_X)\right)
\underset{(\Omega^{\sharp,\on{cl}})_X}\otimes \omega_X\to (\Omega_{q,\Quant}^{\on{sml}})_X
\end{equation}

Let us show how the existence of this action and the 
isomorphism \eqref{e:coinv Quant} implies the assertion of \thmref{t:Lus and sml}. 

\begin{rem}  \label{r:Frob quant}
We emphasize that for the above structure to exist on $\Omega^{\on{Lus}}_{q,\Quant}\simeq \Omega^{\on{Lus}}_q$,
the assumption that $q$ should avoid small torsion is important. This is due to the fact that the construction of this structure 
uses (and is essentially equivalent to) the quantum Frobenius for $U_q^{\on{Lus}}(\cN)$. 
\end{rem} 

\sssec{}  \label{sss:Frob prop 1}

We only need to show that the map
$$(\Frob_{q,\Conf})_!(\Omega^{\sharp,\on{cl}})\to \Omega_{q,\Quant}^{\on{Lus}}$$
induces an isomorphism
$$\Omega^{\sharp,\on{cl}}\to H^0\left((\Frob_{q,\Conf})^!(\Omega_{q,\Quant}^{\on{Lus}})\right).$$

By induction and factorization, this is equivalent to showing that for any $\gamma\in \Lambda^{\sharp,\on{neg}}-0$,
$$(\Delta_\gamma)^!\left(\on{Cone}((\Frob_{q,\Conf})_!(\Omega^{\sharp,\on{cl}})\to \Omega_{q,\Quant}^{\on{Lus}})\right)$$
lives in cohomological degrees $\geq 1$. 

\medskip

By translation invariance, it suffices to show that
$$(\iota_\gamma)^!\left(\on{Cone}((\Frob_{q,\Conf})_!(\Omega^{\sharp,\on{cl}})\to \Omega_{q,\Quant}^{\on{Lus}})\right)$$
lives in cohomological degrees $\geq 2$. 

\sssec{}

Let us denote by $M$ the $(\Lambda^{\on{neg}}-0)$-graded vector space 
$$\underset{\lambda}\oplus\, (\iota_\lambda)^!\left(\on{Cone}((\Frob_{q,\Conf})_!(\Omega^{\sharp,\on{cl}})\to \Omega_{q,\Quant}^{\on{Lus}})\right).$$

It is acted on by the $(\Lambda^{\sharp,\on{neg}}-0)$-graded algebra
$$\underset{\gamma}\oplus\, (\iota_\gamma)^!(\Omega^{\sharp,\on{cl}})\simeq \on{C}^\cdot(\fn^\sharp).$$

Isomorphism \eqref{e:coinv Quant} implies that
\begin{equation} \label{e:fiber of Tor}
M\underset{\on{C}^\cdot(\fn^\sharp)}\otimes \sfe\simeq \underset{\lambda}\oplus\, (\iota_\lambda)^!(\Omega^{\on{sml}}_{q,\on{Quant}})=:M'.
\end{equation} 

We need to show that $H^0(M)$ and $H^1(M)$ do not have non-zero components of degrees that lie in $\Lambda^\sharp$. 

\sssec{}  \label{sss:Frob prop 3}

The datum of action of $\on{C}^\cdot(\fn^\sharp)$ on $M$ and the isomorphism \eqref{e:fiber of Tor} is equivalent to that
of action of the algebraic group $N^\sharp$ on $M'$ and an isomorphism
\begin{equation} \label{e:N sharp action}
\on{C}^\cdot(N^\sharp,M')\simeq M.
\end{equation} 

We know that $M'$ is concentrated in cohomological degrees $\geq 1$, and 
$$H^1(M')\simeq \underset{i}\oplus\, \sfe^{-\alpha_i}.$$

Hence, by the spectral sequence,
$$H^0(M)=0$$
and
$$H^1(M)\simeq H^1(M')\simeq \underset{i}\oplus\, \sfe^{-\alpha_i}.$$

This proves the desired assertion since none of the elements $-\alpha_i$ lie in $\Lambda^\sharp$, due to the non-degeneracy assumption 
on $q$. 

\ssec{The quantum Frobenius map for quantum groups}

\sssec{}

Let $T^\sharp$ denote the torus, whose lattice of characters is $\Lambda^\sharp$; we can identify $T^\sharp$
with the Cartan subgroup of $G^\sharp$, see \secref{sss:sharp}. 

\medskip

The embedding $\Lambda^\sharp\to \Lambda$ gives rise to a map
$$\Rep(T^\sharp) \to Z_{\BE_3}(\Rep_q(\cT)).$$

At the level of abelian categories (which is all we need for the moment), this means that we have a braided monoidal
functor 
$$\Frob^*_{q,T}:(\Rep(T^\sharp))^\heartsuit\to (\Rep_q(\cT))^\heartsuit,$$
such for every $V\in \Rep(T^\sharp)$ and $\CM\in \Rep_q(\cT)$, the square
of the braiding
$$\Frob^*_{q,T}(V)\otimes W\to W\otimes \Frob^*_{q,T}(V)\to \Frob^*_{q,T}(V)\otimes W$$
is the identity map.

\sssec{}

Let $\CC^0$ be an $\BE_3$-category (i.e., an algebra object in the category of unital braided monoidal categories). 
Then the operation of tensor product makes $\on{HopfAlg}(\CC^0)$ into a monoidal category. If $\CC^0$
is a symmetric monoidal category, then $\on{HopfAlg}(\CC^0)$ acquires a symmetric monoidal structure.

\medskip

In particular, it makes sense to talk about coassociative (resp., cocommutative) coalgebra objects in $\on{HopfAlg}(\CC^0)$. 

\medskip

Note that in the commutative case, a structure on a Hopf algebra 
$H^0$ of cocommutative coalgebra object in $\on{HopfAlg}(\CC^0)$ is equivalent to the structure on $H^0$ 
of \emph{cocommutative Hopf algebra}.

\sssec{}

Let $\CC$ be a braided monoidal category; let $\CC^0$ be an $\BE_3$-category and
let us be given a functor
$$\epsilon:\CC^0\to Z_{\BE_3}(\CC).$$

In this case, the operation of tensor product defines an action of $\on{HopfAlg}(\CC^0)$ on $\on{HopfAlg}(\CC)$
$$H^0,H\mapsto \epsilon(H^0)\otimes H.$$

\medskip

In particular, it makes sense to talk about a \emph{coaction} of a coassociative algebra object $H^0\in \on{HopfAlg}(\CC^0)$
on $H\in \on{HopfAlg}(\CC)$. 

\sssec{}

Assume that $\CC^0$ has a t-structure, so that the operation of tensor product is t-exact. Let $H^0$ be a Hopf
algebra in the heart of the t-structure. Then the structure on $H^0$ of coassociative coalgebra object in $\on{HopfAlg}(\CC^0)$
is equivalent to the \emph{condition} that the square
$$
\CD
H^0  @>{\on{comult}}>>  H^0\otimes H^0  \\
@V{\on{id}}VV  @V{R}VV   \\
H^0  @>{\on{comult}}>>  H^0\otimes H^0 
\endCD
$$
commutes. 

\medskip

Let $\CC$ be also equipped with a t-structure for which the tensor product 
and the action of $\CC^0$ on $\CC$ are t-exact functors. 

\medskip

Let $H^0$ be as above. Let $H$ be a Hopf algebra in $\CC$ that also lies in the heart of the t-structure. 

\medskip

Then the datum of coaction of $H^0$ on $H$ is equivalent to that of \emph{co-central} homomorphism
\begin{equation} \label{e:cocentral}
H\to \epsilon(H^0),
\end{equation} 
i.e., this is a homomorphism of Hopf algebras that makes the diagram
$$
\CD
H  @>{\on{comult}}>>  H\otimes H  @>>>  \epsilon(H^0)\otimes H \\
@V{\on{id}}VV & &  @VV{R}V    \\
H @>{\on{comult}}>>  H\otimes H @>>> H\otimes \epsilon(H^0) 
\endCD
$$
commute. 

%

\sssec{}

We now recall that Lusztig's quantum Frobenius is a cocentral map
\begin{equation} \label{e:quant Frob}
\Frob_{q,N}:U_q^{\on{Lus}}(\cN)\to \Frob^*_{q,T}(U(\fn^\sharp)).
\end{equation} 

A basic feature of this map is that the composite map
\begin{equation} \label{e:SES quant}
u_q(\cN) \to U_q^{\on{Lus}}(\cN)\to \Frob^*_{q,T}(U(\fn^\sharp))
\end{equation}
factors through augmentation, i.e., equals
$$u_q(\cN) \overset{\on{counit}}\longrightarrow \sfe \overset{\on{unit}}\longrightarrow \Frob^*_{q,T}(U(\fn^\sharp)).$$

Moreover, the resulting map  
\begin{equation} \label{e:Tor over small}
U_q^{\on{Lus}}(\cN)\underset{u_q(\cN)}\otimes \sfe\to \Frob^*_{q,T}(U(\fn^\sharp))
\end{equation} 
(as plain objects of $\Rep_q(\cT)$), is an isomorphism\footnote{The latter fact follows by considering PBW bases.}. 
This expresses the fact that \eqref{e:SES quant} is a ``short exact sequence"
of Hopf algebras. 

\begin{rem}  \label{r:quant Frob good}
We emphasize that it is for the validity of the above assertions that we made the assumption that
$q$ should avoid small torsion. 
\end{rem}

\sssec{}

The maps in \eqref{e:SES quant} induce maps
\begin{equation} \label{e:SES Omega}
\uHom_{U(\fn^\sharp)}(\sfe,\sfe) \to \uHom_{U^{\on{Lus}}_q(\cN)}(\sfe,\sfe) \to \uHom_{u_q(\cN)}(\sfe,\sfe)
\end{equation}
(as associative algebras in $\Rep_q(\cT)$) 
such that the composite factors through augmentation. 

\medskip

In particular, we obtain a map
\begin{equation} \label{e:ten prod map Omega}
\uHom_{U^{\on{Lus}}_q(\cN)}(\sfe,\sfe) \underset{\uHom_{U(\fn^\sharp)}(\sfe,\sfe)}\otimes \sfe\to \uHom_{u_q(\cN)}(\sfe,\sfe).
\end{equation} 

\begin{prop} \label{p:ten prod map Omega}
The map \eqref{e:ten prod map Omega} is an isomorphism.
\end{prop} 

\sssec{}

In order to prove \propref{p:ten prod map Omega} we will now review the theory of Koszul duality. 

\medskip
 
Let $A$ be an associative algebra in $\Rep_q(\cT)$, with coweights in $\Lambda^{\on{pos}}$, and such that its coweight
$0$ component is $\sfe$ (which automatically gives $A$ an augmentation).

\medskip

Let $A\mod^{\on{loc.nilp}}(\Rep_q(\cT))$ be the category of \emph{locally nilpotent} $A$-modules in $\Rep_q(\cT)$.
By definition, this category is compactly generated by modules of the form $\sfe^\lambda$ with the trivial action.
Set
$$B:=\uHom_A(\sfe,\sfe)^{\on{op}}.$$

Then the functor
$$\CM\mapsto \uHom_A(\sfe,\CM)$$
defines an equivalence
\begin{equation} \label{e:KD}
\on{KD}_A:A\mod^{\on{loc.nilp}}(\Rep_q(\cT))\to B\mod(\Rep_q(\cT)).
\end{equation} 

We have
$$\on{KD}_A(\sfe)\simeq B \text{ and } \on{KD}_A(A^\vee)\simeq \sfe,$$
where $A^\vee$ denotes the ($\Lambda$-graded) dual of $A$, naturally 
considered as an object of $A\mod^{\on{loc.nilp}}(\Rep_q(\cT))$. 

\sssec{}  \label{sss:KD funct}

Let now $$\epsilon_A:A_1\to A_2$$ be a homomorphism of associative algebras. We have the forgetful functor
$$\oblv_{A_2\to A_1}:A_2\mod^{\on{loc.nilp}}(\Rep_q(\cT))\to A_2\mod^{\on{loc.nilp}}(\Rep_q(\cT)),$$
and its right adjoint, denoted $\coind_{A_1\to A_2}$, and given by
$$\CM_1\mapsto \uHom_{A_1}(A_2,\CM_1).$$

\medskip

The homomorphism $\epsilon_A$ gives rise to a homomorphism 
$$\epsilon_B:B_2\to B_1.$$

Under the equivalences \eqref{e:KD} for $A_1$ and $A_2$, the functors $\oblv_{A_2\to A_1}$ and $\coind_{A_1\to A_2}$
correspond to the functors
$$\ind_{B_2\to B_1}:B_2\mod(\Rep_q(\cT)) \rightleftarrows B_1\mod(\Rep_q(\cT)):\oblv_{B_1\to B_2},$$
where $\ind_{B_2\to B_1}$ is the functor
$$\CN_2\mapsto B_1\underset{B_2}\otimes \CN_2.$$

\sssec{}

We are now ready to prove \propref{p:ten prod map Omega}:

\medskip

Let us apply the discussion in \secref{sss:KD funct} to the homomorphism
$$U_q^{\on{Lus}}(\cN)\to \Frob^*_{q,T}(U(\fn^\sharp)).$$

We obtain that under the equivalence
\begin{equation} \label{e:KD q}
\uHom_{U^{\on{Lus}}_q(\cN)}(\sfe,\sfe)\mod \simeq U^{\on{Lus}}_q(\cN)\mod^{\on{loc.nilp}},
\end{equation} 
the left-hand side in \eqref{e:ten prod map Omega} corresponds to 
$$\Frob^*_{q,T}((U(\fn^\sharp))^\vee),$$
viewed as a $U_q^{\on{Lus}}(\cN)$-module via $\Frob_{q,N}$.

\medskip

Consider now the homomorphism
$$u_q(\cN) \to U_q^{\on{Lus}}(\cN).$$
We obtain that the right-hand side in \eqref{e:ten prod map Omega}, viewed as a
$\uHom_{U^{\on{Lus}}_q(\cN)}(\sfe,\sfe)$-module via
$$\uHom_{U^{\on{Lus}}_q(\cN)}(\sfe,\sfe)\to \uHom_{u_q(\cN)}(\sfe,\sfe),$$
corresponds under \eqref{e:KD q} to 
$$\uHom_{u_q(\cN)}(U_q^{\on{Lus}}(\cN),\sfe).$$

Hence, the isomorphism stated in the proposition amounts to the fact that the map
$$\Frob^*_{q,T}((U(\fn^\sharp))^\vee) \to \uHom_{u_q(\cN)}(U_q^{\on{Lus}}(\cN),\sfe)$$
is an isomorphism. 

\medskip

However, this follows from \eqref{e:Tor over small} by dualization. 

\qed[\propref{p:ten prod map Omega}]

\ssec{Quantum Frobenius for factorization algebras}  \label{ss:Frob for fact}

\sssec{}

In this section we will perform the construction of the action of $\Omega^{\sharp,\on{cl}}$ on
$\Omega^{\on{Lus}}_{q,\Quant}$.

\medskip

This is obtained by enhancing the paradigm of \secref{ss:Hopf to Fact}. 

\sssec{}

First, the (symmetric) monoidal structure on $\on{HopfAlg}(\Rep(T^\sharp))$
$$H^0_1,H^0_2\mapsto H^0_1\otimes H^0_2$$
corresponds under the equivalence \eqref{e:from Hopf to fact} to the (symmetric) monoidal structure on
the category of factorization algebras in $\Shv((\Conf^\sharp)^+)$ given by convolution:
$$\CA_{1},\CA_{2} \mapsto \CA_{1} \star \CA_{2},$$
see Remark \ref{r:add unit}.

\medskip

Further, the action of $\on{HopfAlg}(\Rep(T^\sharp))$ on $\on{HopfAlg}(\Rep_q(\cT))$ corresponds
under the equivalence \eqref{e:from Hopf to fact} to the 
action of the category of factorization algebras in $\Shv(\Conf^\sharp)$ on the category of factorization algebras
in $\Shv_{\CG^\Conf_q}(\Conf)$, also given by convolution:
$$\CA^\sharp,\CA\mapsto \CA^\sharp\star \CA.$$

\sssec{}

Hence, the datum of the co-central homomorphism \eqref{e:quant Frob} gives rise to an action of
$\Omega^{\sharp,\on{cl}}$ on $\Omega^{\on{Lus}}_{q,\on{Quant}}$. 

\medskip

Finally, the isomorphism \eqref{e:coinv Quant} follows from the isomorphism of \propref{p:ten prod map Omega}
(see Remark \ref{r:aug}).

\section{Zastava spaces and Whittaker/semi-infinite categories}   \label{s:Zastava}

In this section we review some geometric constructions needed in order to introduce the Whittaker incarnations
of the factorization algebras $\Omega^?_q$.  

\ssec{Affine Grassmannian over the configuration space}

\sssec{}

Recall that we denote by $G$ the semi-simple simply connected group, whose coroot lattice is $\Lambda$.
Let 
$$\Gr^{\omega^\rho}_{G,\Conf}\to \Conf$$
denote the following version of the affine Grassmannian of $G$:

\medskip

Namely, for a test-scheme $Y$, a $Y$-point of $\Gr^{\omega^\rho}_{G,\Conf}$ is a triple $(D,\CP_G,\alpha)$, where:

\begin{itemize}

\item $D$ is a $Y$-point of $\Conf$;

\item $\CP_G$ is a $G$-bundle on $Y\times X$;

\item $\alpha$ is an identification of $\CP_G$ with the $G$-bundle $\CP_G^{\omega^\rho}$, defined on 
$Y\times X-\Gamma_D$. 

\end{itemize} 

Here:

\begin{itemize}

\item $\CP_G^{\omega^\rho}$ is the $G$-bundle induced 
from a (chosen once and for all) square root $\omega^{\otimes \frac{1}{2}}$ of the canonical line bundle $\omega$ via 
$$\BG_m\overset{2\rho}\longrightarrow T \hookrightarrow G,$$

\item $\Gamma_D$ is the preimage of the incidence divisor in $\Conf\times X$
under $$Y\times X \overset{D\times \on{id}}\to \Conf\times X.$$

\end{itemize} 

We let
$$\fs:\Conf\to \Gr^{\omega^\rho}_{G,\Conf}$$
denote the unit section, which sends $D$ to the triple $(D,\CP_G^{\omega^\rho},\on{id})$.

\sssec{}

Along with $\Gr^{\omega^\rho}_{G,\Conf}$ one introduces the corresponding version of the group-(ind)schemes
$$\fL^+(G)^{\omega^\rho}_\Conf \subset \fL(G)^{\omega^\rho}_\Conf,$$
see \cite[Sect. 12.2.1]{GLys2}, equipped with an action on $\Gr^{\omega^\rho}_{G,\Conf}$.  

\medskip

We can identify $\Gr^{\omega^\rho}_{G,\Conf}$ with the prestack quotient 
$$\fL(G)^{\omega^\rho}_\Conf/\fL^+(G)^{\omega^\rho}_\Conf,$$
sheafified in the \'etale topology. 

\sssec{}

A key feature of $\Gr^{\omega^\rho}_{G,\Conf}$ is the factorization structure over $\Conf$, i.e., a canonical isomorphism
\begin{equation} \label{e:factor Gr}
(\Gr^{\omega^\rho}_{G,\Conf}\times \Gr^{\omega^\rho}_{G,\Conf})\underset{\Conf\times \Conf}\times (\Conf\times \Conf)_{\on{disj}}\simeq 
\Gr^{\omega^\rho}_{G,\Conf} \underset{\Conf}\times (\Conf\times \Conf)_{\on{disj}},
\end{equation}
which is associative in the natural sense.  

\medskip

The same applies to the group-(ind)schemes $\fL(G)^{\omega^\rho}_\Conf$ and $\fL^+(G)^{\omega^\rho}_\Conf$. 

\sssec{}

Inside $\fL(G)^{\omega^\rho}_\Conf$ we consider the group ind-subscheme
$$\fL(N)^{\omega^\rho}_\Conf\subset \fL(G)^{\omega^\rho}_\Conf.$$

The purpose of the twist $\omega^\rho$ was that there exists a canonical character
$$\chi:\fL(N)^{\omega^\rho}_\Conf\to \BG_a,$$
see \cite[Sect. 8.2.1]{GLys2}.

\medskip

The character $\chi$ is compatible with the factorization structure in the natural sense. 

\sssec{}

It follows from  \cite[Sect. 3.3.4 and Corollary 3.3.6]{GLys1} that the form $q$ gives rise to a \emph{canonically} defined geometric metaplectic datum,
for $G$ (here we use the assumption that $G$ is simply-connected).  

\medskip

In particular, by \cite[Sect. 4.6.3]{GLys2}, we obtain a gerbe on $\Gr^{\omega^\rho}_{G,\Conf}$, denoted $\CG^G_q$, equipped with a 
factorization structure.

\ssec{Semi-infinite orbits and their closures}

\sssec{}

In what follows we will consider the closed subfunctor
$$\ol{S}{}^0_\Conf\subset \Gr^{\omega^\rho}_{G,\Conf}$$
and its open subfunctor
$$S^0_\Conf\subset \ol{S}{}^0_\Conf.$$

Namely, $\ol{S}{}^0_\Conf$ consists of those triples $(D,\CP_G,\alpha)$, for which for every dominant weight $\clambda$, the map
\begin{equation} \label{e:kappa map}
\omega^{\otimes \langle \rho,\clambda\rangle} \to \CV^\clambda_{\CP_G^{\omega^\rho}}\to \CV^\lambda_{\CP_G},
\end{equation} 
defined away from $\Gamma_D$, extends to all of $Y\times X$.

\medskip

In the above formula:

\begin{itemize}

\item $\CV^\clambda$ denotes the Weyl module for $G$ of highest weight $\clambda$;

\item $\CV^\clambda_{\CP_G}$ is the vector bundle, obtained as the twist of 
$\CV^\lambda$ by the given $G$-bundle;

\item $\omega^{\otimes \langle \rho,\clambda\rangle}:=(\omega^{\otimes \frac{1}{2}})^{\otimes \langle 2\rho,\clambda\rangle}$; 

\item The map $\omega^{\otimes \langle \rho,\clambda\rangle} \to \CV^\clambda_{\CP_G^{\omega^\rho}}$ is given by
the highest weight line in $\CV^\clambda$.

\end{itemize}

\sssec{}

We let 
$$S^0_\Conf \overset{\bj}\longrightarrow \ol{S}{}^0_\Conf$$ be the locus corresponding to the condition that \eqref{e:kappa map} are
\emph{injective bundle maps}, i.e., that the quotient is a vector bundle on $Y\times X$.  

\medskip

The subfunctors 
\begin{equation} \label{e:S subfunctors}
S^0_\Conf\subset \ol{S}{}^0_\Conf\subset \Gr^{\omega^\rho}_{G,\Conf}
\end{equation} 
inherit a factorization structure from that of $\Gr^{\omega^\rho}_{G,\Conf}$. 

\sssec{}

Both subfunctors \eqref{e:S subfunctors} 
are preserved by the action of $\fL(N)^{\omega^\rho}_\Conf$. Moreover, the action of $\fL(N)^{\omega^\rho}_\Conf$
on the unit section
$$\fs:\Conf\to S^0_\Conf\subset \Gr^{\omega^\rho}_{G,\Conf}$$
defines an isomorphism 
\begin{equation} \label{e:open S as quotient}
\fL(N)^{\omega^\rho}_\Conf/\fL^+(N)^{\omega^\rho}_\Conf\simeq S^0_\Conf,
\end{equation} 

\medskip

From here we obtain that the character $\chi$ induces a map 
$$S^0_\Conf\to \BG_a,$$
which we will denote by the same character $\chi$. 

\sssec{}

The gerbe $\CG^G_q$ is \emph{equivariant} with respect to the action of $\fL(N)^{\omega^\rho}_\Conf$ on $\Gr^{\omega^\rho}_{G,\Conf}$. 
Moreover, for any point of $\Gr^{\omega^\rho}_{G,\Conf}$, the resulting character sheaf on its stabilizer in $\fL(N)^{\omega^\rho}_\Conf$ is trivial. 

\medskip

In particular, the restriction
$$\CG^G_q|_{S^0_\Conf}$$
admits a canonical $\fL(N)^{\omega^\rho}_\Conf$-equivariant trivialization. 

\sssec{}

Let 
$$\ol{S}^{-,\Conf}_\Conf\subset \Gr^{\omega^\rho}_{G,\Conf}$$
be the closed subfunctor defined as follows: a $Y$-point $(D,\CP_G,\alpha)$ of $\Gr^{\omega^\rho}_{G,\Conf}$
belongs to $\ol{S}^{-,\Conf}_\Conf$ if for every dominant weight $\clambda$, the composite map
\begin{equation} \label{e:kappa - map}
\CV^{\vee,\lambda}_{\CP_G} \simeq \CV^{\vee,\clambda}_{\CP_G^{\omega^\rho}}\to \omega^{\otimes \langle \rho,\clambda\rangle}, 
\end{equation} 
which defined away from $\Gamma_D$, extends to a regular map
\begin{equation} \label{e:kappa - map corrected}
\CV^{\vee,\lambda}_{\CP_G} \to \omega^{\otimes \langle \rho,\clambda\rangle}(-\clambda(D)), 
\end{equation} 
where

\begin{itemize}

\item $\CV^{\vee,\clambda}$ denotes the dual Weyl module for $G$ of highest weight $\clambda$;

\item The map $\CV^{\vee,\clambda}_{\CP_G^{\omega^\rho}}\to \omega^{\otimes \langle \rho,\clambda\rangle}$ corresponds 
to the canonical $N^-$-invariant functional on  $\CV^{\vee,\clambda}$;

\item $-\clambda(D)$ is the effective (!) Cartier divisor on $Y\times X$ equal to the pullback of the tautological
effective divisor on $\on{Div}^{\on{eff}}\times X$ along the map
$$Y\times X \overset{D\times \on{id}}\longrightarrow \Conf\times X \overset{-\clambda\times \on{id}}\longrightarrow 
\on{Div}^{\on{eff}}\times X.$$

\end{itemize} 

\sssec{}

Let 
\begin{equation} \label{e:embed open - orbit}
S^{-,\Conf}_\Conf \overset{\bj^-}\longrightarrow \ol{S}^{-,\Conf}_\Conf
\end{equation} 
be the open subfunctor corresponding to the locus, where the maps \eqref{e:kappa - map corrected} are regular bundle maps.

\sssec{}

Note that in addition to the unit section
$$\fs:\Conf\to \Gr^{\omega^\rho}_{G,\Conf},$$ there exists another canonical section
\begin{equation} \label{e:Conf section}
\fs^-:\Conf\to \Gr^{\omega^\rho}_{G,\Conf}
\end{equation}
that sends a $Y$-point of $\Conf$, given by $D$, to the $G$-bundle induced from the $T$-bundle
$$\clambda \mapsto \omega^{\otimes \langle \rho,\clambda\rangle}(-\clambda(D)).$$
This image of this section belongs to $S^{-,\Conf}_\Conf$. 

\medskip

The action of $\fL(N^-)^{\omega^\rho}_\Conf$ on this section defines an isomorphism
\begin{equation} \label{e:negative orbit}
\fL(N^-)^{\omega^\rho}_\Conf/\on{Ad}_{\fs^-}(\fL^+(N^-)^{\omega^\rho}_\Conf)\to S^{-,\Conf}_\Conf.
\end{equation} 

\sssec{}

The gerbe $\CG^G_q$ is equivariant also with respect to $\fL(N^-)^{\omega^\rho}_\Conf$, and for any point of $\Gr^{\omega^\rho}_{G,\Conf}$, 
the resulting character sheaf on its stabilizer in $\fL(N^-)^{\omega^\rho}_\Conf$ is trivial. 

\medskip

Moreover, according to \cite[Sect. 5.1]{GLys1}, the pullback of $\CG^G_q$ along $\fs^-$ identifies canonically with $\CG^\Conf_q$.

\medskip

From here it follows that the restriction of $\CG^G_q$ to $S^{-,\Conf}_\Conf$
carries a (unique) $\fL(N^-)^{\omega^\rho}_\Conf$-equivariant identification with the pullback of the gerbe $\CG_q^\Conf$
along 
$$S^{-,\Conf}_\Conf\to \Conf.$$

\sssec{}

Let $\CG^{\on{ratio}}_q$ denote the gerbe on $\Gr^{\omega^\rho}_{G,\Conf}$ equal to the ratio\footnote{The notation diverges from one in 
\cite{GLys2} by replacing $q$ by $q^{-1}$.}
$$\CG^{\on{ratio}}_q:=\CG^\Conf_q \otimes (\CG^G_q)^{-1},$$
where, by a slight abuse of notation, we denote by $\CG^\Conf_q$ the pullback of the same-named gerbe along
$$\Gr^{\omega^\rho}_{G,\Conf}\to \Conf.$$

\medskip

We obtain that the restriction of $\CG^{\on{ratio}}_q$ to $S^{-,\Conf}_\Conf$ admits
a canonical $\fL(N^-)^{\omega^\rho}_\Conf$-equivariant trivialization. 

\ssec{Zastava spaces}  \label{ss:Zast}

In this subsection we recall the definition of the (several versions of the) Zastava space.

\sssec{}

The (compactified) Zastava space $\ol\CZ$ is defined as 
$$\ol\CZ:=\ol{S}{}^0_\Conf\cap \ol{S}^{-,\Conf}_\Conf\subset \Gr^{\omega^\rho}_{G,\Conf}.$$

Let $\pi$ denote the projection $\ol\CZ\to \Conf$.

\medskip

A priori, $\ol\CZ$ is an ind-scheme, but one shows (using global considerations, see \cite[Corollary 20.2.3]{GLys2}) 
that $\ol\CZ$ is actually a scheme. Since $\Gr^{\omega^\rho}_{G,\Conf}$ is ind-proper over $\Conf$, we
obtain that the natural projection
$$\ol\pi:\ol\CZ\to \Conf$$
is proper. 

\sssec{}

The functor represented by $\ol\CZ$ can be explicitly described as follows. For a test-scheme $Y$, a $Y$-point of $\ol\CZ$ is a triple
$$(D,\CP_G,\{\kappa\},\{\kappa^-\}),$$ where

\medskip

\begin{itemize}

\item $D$ is a $Y$-valued point of $\Conf$;

\item $\CP_G$ is a $G$-bundle on $Y\times X$;

\item $\{\kappa\}$ is a collection of maps
\begin{equation} \label{e:kappa}
\omega^{\otimes \langle \rho,\clambda\rangle}\to \CV^\clambda_{\CP_G}, \quad \clambda\in \cLambda^+
\end{equation} 
satisfying the Pl\"ucker relations;

\item $\{\kappa^-\}$ is a collection of maps
\begin{equation} \label{e:kappa-}
\CV^{\vee,\clambda}_{\CP_G}\to \omega^{\otimes \langle \rho,\clambda\rangle}(-\clambda(D)), \quad\clambda\in \cLambda^+,
\end{equation} 
satisfying the Pl\"ucker relations.

\end{itemize} 

\medskip

We require that:

\begin{itemize}

\item the composite maps
$$\omega^{\otimes \langle \rho,\clambda\rangle}\to \CV^\clambda_{\CP_G} 
\to \CV^{\vee,\clambda}_{\CP_G}\to \omega^{\otimes \langle \rho,\clambda\rangle}(-\clambda(D))$$
are the tautological maps 
$$\omega^{\otimes \langle \rho,\clambda\rangle}\to \omega^{\otimes \langle \rho,\clambda\rangle}(-\clambda(D)).$$

\end{itemize} 

\sssec{}  \label{sss:usual Zastava}

We define the open subfunctors
$$
\CD
& &  \CZ  \\
& &  @VV{\bj^-_Z}V  \\
\CZ^-  @>{\bj_Z}>>  \ol\CZ 
\endCD
$$
to be
$$\CZ :=\ol{S}{}^0_\Conf\cap S^{-,\Conf}_\Conf \,\,\text{ and }\,\,
\CZ^- :=S^0_\Conf\cap \ol{S}^{-,\Conf}_\Conf,$$
respectively. 

\medskip

These subschemes correspond to the condition that the maps \eqref{e:kappa-} (resp., \eqref{e:kappa}) are injective bundle maps. 

\medskip

Let $\pi$ (resp., $\pi^-$) denote the restriction of $\ol\pi$ to $\CZ$ (resp., $\CZ^-$). 

\medskip

Note that since the maps 
$$S^0_\Conf \to \Conf \leftarrow S^{-,\Conf}_\Conf$$
are ind-affine, the maps
$$\CZ \overset{\pi}\to \Conf \overset{\pi^-}\leftarrow \CZ^-$$
are affine. 

\sssec{}

Set also
$$\oCZ:=\CZ\cap \CZ^-\subset \ol\CZ.$$

We have the corresponding open embeddings
$$
\CD
\oCZ  @>{\overset{\circ}\bj{}_Z}>>  \CZ  \\
@V{\overset{\circ}\bj{}^-_Z}VV   @VV{\bj^-_Z}V  \\
\CZ^-  @>>{\bj_Z}>  \ol\CZ.
\endCD
$$

Let $\opi$ denote the restriction of $\pi$ to $\oCZ$. This map is also affine, as are the maps $\overset{\circ}\bj{}_Z$ and 
$\overset{\circ}\bj{}^-_Z$.

\medskip

It is known that $\oCZ$ is a smooth scheme; its connected component living over $\Conf^\lambda\subset \Conf$
has dimension $-\langle \lambda,2\check\rho\rangle$. 

\sssec{}

We let $\CG^G_q$ denote the restriction of the same-named gerbe along
$$\ol\CZ\hookrightarrow \Gr_{G,\Conf}.$$

Let $\CG^\Conf_q$ denote the pullback of the same-named gerbe along $\ol\pi$. 
Let $\CG^{\on{ratio}}_q$ denote the ratio of the latter by the former.

\medskip

We obtain that 
$$\CG^G_q|_{\CZ^-} \text{ and } \CG^{\on{ratio}}_q|_{\CZ}$$
admit canonical trivializations.

\sssec{}

The spaces 
$$\oCZ,\,\, \CZ,\,\, \CZ^-, \,\,\ol\CZ$$
all inherit a factorization structure from that on $\Gr^{\omega^\rho}_{G,\Conf}$,
along with the gerbes
$$\CG^G_q,\,\,\CG^\Conf_q,\,\, \CG^{\on{ratio}}_q.$$

\ssec{Sheaves on Zastava spaces}    \label{ss:shv on Zast}

\sssec{}

The restriction of the map 
$$\chi:S^0_\Conf\to \BG_a$$
along $\CZ^-\hookrightarrow S^0_\Conf$
defines a map 
$$\CZ^-\to \BG_a,$$
which we denote by the same symbol $\chi$.

\medskip

Let
$$\on{Vac}_{\Whit,Z}\in \Shv(\CZ^-)$$
denote the object
$$\omega_{\CZ^-}\overset{*}\otimes \chi^*(\on{exp}),$$
where $\on{exp}$ denotes the exponential/Artin-Schreier local system on $\BG_a$. 

\medskip

The object $\on{Vac}_{\Whit,Z}$ is equipped with a factorization structure along $\Conf$ in a natural sense. 

\medskip

We will refer to $\on{Vac}_{\Whit,Z}$ as the basic (a.k.a. vacuum Whittaker) sheaf on $\CZ^-$.

\sssec{}

Due to the trivialization of $\CG^G_q|_{\CZ^-}$, we can think of $\on{Vac}_{\Whit,Z}$ as an object of the twisted category $\Shv_{\CG^G_q}(\CZ^-)$,
and in this capacity we will denote it by
$$\on{Vac}_{q,\Whit,Z}\in \Shv_{\CG^G_q}(\CZ^-).$$

This object is also equipped with a factorization structure along $\Conf$. 

\sssec{}

Consider the category $\Shv(\oCZ)$ and the object
$$\IC_{\oCZ}\in \Shv(\oCZ).$$

Note, however, that since $\oCZ$ is smooth, 
$$\IC_{\oCZ}\simeq \omega_{\oCZ}[-\deg],$$
where $\deg$ takes value $-\langle \lambda,2\check\rho\rangle$ over the connected component $\Conf^\lambda$ of $\Conf$.

\medskip

Due to the trivialization of $\CG^{\on{ratio}}_q|_{\CZ}$, we can think of $\IC_{\oCZ}$ as an object of
the category $\Shv_{\CG^{\on{ratio}}_q}(\oCZ)$. When viewed in this capacity, we will denote it by
$$\IC_{q,\oCZ}\in \Shv_{\CG^{\on{ratio}}_q}(\oCZ).$$

\medskip

We will consider several variants of its extension to an object of 
$\Shv_{\CG^{\on{ratio}}_q}(\CZ^-)$:
$$\nabla^-_{q,Z}:=(\overset{\circ}\bj{}^-_Z)_*(\IC_{q,\oCZ}),\quad\Delta^-_{q,Z}:=(\overset{\circ}\bj{}^-_Z)_!(\IC_{q,\oCZ}),$$
$$\IC^-_{q,Z}:=(\overset{\circ}\bj{}^-_Z)_{!*}(\IC_{q,\oCZ}).$$

We have the tautological maps
\begin{equation} \label{e:semiinf maps}
\Delta^-_{q,Z} \to \IC^-_{q,Z} \to \nabla^-_{q,Z}.
\end{equation} 

We have the following result (see \cite[Theorem 7.6]{Ga3}):

\begin{prop} \label{p:IC Zast irrat}
For $q$ which is \emph{non-torsion valued}, the maps \eqref{e:semiinf maps} are isomorphisms.
\end{prop} 

\sssec{}

Set\footnote{See Remark \ref{r:Gauss} for the explanation of the name Gauss.}
$$\on{Gauss}^-_{q,*}:=\on{Vac}_{\Whit,Z}\sotimes \nabla^-_{q,Z}, \quad \on{Gauss}^-_{q,!}:=\on{Vac}_{\Whit,Z}\sotimes \Delta^-_{q,Z},$$
$$\on{Gauss}^-_{q,!*}:=\on{Vac}_{\Whit,Z}\sotimes \IC^-_{q,Z}.$$

The above three are objects of the category
$$\Shv_{\CG^\Conf_q}(\CZ^-),$$
where we are using $\sotimes$ as a functor
$$\Shv_{\CG^G_q}(\CZ^-)\times \Shv_{\CG^{\on{ratio}}_q}(\CZ^-)\to \Shv_{\CG^\Conf_q}(\CZ^-).$$

All three are perverse sheaves, and they each identify with the !-, *- and !*-extension, respectively, of their common restriction to $\oCZ$. 
The maps \eqref{e:semiinf maps} induce maps
\begin{equation} \label{e:Gauss- maps}
\on{Gauss}^-_{q,!}\to \on{Gauss}^-_{q,!*}\to \on{Gauss}^-_{q,*}.
\end{equation} 

According to \propref{p:IC Zast irrat}, the maps \eqref{e:Gauss- maps} are isomorphisms for $q$ which is non-torsion valued. 

\medskip

In addition, we have
$$\BD^{\on{Verdier}}(\on{Gauss}^-_{q,*})\simeq \on{Gauss}^-_{q^{-1},!} \text{ and }
\BD^{\on{Verdier}}(\on{Gauss}^-_{q,!*})\simeq \on{Gauss}^-_{q^{-1},!*},$$
up to replacing $\on{exp}$ by its inverse.

\sssec{}

We now have the following key assertion:

\begin{thm}  \label{t:acyclicity}
The maps 
$$(\bj_Z)_!\left(\on{Gauss}^-_{q,*}\right)\to (\bj_Z)_*\left(\on{Gauss}^-_{q,*}\right)$$
$$(\bj_Z)_!\left(\on{Gauss}^-_{q,!}\right)\to (\bj_Z)_*\left(\on{Gauss}^-_{q,!}\right)$$
and
$$(\bj_Z)_!\left(\on{Gauss}^-_{q,!*}\right)\to (\bj_Z)_*\left(\on{Gauss}^-_{q,!*}\right)$$
are isomorphisms.
\end{thm}

In other words, \thmref{t:acyclicity} says that extension of each of the objects $\on{Gauss}^-_{q,*}$, $\on{Gauss}^-_{q,!}$ and 
$\on{Gauss}^-_{q,!*}$ along 
$$\bj_Z:\CZ^-\to \ol\CZ$$
is clean.

\begin{proof}[Proof of \thmref{t:acyclicity}] 

The proof of the assertion concerning $\on{Gauss}^-_{q,!*}$ repeats that of \cite[Theorem 7.3]{Ga3}, where we replace 
\cite[Proposition 7.9]{Ga3} by a more precise reference, namely \cite[Theorem 4.2.1]{Camp}. 

\medskip

The statements concerning $\on{Gauss}^-_{q,*}$ and $\on{Gauss}^-_{q,!}$ are Verdier duals of one another, so we will prove the latter.

\medskip

The assertion concerning $\on{Gauss}^-_{q,!}$ follows from \corref{c:Frob for Gauss} below: the object $\on{Gauss}^-_{q,!}$ admits a filtration indexed
by $\Lambda^{\on{neg}}$ by objects of the form
$$(\bi^\lambda)_*(\Omega^{\sharp,\on{cl}}|_{X^\lambda}\boxtimes \on{Gauss}^-_{q,!*}),$$
where $\bi^\lambda$ is the (finite) map
\begin{equation} \label{e:act on Zast}
X^\lambda\times \ol\CZ\to \ol\CZ,
\end{equation} 
given by adding the divisor: 
$$\bi^\lambda(D',(D,\CP_G,\alpha))=(D+D',\CP_G,\alpha),$$
so that we have a commutative diagram
$$
\CD
X^\lambda\times \CZ^-  @>{\bi^\lambda}>>  \CZ^-   \\
@V{\on{id}\times \bj_Z}VV   @VV{\bj_Z}V  \\
X^\lambda\times \ol\CZ   @>{\bi^\lambda}>> \ol\CZ. 
\endCD
$$

\end{proof} 

\sssec{}

Denote by 
\begin{equation} \label{e:Gauss objects}
\on{Gauss}_{q,*}, \quad \on{Gauss}_{q,!} \text{ and } \on{Gauss}_{q,!*}
\end{equation} 
the resulting three objects of the category $\Shv_{\CG^\Conf_q}(\ol\CZ)$.

\medskip

By construction, they carry a natural factorization structure with respect to $\Conf$. The maps \eqref{e:Gauss- maps} induce maps
\begin{equation} \label{e:Gauss maps}
\on{Gauss}_{q,!}\to \on{Gauss}_{q,!*}\to \on{Gauss}_{q,*}.
\end{equation} 

These maps are isomorphisms for $q$ which is non-torsion valued. 

\begin{cor} \label{c:Gauss duality} \hfill

\smallskip

\noindent{\em(a)} 
The objects $\on{Gauss}_{q,*}$, $\on{Gauss}_{q,!}$ and $\on{Gauss}_{q,!*}$ 
are perverse, and identify with the *-, !- and Goresky-MacPherson extensions along
$\oCZ\hookrightarrow \ol\CZ$ of 
$$\chi^*(\on{exp})\overset{*}\otimes \IC_{q,\oCZ}.$$

\smallskip

\noindent{\em(b)} 
We have canonical isomorphisms
$$\BD^{\on{Verdier}}(\on{Gauss}_{q,*})\simeq \on{Gauss}_{q^{-1},!} \text{ and }
\BD^{\on{Verdier}}(\on{Gauss}_{q,!*})\simeq \on{Gauss}_{q^{-1},!*},$$
up to replacing $\on{exp}$ by its inverse.
\end{cor}

\ssec{The Whittaker category}  \label{ss:Whit}

In the next few subsections we will explain an alternative (in a sense more conceptual) construction of the objects \eqref{e:Gauss objects},
using some material from \cite{GLys2}. 

\medskip

The contents of Sects. \ref{ss:Whit}-\ref{ss:Gauss via semiinf} will not be used in the sequel. 

\sssec{}

Since the gerbe $\CG^G_q$ is $\fL(N)^{\omega^\rho}_\Conf$-equivariant, it makes sense to consider the category 
$$\Whit_{q,\Conf}(G):=\Shv_{\CG^G_q}(\Gr^{\omega^\rho}_{G,\Conf})^{\fL(N)^{\omega^\rho}_\Conf,\chi^*(\on{exp})},$$
see \cite[Sect, 8.4.2]{GLys2}.

\medskip

In the above formula, the superscript $\fL(N)^{\omega^\rho}_\Conf,\chi^*(\on{exp})$ means twisted-equivariance with respect to
$\fL(N)^{\omega^\rho}_\Conf$ against the pullback of $\on{exp}$ by means of $\chi$.

\sssec{}

Consider also the categories 
$$\Whit_{q,\Conf}(G)^{\leq 0}:=\Shv_{\CG^G_q}(\ol{S}{}^0_\Conf)^{\fL(N)^{\omega^\rho}_\Conf,\chi^*(\on{exp})},$$
and 
$$\Whit_{q,\Conf}(G)^{=0}:=\Shv_{\CG^G_q}(S^0_\Conf)^{\fL(N)^{\omega^\rho}_\Conf,\chi^*(\on{exp})}.$$

We can view $\Whit_{q,\Conf}(G)^{\leq 0}$ as a full subcategory of $\Whit_{q,\Conf}(G)$ via the closed embedding
$$\ol{S}{}^0_\Conf\subset \Gr^{\omega^\rho}_{G,\Conf}.$$

\medskip

From the identification \eqref{e:open S as quotient}, we obtain:

\begin{lem} \label{l:Whit on open}
Restriction along $\Conf\overset{\fs}\to S^0_\Conf$
defines an equivalence
$$\Whit_{q,\Conf}(G)^{=0}\to \Shv(\Conf).$$
\end{lem} 

In addition, we have: 

\begin{lem} \label{l:Whit on closed}
Any object of $\Whit_{q,\Conf}(G)^{=0}$ supported on $\ol{S}^0_\Conf-S^0_\Conf$ is zero.
\end{lem}

\begin{proof} 

This is a standard stabilizer calculation: for a point on $\ol{S}^0_\Conf-S^0_\Conf$, the character
sheaf on its stabilizer, given by the restriction of $\on{exp}$, is non-trivial.

\end{proof}

\begin{cor} \label{c:Whit on closed}
The functor 
$$\Whit_{q,\Conf}(G)^{\leq 0}\to \Whit_{q,\Conf}(G)^{=0},$$
given by restriction, is an equivalence.
\end{cor} 

Combining \lemref{l:Whit on open} and \corref{c:Whit on closed}, we obtain:
\begin{cor} \label{c:Whit on sect}
Restriction along $\Conf\overset{\fs}\to \ol{S}^0_\Conf$
defines an equivalence
$$\Whit_{q,\Conf}(G)^{\leq 0}\to \Shv(\Conf).$$
\end{cor} 

\sssec{}

Let 
$$\on{Vac}_{q,\Whit}\in \Whit_{q,\Conf}(G)^{\leq 0}\subset \Shv_{\CG^G_q}(\ol{S}{}^0_\Conf)\subset \Shv_{\CG^G_q}(\Gr^{\omega^\rho}_{G,\Conf})$$
be the object that corresponds under the equivalence of \corref{c:Whit on sect} to
$$\omega_\Conf\in \Shv(\Conf).$$

It follows from \lemref{l:Whit on closed} that the canonical map
\begin{equation} \label{e:Whit Vac isom}
\bj_!(\on{Vac}_{q,\Whit}|_{S^0_\Conf})\to \bj_*(\on{Vac}_{q,\Whit}|_{S^0_\Conf})
\end{equation} 
is an isomorphism. 

\sssec{}

It is immediate from the definitions that we have a canonical isomorphism
\begin{equation} \label{e:Vac Whit Z-}
\on{Vac}_{q,\Whit}|_{\CZ^-}\simeq \on{Vac}_{q,\Whit,Z},
\end{equation} 
where $-|_-$ denotes the !-restriction. 

\ssec{The semi-infinite category}


%
%
%


\sssec{}

Set
$$\SI^-_{q,\Conf}(G):=\Shv_{\CG^{\on{ratio}}_q}(\Gr^{\omega^\rho}_{G,\Conf})^{\fL(N^-)^{\omega^\rho}_\Conf}.$$

Set also
$$\SI^-_{q,\Conf}(G)^{\leq 0}:=\Shv_{\CG^{\on{ratio}}_q}(\ol{S}^{-,\Conf}_\Conf)^{\fL(N^-)^{\omega^\rho}_\Conf}$$
and
$$\SI^-_{q,\Conf}(G)^{=0}:=\Shv_{\CG^{\on{ratio}}_q}(S^{-,\Conf}_\Conf)^{\fL(N^-)^{\omega^\rho}_\Conf}.$$

\medskip

As in the case of $\Whit$, restriction along $\fs^-$ defines an equivalence
\begin{equation} \label{e:semiinf open}
\SI^-_{q,\Conf}(G)^{=0}\to \Shv(\Conf).
\end{equation} 

\sssec{}

The trivialization of $\CG^{\on{ratio}}_q$ over $S^{-,\Conf}_\Conf$ defines an equivalence
$$\SI^-_{q,\Conf}(G)^{=0}:=\Shv_{\CG^{\on{ratio}}_q}(S^{-,\Conf}_\Conf)^{\fL(N^-)^{\omega^\rho}_\Conf}\simeq
\Shv(S^{-,\Conf}_\Conf)^{\fL(N^-)^{\omega^\rho}_\Conf}.$$

Thus, $\omega_{S^{-,\Conf}_\Conf}$ makes sense as an object of $\SI^-_{q,\Conf}(G)^{=0}$; when considered in this role
we will denote it by $\omega_{q,S^{-,\Conf}_\Conf}$. Under the equivalence
\eqref{e:semiinf open}, it corresponds to $\omega_\Conf\in \Shv(\Conf)$.

\medskip

We will consider several variants of the extension of $\omega_{q,S^{-,\Conf}_\Conf}$ to an object of
$\SI^-_{q,\Conf}(G)^{\leq 0}$.

\sssec{}

One such extension is 
$$\nabla^-_q:=(\bj^-)_*(\omega_{q,S^{-,\Conf}_\Conf}),$$
i.e., we apply *-extension with respect to the open embedding \eqref{e:embed open - orbit}.

\medskip 

Another extension is
$$\Delta^-_q:=(\bj^-)_!(\omega_{q,S^{-,\Conf}_\Conf}).$$

\sssec{}

We have a canonical map
$$\Delta^-_q\to \nabla^-_q.$$

We claim:

\begin{prop}  \label{p:semiinf clean}
Assume that $q$ is \emph{non-torsion valued}. Then the map
$$\Delta^-_q\to \nabla^-_q$$
is an isomorphism. 
\end{prop}

\begin{proof}

The indscheme $\ol{S}^{-,\Conf}_\Conf$ admits a stratification parameterized by $\Lambda^{\on{pos}}$
$$\ol{S}^{-,\Conf}_\Conf=\underset{\mu}\cup\, S^{-,\Conf+\mu}_\Conf$$
with 
$$S^{-,\Conf+0}_\Conf=S^{-,\Conf}_\Conf;$$
see \cite[Sect. 12.2]{GLys2}. 
We need to show that the !-restriction of $\Delta^-_q$ to any $S^{-,\Conf+\mu}_\Conf$ with $\mu\neq 0$
vanishes. We will obtain the required vanishing from 
considerations of equivariance with respect to 
$$\fL(N)^{\omega^\rho}_\Conf\cdot \fL^+(T)_\Conf:=\fL(B)^{\omega^\rho}_\Conf\underset{\fL(T)_\Conf}\times \fL^+(T)_\Conf$$
acting on $\Gr^{\omega^\rho}_{G,\Conf}$. 

\medskip

The initial observation is that the gerbe $\CG^G_q$ is equivariant with respect to $\fL^+(T)_\Conf \cdot \fL(N)^{\omega^\rho}_\Conf$.

\medskip

The action of $\fL^+(T)_\Conf$ preserves the section $\fs^-$. We equip $\CG^\Conf_q$ with a structure of $\fL^+(T)_\Conf$-equivariance
so that the identification 
$$(\fs^-)^*(\CG^G_q)\simeq \CG^\Conf_q$$
is $\fL^+(T)_\Conf$-equivariant. 
We regard $\CG^\Conf_q$ as $\fL(N)^{\omega^\rho}_\Conf\cdot \fL^+(T)_\Conf$-equivariant via the projection 
$$\fL(N)^{\omega^\rho}_\Conf\cdot \fL^+(T)_\Conf\twoheadrightarrow \fL^+(T)_\Conf.$$
This equips the gerbe $\CG^{\on{ratio}}_q$ also with a $\fL(N)^{\omega^\rho}_\Conf\cdot \fL^+(T)_\Conf$-equivariant structure.

\medskip

With respect to the above equivariance structure on $\CG^{\on{ratio}}_q$, the object $\omega_{q,S^{-,\Conf}_\Conf}$ is 
equivariant. Hence, so is $\Delta^-_q$. Therefore, so is its !-restriction to any $S^{-,\Conf+\mu}_\Conf$.

\medskip

However, we claim that if $q$ is non-torsion valued, and $\mu\neq 0$, the category
$$\Shv_{\CG^{\on{ratio}}_q}(S^{-,\Conf+\mu}_\Conf)^{\fL(N)^{\omega^\rho}_\Conf\cdot \fL^+(T)_\Conf}$$
is zero. 

\medskip

Indeed, we claim that for any point on this stratum, the character sheaf on the stabilizer of this point in $\fL(N)^{\omega^\rho}_\Conf\cdot \fL^+(T)_\Conf$,
arising from the structure of $\fL(N)^{\omega^\rho}_\Conf\cdot \fL^+(T)_\Conf$-equivariance on $\CG^{\on{ratio}}_q$, is non-trivial. 

\medskip

Indeed, the point in question lives over a point
$$D=\Sigma\, \lambda_k\cdot x_k\in \Conf,$$
and it belongs to
$$\underset{k}\Pi\, S^{-,\lambda_k+\mu_k}_{x_k} \subset \underset{k}\Pi\, \Gr^{\omega^\rho}_{G,x_k}=
\Gr^{\omega^\rho}_{G,\Conf} \underset{\Conf}\times \{D\}, \quad \Sigma\, \mu_k=\mu.$$

\medskip

Translating by means of 
$$\underset{k}\Pi\, \fL(N)^{\omega^\rho}_{x_k}\simeq \fL(N)^{\omega^\rho}_\Conf \underset{\Conf}\times \{D\},$$
we can assume that our point is invariant with respect to
\begin{equation} \label{e:prod of tori}
\underset{k}\Pi\, \fL^+(T)_{x_k}\simeq \fL^+(T)_\Conf \underset{\Conf}\times \{D\}.
\end{equation} 

Now, the character sheaf on \eqref{e:prod of tori} arising from the structure of $\fL^+(T)_\Conf$-equivariance on $\CG^{\on{ratio}}_q$
equals the pullback of
$$\underset{k}\boxtimes \, \Psi_{b(\mu_k,-)}$$
along
$$\underset{k}\Pi\, \fL^+(T)_{x_k}\to \underset{k}\Pi\, T,$$
where $\Psi$ denotes the Kummer local system. 

\medskip

This character sheaf is non-trivial, since for $\mu'\neq 0$, the element 
$$b(\mu',-)\in \Hom(\Lambda,\fZ)$$
is non-zero, by the assumption on $q$. 

\end{proof}

\sssec{}

Finally, we introduce the third object of $\SI^-_{q,\Conf}(G)^{\leq 0}$ that we will consider, to be denoted
$\ICsm_q$.

\medskip

Breaking $G$ into simple factors, we can assume that $q$ is either torsion-valued or non-torsion valued. 

\medskip

When $q$ is non-torsion valued, we set
$$\Delta^-_q=:\ICsm_q:=\nabla^-_q,$$
where the composite isomorphism is justified by \propref{p:semiinf clean}. 

\medskip

When $q$ is torsion-valued, we let $\ICsm_q$ be the object introduced in 
\cite[Sect. 18.3.2]{GLys2} under the name ``metaplectic semi-infinite IC sheaf", and denoted there by $\ICsm_{q,\Conf}$.

\begin{rem} Here are three ways that describe the nature of $\ICsm_{q,\Conf}$:

\medskip

\noindent{(i)} We have a naturally defined projection 
$p:\ol{S}^{-,\Conf}_\Conf\to \ol{\Bun}_{B^-}$,
and the gerbe $\CG_q^{\on{ratio}}$ identifies with the pullback of the same-named gerbe on
$\ol{Bun}_{B^-}$, where the latter is the ratio of the pullbacks of the corresponding gerbes
along the projections
$$\Bun_G \leftarrow \ol{\Bun}_{B^-}\to \Bun_T,$$
respectively. The gerbe $\CG_q^{\on{ratio}}$ on $\ol{\Bun}_{B^-}$ is canoniocally trivialized 
over $\Bun_{B^-}$, and let $\ICsm_{q,\on{glob}}$ define the Goresky-MacPherson extension 
of the (shifted) constant sheaf on $\ol{\Bun}_{B^-}$, viewed as a $\CG_q^{\on{ratio}}$-twisted
perverse sheaf. Then we have
$$\ICsm_{q,\Conf}\simeq p^!(\ICsm_{q,\on{glob}}).$$

The above isomorphism describes $\ICsm_{q,\Conf}$ uniquely. However, its one disadvantage is
that it is non-local on $X$ (since $\ol\Bun_{B^-}$ is a global object). In particular, from this definition,
the factorization property of $\ICsm_{q,\Conf}$ is not easy to obtain. 

\medskip

\noindent{(ii)} One can introduce a Ran space version of the semi-infinite category, to be denoted 
$\SI^-_{q,\Ran}(G)$, and equip it with a t-structure, and define an object $\ICsm_{q,\Ran}$ as
the minimal extension of the dualizing sheaf on $S^{-,\Conf}_{\Ran}$. The object $\ICsm_{q,\Conf}$
can be obtained from $\SI^-_{q,\Ran}(G)$ by restriction.

\medskip

\noindent{(iii)} One can define $\ICsm_{q,\Conf}$ by an explicit colimit procedure, mimicking
\cite[Sect. 2.3]{Ga1}.  

\end{rem} 

\sssec{}

Note that by construction, we have an isomorphism
\begin{equation} \label{e:IC Ran via conf open} 
\on{IC}_{q,\oCZ}\simeq \omega_{q,S^{-,\Conf}_\Conf}|_{\oCZ}[-\on{deg}], 
\end{equation}
as objects of $\Shv_{\CG_q^\Conf}(\oCZ)$, 
where we remind that $-|_-$ stands for !-restriction.

\medskip

The isomorphism extends to an isomorphism
\begin{equation} \label{e:IC Ran via conf *} 
\nabla^-_{q,Z}\simeq \nabla^-_q|_{\CZ^-}[\on{deg}],
\end{equation}
as objects of $\Shv_{\CG_q^\Conf}(\CZ^-)$ (indeed, both sides are *-extensions from $\oCZ$). 

\medskip

In addition, we claim:

\begin{prop} \label{p:IC Ran via conf}
The isomorphism \eqref{e:IC Ran via conf open} extends to isomorphisms 
\begin{equation} \label{e:IC Ran via conf !} 
\Delta^-_{q,Z}\simeq \Delta^-_q|_{\CZ^-}[-\on{deg}]
\end{equation}
\begin{equation} \label{e:IC Ran via conf !*} 
\on{IC}_{q,\CZ^-}\simeq \ICsm_q|_{\CZ^-}[-\on{deg}]
\end{equation}
as objects of $\Shv_{\CG_q^\Conf}(\CZ^-)$. 
\end{prop}

\begin{proof}

When $q$ is non-torsion-valued, the assertion follows from Propositions \ref{p:IC Zast irrat} and
\ref{p:semiinf clean}.

\medskip

Assume now that $q$ is torsion-valued. In this case, the isomorphism  \eqref{e:IC Ran via conf !} is a metaplectic 
version of the combination of \cite[Theorem 3.2.4]{Ga1} and \cite[Corollary 3.6.5]{Ga2}.
The isomorphism \eqref{e:IC Ran via conf !*} is a metaplectic version of \cite[Proposition 3.8.3]{Ga2}.
\end{proof} 

\ssec{Gauss objects via the Whittaker and the semi-infinite category}   \label{ss:Gauss via semiinf}

\sssec{}

Consider the objects 
\begin{equation} \label{e:Gauss objects concept}
\on{Vac}_{q,\Whit}\sotimes \nabla^-_q,\quad \on{Vac}_{q,\Whit}\sotimes \Delta^-_q \text{ and }
 \on{Vac}_{q,\Whit}\sotimes \ICsm_q
\end{equation}
in the category 
$$\Shv_{\CG_q^\Conf}(\Gr_{G,\Conf}^{\omega^\rho}).$$

By construction, they are supported on 
$$\ol\CZ\subset \Gr_{G,\Conf}^{\omega^\rho},$$
so we can think of them as objects of $\Shv_{\CG_q^\Conf}(\ol\CZ)$. 

\medskip

Furthermore, it follows from the definitions that that their further !-restrictions to $\oCZ\subset \ol\CZ$ all identify canonically with
$$\on{Gauss}_{q,*}|_{\oCZ}\simeq \on{Gauss}_{q,!}|_{\oCZ} \simeq \on{Gauss}_{q,!*}|_{\oCZ}.$$

We claim:

\begin{prop} \label{p:Guass via semiinf}
The above identifications extend to identifications 
$$\on{Vac}_{q,\Whit}\sotimes \nabla^-_q\simeq \on{Gauss}_{q,*},\quad \on{Vac}_{q,\Whit}\sotimes \Delta^-_q\simeq \on{Gauss}_{q,!},$$
$$\on{Vac}_{q,\Whit}\sotimes \ICsm_q \simeq  \on{Gauss}_{q,!*}$$
over all of $\ol\CZ$. 
\end{prop} 

\begin{proof}

By the isomorphism \eqref{e:Whit Vac isom}, the objects in the left-hand side are *-extensions of their respective restrictions to
$\CZ^-$. By \thmref{t:acyclicity}, the same is true for the objects appearing in the right-hand side.  Hence, it is enough to establish
the corresponding isomorphisms over $\CZ^-$. Now the assertion follows from the isomorphisms \eqref{e:IC Ran via conf *},
\eqref{e:IC Ran via conf !} and \eqref{e:IC Ran via conf !*}, respectively, 

\end{proof} 

\section{Factorization algebras arising from the Whittaker sheaf}  \label{s:Whit}

In this section we will state and prove the main results of this paper: they assert that the factorization algebras
$\Omega_q$, introduced earlier, can be obtained geometrically from sheaves on the Zastava space. 

\ssec{Factorization algebras $\Omega^{\on{sml}}_{q,\Whit}$, $\Omega^{\on{Lus}}_{q,\Whit}$ and $\Omega^{\on{DK}}_{q,\Whit}$}  \label{ss:Omega Whit}

\sssec{}

We define the objects 
\begin{equation}  \label{e:Omega Whit}
\Omega^{\on{DK}}_{q,\Whit},\,\, \Omega^{\on{Lus}}_{q,\Whit} \text{ and } \Omega^{\on{sml}}_{q,\Whit}
\end{equation}
in $\Shv_{\CG^\Conf_q}(\Conf)$ to be
$$\ol\pi_!(\on{Gauss}_{q,*}),\,\, \ol\pi_!(\on{Gauss}_{q,!}) \text{ and } \ol\pi_!(\on{Gauss}_{q,!*}),$$
respectively.

\medskip

Note that since the morphism $\ol\pi$ is proper, in the above formula, $\ol\pi_!$ could be replaced by $\ol\pi_*$. 

\begin{rem}
The interpretation of the Gauss objects given in \secref{ss:Gauss via semiinf} shows that we can construct 
$\Omega^{?}_{q,\Whit}$ as a value on $\on{Vac}_{q,\Whit}$ of a Jacquet-type functor
$$J_{\Whit}^?:\Whit_{q,\Conf}(G)\to \Shv_{\CG^\Conf_q}(\Conf),$$
defined by
$$J_{\Whit}^{\on{DK}}(\CF):=\ol\pi_!(\CF\sotimes \Delta^-_q),\,\,\, J_{\Whit}^{\on{Lus}}(\CF):=\ol\pi_!(\CF\sotimes \nabla^-_q)$$
and
$$J_{\Whit}^{\on{sml}}(\CF):=\ol\pi_!(\CF\sotimes \ICsm_q),$$
respectively. 

\end{rem}

\sssec{}

The factorization structure on the objects \eqref{e:Gauss objects} gives rise to a structure of 
factorization algebra in $\Shv_{\CG^\Conf_q}(\Conf)$ on the objects \eqref{e:Omega Whit}.

\sssec{}

From \corref{c:Gauss duality} we obtain:

\begin{cor} \label{c:Omega Whit duality}
We have the isomorphisms
$$\BD^{\on{Verdier}}(\Omega^{\on{DK}}_{q,\Whit})\simeq \Omega^{\on{Lus}}_{q^{-1},\Whit} \text{ and }
\BD^{\on{Verdier}}(\Omega^{\on{sml}}_{q,\Whit})\simeq \Omega^{\on{sml}}_{q^{-1},\Whit}.$$
\end{cor}

\begin{proof}

We only have to show that replacing $\on{exp}$ by its inverse does not affect the objects
\eqref{e:Omega Whit}. 

\medskip

The above replacement can be affected by the action of $(-1)\in \BG_m$ as an automorphism
of $\BG_a$. The morphism $\chi:\ol\CZ\to \BG_a$ is $\BG_m$-equivariant with respect to the
action of $\BG_m$ on $\ol\CZ$ given by
$$\BG_m \overset{2\rho}\to T\subset \fL^+(T)$$
and the action of the latter on $\Gr_G^{\omega^\rho}$, 
and the \emph{square} of the action of $\BG_m$ on $\BG_a$ by dilations. 

\medskip

The result now follows from the fact that the morphism $\ol\pi$ is $\BG_m$-equivariant. 

\end{proof}

\sssec{}

We now claim:

\begin{prop} \label{p:Omega Whit perv}
The objects \eqref{e:Omega Whit} are perverse.
\end{prop}

\begin{proof}

Recall that $\pi^-$ denotes the restriction of $\pi$ to $\CZ^-$. 
Recall also that the morphism $\pi^-$ is affine, while 
$$\on{Gauss}^-_{q,?}=\on{Gauss}_{q,?}|_{\CZ^-}$$ are perverse for $?=*,!$ or $!*$. 
Hence $\pi^-_!(\on{Gauss}^-_{q,?})$
is cohomologically $\geq 0$, while  $\pi^-_*(\on{Gauss}^-_{q,?})$ is cohomologically $\leq 0$.

\medskip

However, by \thmref{t:acyclicity}, we have
$$\pi^-_!(\on{Gauss}^-_{q,?})\simeq \Omega^{?}_{q,\Whit}\simeq \pi^-_*(\on{Gauss}^-_{q,?})$$
for each of the three versions.

\end{proof}

\ssec{Statement of the main results}

\sssec{}

First, we claim:

\begin{prop}  \label{p:Whit on open}
Assume that $q$ is non-degenerate. Then the natural maps 
$$\Omega^{\on{Lus}}_{q,\Whit}\to \Omega^{\on{sml}}_{q,\Whit}\to \Omega^{\on{DK}}_{q,\Whit}$$
become isomorphisms after restriction to $\oConf$. The resulting perverse sheaf on $\oConf$ identifies 
with $\oOmega_q$.
\end{prop}  

\begin{proof}

By factorization, the statement reduces to showing that for every simple coroot, the restriction of the objects
\eqref{e:Omega Whit} to $\Conf^{-\alpha_i}$ identifies with $\sfe_X[1]$. We have:
\begin{equation} \label{e:ident simple Zast}
\ol\CZ\underset{\Conf}\times \Conf^{-\alpha_i}\simeq X\times \BP^1,
\end{equation} 
while
$$\oCZ\underset{\Conf}\times \Conf^{-\alpha_i}\subset \CZ^-\underset{\Conf}\times \Conf^{-\alpha_i}\subset \ol\CZ\underset{\Conf}\times \Conf^{-\alpha_i}$$
correspond under the identification \eqref{e:ident simple Zast} 
$$X\times \BG_m\subset X\times \BA^1\subset X\times \BP^1,$$
respectively. 

\medskip

Now, it is shown in \cite[Sect. 18.4.10]{GLys2} that
in terms of the identification \eqref{e:ident simple Zast} (and the trivialization of the gerbe $\CG^\Conf_q$ over $\Conf^{-\alpha_i}$), we have:
$$\on{Gauss}^-_{q,?}|_{\oCZ}\simeq \sfe_X[1]\boxtimes \on{Gauss}_{q(\alpha_i)}[1]$$
where for $\zeta\in \fZ$ we denote 
\begin{equation} \label{e:Gauss zeta}
\on{Gauss}_\zeta:=\on{exp}|_{\BG_m}\overset{*}\otimes \Psi_\zeta,
\end{equation}
where we recall that $\Psi_\zeta$ denotes the Kummer local system on $\BG_m$ corresponding to the element $\zeta\in \fZ$.

\medskip

The objects
$$\on{Gauss}^-_{q,*},\,\, \on{Gauss}^-_{q,!}\,\, \on{Gauss}^-_{q,!*}$$
are given by *-, !- and !*- extensions of $\on{Gauss}^-_{q,?}|_{\oCZ}$, respectively. However, for $q(\alpha_i)\neq 0$
(the condition that $q$ is non-degenerate), the extension of $\Psi_{q(\alpha_i)}$ along $\BG_m\hookrightarrow \BA^1$ is clean. 

\medskip

Finally, we have:
$$H^i(\BG_m,\on{Gauss}_{\zeta})=
\begin{cases}
&0 \text{ for } i\neq 1,\\
&\sfe \text{ for } i=1. 
\end{cases}$$

\end{proof} 

\begin{rem}  \label{r:Gauss}

The name Gauss stems from the fact that the right-hand side in \eqref{e:Gauss zeta} is the local system,
whose cohomology is the geometric counterpart of the Gauss sum.

\end{rem}

\sssec{}

We are now ready to state the two main results of this paper. 

\begin{thm} \label{t:main 1}
Assume that $q$ avoids small torsion\footnote{See \secref{sss:small root} for what this means.}. Then the identification of 
\propref{p:Whit on open} extends to isomorphisms
$$\Omega^{\on{DK}}_{q,\Whit}\simeq \Omega^{\on{DK}}_{q}\,\, \text{ and } \,\,
\Omega^{\on{Lus}}_{q,\Whit}\simeq \Omega^{\on{Lus}}_{q}.$$
\end{thm}

Note that the two isomorphisms stated in \thmref{t:main 1} are obtained from one another
by Verdier duality. 

\begin{thm} \label{t:main 2}
Assume that $q$ avoids small torsion. Then the identification of \propref{p:Whit on open} extends to an isomorphism
$$\Omega^{\on{sml}}_{q,\Whit}\simeq \Omega^{\on{sml}}_{q}.$$
\end{thm}

\begin{rem} 
It is shown in \cite[Theorem 3.4.1]{Ras} that for $q=0$, the object $\Omega^{\on{sml}}_{q,\Whit}$ vanishes. It is then
a formal consequence of \thmref{t:Frob for Whit} below that in this case $\Omega^{\on{Lus}}_{q,\Whit}\simeq \Omega^{\on{cl}}$. 
\end{rem} 

\sssec{}

The plan of the rest of the paper is the following: in the next subsection we will show that \thmref{t:main 2}
is equivalent to \thmref{t:main 3}, which asserts a certain sub-top cohomology vanishing.  In \secref{ss:formal param Whit}
we will show that Theorems \ref{t:main 1} and \ref{t:main 2} are equivalent to one another. 

\medskip

In \secref{ss:quant Frob Whit}
we will prove Theorems \ref{t:main 1} and \ref{t:main 2} over a ground field of characteristic $0$, thereby also proving
\thmref{t:main 3}. 

\medskip

Finally, in \secref{s:subtop} we will show that the assertion of \thmref{t:main 3} over a ground field
of characteristic $0$ implies the assertion of \thmref{t:main 3} over any ground field. This will complete
the proof of Theorems \ref{t:main 1} and \ref{t:main 2} over any ground field as well. 

\ssec{Interpretation as sub-top cohomology vanishing} \label{ss:subtop}

The goal of this subsection is to reduce the assertion of \thmref{t:main 2} to an essentially 
combinatorial statement about MV cycles.  

\sssec{}

Note that by \corref{c:Omega Whit duality}, induction on $|\lambda|$, \'etale invariance and factorization, the assertion of \thmref{t:main 2} 
is equivalent to the fact that 
$$(\iota_\lambda)^!(\Omega^{\on{sml}}_{q,\Whit})\in \Vect_{\CG^\lambda_{q,x}}$$
lives in cohomological degrees $\geq 2$, as long as $\lambda$ is not a negative simple coroot. 

\medskip

Realizing $\Omega^{\on{sml}}_{q,\Whit}$ as 
$$\pi^-_*(\on{Gauss}_{q,!*}^-),$$
by base change, we have
$$(\iota_\lambda)^!(\Omega^{\on{sml}}_{q,\Whit})\simeq
\on{C}^\cdot\left(S^0\cap \ol{S}{}^{-,\lambda},\IC^-_{q,Z}|_{S^0\cap \ol{S}{}^{-,\lambda}}\overset{*}\otimes \chi^*(\on{exp})\right).$$

\medskip

Here $S^0$ (resp., $\ol{S}{}^{-,\lambda}$) are the fibers of $S^0_\Conf$ (resp., $S^{-,\Conf}_\Conf$) over the point $\lambda\cdot x\in \Conf$.
In other words, they are the $0$-th semi-infinite orbit (resp., closure of the $\lambda$ negative semi-infinite orbit). 

\medskip

Thus, the assertion of \thmref{t:main 2} is equivalent to the fact that the cohomology
\begin{equation} \label{e:subtop}
H^i\left(S^0\cap \ol{S}{}^{-,\lambda},\IC^-_{q,Z}|_{S^0\cap \ol{S}{}^{-,\lambda}}\overset{*}\otimes \chi^*(\on{exp})\right) 
\end{equation} 
vanishes for  $i=1$ as long as $\lambda$ is not a negative simple coroot. 

\medskip

We note also that the cohomology \eqref{e:subtop} vanishes for $i=0$ (for all $\lambda\in \Lambda^{\on{neg}}-0$, including the negative simple coroots). 
This expresses the fact that $\Omega^{\on{sml}}_{q,\Whit}$ lives in perverse degrees $\geq 0$.

\begin{rem}   \label{r:fiber Whit DK}

For future reference, we note that we also have
\begin{multline*} \label{e:!-fiber Omega DK Whit}
(\iota_\lambda)^!(\Omega^{\on{DK}}_{q,\Whit})\simeq
\on{C}^\cdot\left(S^0\cap S^{-,\lambda},\IC_{q,\oCZ}|_{S^0\cap S^{-,\lambda}}\overset{*}\otimes \chi^*(\on{exp})\right) \simeq \\
\simeq 
\on{C}^\cdot\left(S^0\cap S^{-,\lambda},\omega_{S^0\cap S^{-,\lambda}}[\langle \lambda,2\check\rho\rangle]\overset{*}\otimes \chi^*(\on{exp})\right).
\end{multline*} 

Consider the individual cohomologies. 
\begin{equation} \label{e:top}
H^i\left(S^0\cap S^{-,\lambda'},\omega_{S^0\cap S^{-,\lambda'}}[\langle \lambda',2\check\rho\rangle]\overset{*}\otimes \chi^*(\on{exp})\right)
\end{equation} 

The expression \eqref{e:top} vanishes for $i=0$; indeed, this is equivalent to the fact that 
$\Omega^{\on{DK}}_{q,\Whit}$ lives in perverse degrees $\geq 0$. However, we will shortly see a direct analysis proving this 
vanishing (this will be a rather simple cohomological estimate), see Sects. \ref{sss:case 2,i}-\ref{sss:case 2,ii}. 

\medskip

The vanishing of \eqref{e:top} in degree $i=1$ is equivalent to the fact that the map
$$\Omega^{\on{DK}}_{q,\Whit} \to H^0\left((\jmath^\lambda)_*\circ (\jmath^\lambda)^*(\Omega^{\on{DK}}_{q,\Whit})\right)$$
is an injection of perverse sheaves. As we will see in \secref{ss:formal param Whit}, this is essentially equivalent to the assertion of
\thmref{t:main 1}.  In \secref{sss:need to prove subtop} we will see that this vanishing is equivalent also to the assertion of
\thmref{t:main 2}.

\medskip

Thus, we will eventually prove that the cohomology \eqref{e:top} vanishes also in degree $i=1$
(provided that $q$ avoids small torsion). But this will be a much subtler analysis.

\end{rem} 

\sssec{}

First, we claim:

\begin{lem} \label{l:remove other strata}
The restriction map along $S^0\cap S^{-,\lambda}\hookrightarrow S^0\cap \ol{S}{}^{-,\lambda}$
\begin{multline*} 
H^i\left(S^0\cap \ol{S}{}^{-,\lambda},\IC^-_{q,Z}|_{S^0\cap \ol{S}{}^{-,\lambda}}\overset{*}\otimes 
\chi^*(\on{exp})\right)\to  \\
\to H^i\left(S^0\cap S^{-,\lambda},\IC_{q,\oCZ}|_{S^0\cap S^{-,\lambda}}\overset{*}\otimes 
\chi^*(\on{exp})\right)\simeq H^i\left(S^0\cap S^{-,\lambda},\omega_{S^0\cap S^{-,\lambda}}[\langle \lambda,2\check\rho\rangle]\overset{*}\otimes 
\chi^*(\on{exp})\right)
\end{multline*}
induces an isomorphism in degree  $i=1$.
\end{lem}

\begin{proof} 

Writing $\ol{S}{}^{-,\lambda}-S^{-,\lambda}$ as the union of $S^{-,\lambda'}$ with $\lambda'<\lambda$, 
by the Cousin spectral sequence, it suffices to show that
$$H^i\left(S^0\cap S^{-,\lambda'},(\IC^-_{q,Z}|_{S^0\cap \ol{S}{}^{-,\lambda}})|_{S^0\cap S^{-,\lambda'}}\overset{*}\otimes \chi^*(\on{exp})\right)$$
vanishes for $i=1$ and $i=2$ and $\lambda'\neq \lambda$. 

\medskip

First, it follows from Remark \ref{r:!-fibers IC} that 
$$(\IC^-_{q,Z}|_{S^0\cap \ol{S}{}^{-,\lambda}})|_{S^0\cap S^{-,\lambda'}}$$ has the form
$$\omega_{S^0\cap S^{-,\lambda'}}[\langle \lambda',2\check\rho\rangle]\otimes \CE,$$
where $\CE\in \Vect$ lives in cohomological degrees $\geq 2$. 

\medskip

So it suffices to show that
$$H^i\left(S^0\cap S^{-,\lambda'},\omega_{S^0\cap S^{-,\lambda'}}[\langle \lambda',2\check\rho\rangle]\overset{*}\otimes \chi^*(\on{exp})\right)$$
vanishes for $i\leq 0$.  However, this vanishing is the expression of the fact that $\Omega^{\on{DK}}_{q,\Whit}$ is $\geq 0$ in the 
perverse t-structure, see Remark \ref{r:fiber Whit DK}.
  
\end{proof} 

\sssec{}  \label{sss:need to prove subtop}

Thus, the assertion of \thmref{t:main 2} is equivalent to the vanishing of 
\begin{equation} \label{e:subtop open}
H^i\left(S^0\cap S^{-,\lambda},\omega_{S^0\cap S^{-,\lambda}}[\langle \lambda,2\check\rho\rangle]\overset{*}\otimes \chi^*(\on{exp})\right) 
\in \Vect_{\CG^\lambda_{q,x}}
\end{equation} 
for $i=1$ when $\lambda$ is not a negative simple coroot. 
(Along the way we will also see that it vanishes for $i=0$ for all $\lambda\in \Lambda^{\on{neg}}-0$ for a much simpler reason.) 

\medskip

In formula \eqref{e:subtop open} we view the *-restriction of $\chi^*(\on{exp})$ to $S^0\cap S^{-,\lambda}$ as a 
$\CG^\lambda_{q,x}$-twisted sheaf via the identifications
\begin{equation} \label{e:two gerbe triv}
(\CG^G_q)_0|_{S^0\cap S^{-,\lambda}} \simeq \CG^G_q|_{S^0\cap S^{-,\lambda}}\simeq \CG^\lambda_q|_{S^0\cap S^{-,\lambda}},
\end{equation} 
where:

\begin{itemize}

\item $(\CG^G_q)_0$ is the trivial gerbe;

\medskip

\item The identification $(\CG^G_q)_0|_{S^0\cap S^{-,\lambda}} \simeq \CG^G_q|_{S^0\cap S^{-,\lambda}}$ comes by restriction along
$S^0\cap S^{-,\lambda}\to S^0$ from the trivialization of $\CG^G_q|_{S^0}$;

\item The identification $\CG^G_q|_{S^0\cap S^{-,\lambda}}\simeq \CG^\lambda_q|_{S^0\cap S^{-,\lambda}}$ 
comes by restriction along $S^0\cap S^{-,\lambda}\to S^{-,\lambda}$
from the identification $\CG^G_q|_{S^{-,\lambda}}\simeq \CG^\lambda_q|_{S^{-,\lambda}}$.

\end{itemize} 

\sssec{}  \label{sss:Psi lambda}

Let us choose a trivialization of the fiber of $\CG^\lambda_{q,x}$ of $\CG^\lambda_q$ at $x\in X$. We obtain that \eqref{e:two gerbe triv}
gives rise to a local system on $S^0\cap S^{-,\lambda}$, to be denoted $\Psi_{q,\lambda}$, which is well-defined up to a $\sfe^\times$-torsor. Moreover,
$\Psi_{q,\lambda}$ is twisted $T$-equivariant against a Kummer sheaf on $T$ corresponding to the homomorphism
\begin{equation} \label{e:Kummer lambda}
\Lambda\to \fZ,\quad \mu\mapsto b(\lambda,\mu).
\end{equation} 
We will
give a more explicit description of $\Psi_{q,\lambda}$ in \secref{sss:det funct}. 

\medskip

We obtain that \eqref{e:subtop open}, shifted cohomologically by $-[\langle \lambda,2\check\rho\rangle]$, 
viewed as a plain vector space (due to the chosen trivialization of $\CG^\lambda_{q,x}$), 
identifies with
\begin{equation} \label{e:subtop open as plain}
H^i\left(S^0\cap S^{-,\lambda},\omega_{S^0\cap S^{-,\lambda}}\overset{*}\otimes \Psi_{q,\lambda} \overset{*}\otimes \chi^*(\on{exp})\right).
\end{equation} 

Thus, \thmref{t:main 2} is equivalent to the vanishing of \eqref{e:subtop open as plain} in degree
$$i=1+\langle \lambda,2\check\rho\rangle$$ 
(and also $i=\langle \lambda,2\check\rho\rangle$). 

\sssec{}

Up to replacing $\on{exp}$ by its inverse, the cohomology in \eqref{e:subtop open as plain} is dual to 
\begin{equation} \label{e:subtop open as plain c}
H^i_c\left(S^0\cap S^{-,\lambda},\Psi_{q,\lambda} \otimes \chi^*(\on{exp})\right).
\end{equation} 

Thus, we obtain that \thmref{t:main 2} is equivalent to the following one:

\begin{thm} \label{t:main 3}
For $q$ that avoids small torsion and $\lambda$ not a negative simple coroot, 
the cohomology \eqref{e:subtop open as plain c} vanishes in (the \emph{sub-top}) degree 
$$i=-\langle \lambda,2\check\rho\rangle-1.$$
\end{thm}

We will also see that that the cohomology in \eqref{e:subtop open as plain c} vanishes in (the \emph{top}) degree 
$i=-\langle \lambda,2\check\rho\rangle$ for all $\lambda\in \Lambda^{\on{neg}}-0$.

\ssec{The non-torsion valued case} 

\sssec{}  

In this subsection we will prove the following assertion:

\begin{thm} \label{t:Whit non-root}
Let $q$ be non-torsion-valued. Then:

\smallskip

\noindent{\em(a)}
The maps
$$\Omega^{\on{DK}}_{q,\Whit}\to \Omega^{\on{sml}}_{q,\Whit}\to \Omega^{\on{Lus}}_{q,\Whit}$$
are isomorphisms.

\smallskip

\noindent{\em(b)}
The isomorphism $ \Omega^{\on{sml}}_{q,\Whit}|_{\oConf}\simeq \oOmega$ extends to an isomorphism
$$\Omega^{\on{sml}}_{q,\Whit}\simeq \Omega^{\on{sml}}_q.$$
\end{thm} 

\sssec{}

Note that point (a) of \thmref{t:Whit non-root} follows from \propref{p:IC Zast irrat}. So, the essence of the theorem is point (b), which is a particular
case of \thmref{t:main 2}.

\medskip

As we have just seen, \thmref{t:main 2} is equivalent to \thmref{t:main 3}. Thus, we claim that \thmref{t:main 3} holds when $q$ is
non-torsion valued. 

\medskip

The required cohomological estimate is performed in \secref{ss:indict}, see Remark \ref{r:non-tors indict}. 
The same calculation is also performed in \cite[Sects. 6.3-6.5]{Ga3}. 

\qed[\thmref{t:Whit non-root}]

\ssec{Adding the formal parameter}  \label{ss:formal param Whit}

The goal of this subsection is to show that Theorems \ref{t:main 1} and \ref{t:main 2} are logically equivalent.

\sssec{}

Following \secref{ss:with hbar}, we can introduce a version of $\on{Gauss}_{q,*}$ over $\sfe\hbart$, to be denoted
$$\on{Gauss}_{q_{\hbart},*}\in \Shv_{\CG^\Conf_{q_\hbart}}(\Gr^{\omega^\rho}_{G,\Conf}).$$

Set
$$\Omega^{\on{DK}}_{q_{\hbart},\Whit}:=\ol\pi_*(\on{Gauss}_{q_{\hbart},*})\in  \Shv_{\CG^\Conf_{q_\hbart}}(\Conf).$$

\medskip

As in \propref{p:Whit on open}, we have:

\begin{prop} \label{p:Whit on open Taylor}
There exists a canonical isomorphism
\begin{equation} \label{e:Omega Whit open Taylor}
\Omega^{\on{DK}}_{q_{\hbart},\Whit}|_{\oConf}\simeq \oOmega_{q_{\hbart}}.
\end{equation}
\end{prop}

\sssec{}  

Let 
$$\Omega_{q_{\hbarl},\Whit}:=\ol\pi_*(\on{Gauss}_{q_{\hbarl},*})\in  \Shv_{\CG^\Conf_{q_{\hbarl}}}(\Conf)$$
be the localization of $\Omega^{\on{DK}}_{q_{\hbart},\Whit}$. 

\medskip

As in \thmref{t:Whit non-root} we obtain that the isomorphism
$$\Omega_{q_{\hbarl},\Whit}|_{\oConf}\simeq \oOmega_{q_\hbarl},$$
induced by the isomorphism of \propref{p:Whit on open Taylor}, 
extends to an isomorphism
\begin{equation} \label{e:Omega Whit Laurent}
\Omega_{q_{\hbarl},\Whit}\simeq \Omega_{q_\hbarl}.
\end{equation}

\sssec{}

By construction
$$\Omega^{\on{DK}}_{q_{\hbart},\Whit}\underset{\sfe\hbart}\otimes \sfe\simeq \Omega^{\on{DK}}_{q,\Whit};$$
in particular, $\Omega^{\on{DK}}_{q_{\hbart},\Whit}$ is cohomologically $\leq 0$. 

\medskip

From here and the isomorphism \eqref{e:Omega Whit Laurent}, we obtain that the isomorphism
\eqref{e:Omega Whit open Taylor} extends uniquely to a map
\begin{equation} \label{e:Omega Whit Taylor}
\Omega^{\on{DK}}_{q_{\hbart},\Whit}\to \Omega^{\on{DK}}_{q_{\hbart}}
\end{equation}

We claim:

\begin{prop} \label{p:1-2 equiv}
For a given $q$, the following assertions are equivalent:

\smallskip

\noindent{\em(i)} The map \eqref{e:Omega Whit Taylor} is an isomorphism;

\smallskip

\noindent{\em(ii)} The isomorphism of \propref{p:Whit on open} extends to an isomorphism
$$\Omega^{\on{DK}}_{q,\Whit}\simeq \Omega^{\on{DK}}_{q}.$$

\smallskip

\noindent{\em(iii)} The isomorphism of \propref{p:Whit on open} extends to an isomorphism
$$\Omega^{\on{sml}}_{q,\Whit}\simeq \Omega^{\on{sml}}_{q}.$$

\end{prop}

\begin{proof}

Assertion (i) implies assertion (ii) by reduction modulo $\hbar$. 

\medskip

Before we prove the other equivalences, let us note that by Remark \ref{r:fiber Whit DK} and \lemref{l:remove other strata}, point (iii) is 
equivalent to the fact that the map 
\begin{equation} \label{e:Omega extTaylor}
\Omega^{\on{DK}}_{q,\Whit} \to H^0\left((\jmath^\lambda)_*\circ (\jmath^\lambda)^*(\Omega^{\on{DK}}_{q,\Whit})\right)
\end{equation} 
is an injection of perverse sheaves for any $\lambda$ that is not a negative simple coroot. 

\medskip

Let us now assume that (ii) holds, and let us deduce (iii). Indeed, by the above, this follows from the fact that 
$$\Omega^{\on{DK}}_{q}\to H^0\left((\jmath_\lambda)_*\circ (\jmath_\lambda)^*(\Omega^{\on{DK}}_q)\right)$$
\emph{is} an injection of perverse sheaves. 

\medskip

Finally, let us assume (iii) and deduce (i). By induction and factorization, we can assume that
$$(\jmath_\lambda)^*(\Omega^{\on{DK}}_{q_{\hbart},\Whit})\to (\jmath_\lambda)^*(\Omega^{\on{DK}}_{q_{\hbart}})$$
is an isomorphism. Furthermore, by \eqref{e:Omega Whit Laurent}, the map \eqref{e:Omega Whit Taylor} becomes an
isomorphism after inverting $\hbar$. Let $\CF$ denote the cone of \eqref{e:Omega Whit Taylor}. By the above, $\CF$ is
$\hbar$-torsion and concentrated in (perverse) degrees $\leq 0$. Hence if $\CF\neq 0$, the object
$\CF\underset{\sfe\hbart}\otimes \sfe$ would have non-trivial cohomology in (perverse) degrees $<0$.  This would mean
that the map 
$$\Omega^{\on{DK}}_{q,\Whit}\to \Omega^{\on{DK}}_{q}$$
is \emph{not} an injection of perverse sheaves. However, from the commutative diagram
$$
\CD
\Omega^{\on{DK}}_{q,\Whit}  @>>>  \Omega^{\on{DK}}_{q}  \\
@VVV   @VVV   \\
(\jmath_\lambda)_*\circ (\jmath_\lambda)^*(\Omega^{\on{DK}}_{q,\Whit})  @>{\sim}>>  (\jmath_\lambda)_*\circ (\jmath_\lambda)^*(\Omega^{\on{DK}}_{q}) 
\endCD
$$
we would obtain that \eqref{e:Omega extTaylor} is not an injection either, contradicting (iii). 

\end{proof}

\sssec{}  \label{sss:proof of gen via Whit}

Note that the equivalence (ii) $\Leftrightarrow$ (iii) in \propref{p:1-2 equiv} combined with \thmref{t:Whit non-root} implies
the assertion of \thmref{t:DK vs small gen}. 

\ssec{Quantum Frobenius for $\Omega^{\on{Lus}}_{q,\Whit}$}  \label{ss:quant Frob Whit}

In this subsection we will state \thmref{t:Frob for Whit} and using it will deduce the assertions of Theorems 
\ref{t:main 1} and \ref{t:main 2} over a ground field of characteristic $0$. 
 
\sssec{}  

Let $q$ be torsion-valued. We claim:

\begin{thm} \label{t:Frob for Whit}
There exists a canonically defined action of $\Omega^{\sharp,\on{cl}}$ on
$\Omega^{\on{Lus}}_{q,\Whit}$ such that the map
$$\Omega_{q,\Whit}^{\on{Lus}}\to \Omega_{q,\Whit}^{\on{sml}}$$
is compatible with the trivial $\Omega^{\sharp,\on{cl}}$-action on $\Omega_{q,\Whit}^{\on{sml}}$, and 
gives rise to an identification
\begin{equation} \label{e:coinv Whit}
\left((\on{Cone}(\Omega^{\sharp,\on{cl}})_X\to (\Omega_{q,\Whit}^{\on{Lus}})_X)\right)
\underset{(\Omega^{\sharp,\on{cl}})_X}\otimes \omega_X\to \Omega_{q,\Whit}^{\on{sml}}.
\end{equation}
\end{thm} 

This theorem will be proved in \secref{ss:Whit quant Frob}. 

\begin{rem} \label{r:Whit good}
The assertion of \thmref{t:Frob for Whit} holds for any torsion-valued $q$ (i.e., we do not need the assumption 
that $q$ avoid small torsion). 

\medskip

So, quantum Frobenius exists for $\Omega^{\on{Lus}}_{q,\Whit}$ for all torsion-valued $q$ (including the case when $q$ is degenerate),
but for it to exist for $\Omega^{\on{Lus}}_q$ (which we already know to be isomorphic to $\Omega^{\on{Lus}}_{q,\on{Quant}}$)
we need the assumption that $q$ avoid small torsion, see Remark \ref{r:Frob quant}. 

\medskip

This justifies the point of view that $\Omega^{\on{Lus}}_{q,\Whit}$ is, in general, a more relevant object than 
$\Omega^{\on{Lus}}_{q,\on{Quant}}\simeq \Omega^{\on{Lus}}_q$. 

\end{rem} 

\sssec{}  \label{sss:proof of Lus and sml via Whit}

Note that once Theorems \ref{t:main 1} and \ref{t:main 2} are proved, the assertion of \thmref{t:Lus and sml}
would follow from that of \thmref{t:Frob for Whit} as in Sects. \ref{sss:Frob prop 1}-\ref{sss:Frob prop 3}. 

\sssec{}

Let us assume \thmref{t:Frob for Whit} and deduce \thmref{t:main 2} over a ground field
of characteristic $0$. (This would also imply \thmref{t:main 1} by \propref{p:1-2 equiv}). 

\medskip

By induction and factorization, we can assume that the map
\begin{equation} \label{e:DK Whit vs abs}
(\jmath_\lambda)^*(\Omega^{\on{sml}}_{q,\Whit})\to (\jmath_\lambda)^*(\Omega^{\on{sml}}_{q}),
\end{equation}
is an isomorphism. We wish to show that this isomorphism extends across the main diagonal 
$$X\overset{\Delta_\lambda}\hookrightarrow X^\lambda.$$

\medskip

A priori, $\Omega^{\on{sml}}_{q,\Whit}|_{X^\lambda}$ has a 3-step filtration 
$$0=F_0\subset F_1\subset F_2\subset F_3=\Omega^{\on{sml}}_{q,\Whit}|_{X^\lambda}$$
with 
$$F_2/F_1\simeq \Omega^{\on{sml}}_{q}|_{X^\lambda}$$
and $F_1$ and $F_3/F_2$ supported on the main diagonal. We wish to show that $F_1=F_3/F_2=0$.
This is equivalent to showing that the (signed) Jordan-Holder contents of
$$\Delta_\lambda^!(\Omega^{\on{sml}}_{q,\Whit}) \text{ and } \Delta_\lambda^!(\Omega^{\on{sml}}_{q})$$
are equal.  

\medskip 

Recall that the map \eqref{e:Omega Whit Taylor} becomes an isomorphism after inverting $\hbar$. 
Hence, the same is true for the map
$$\Delta_\lambda^*(\Omega^{\on{DK}}_{q_\hbart,\Whit})\to \Delta_\lambda^*(\Omega^{\on{DK}}_{q_\hbart,\Whit}).$$
From this it follows that the (signed) Jordan-Holder contents of
$$\Delta_\lambda^*(\Omega^{\on{DK}}_{q,\Whit}) \text{ and } \Delta_\lambda^*(\Omega^{\on{DK}}_{q})$$
are equal. By Verdier duality, we obtain that the (signed) Jordan-Holder contents of
$$\Delta_\lambda^!(\Omega^{\on{Lus}}_{q,\Whit}) \text{ and } \Delta_\lambda^!(\Omega^{\on{Lus}}_{q})$$
are equal. 

\medskip

This implies that the (signed) Jordan-Holder contents of 
$$\left((\on{Cone}(\Omega^{\sharp,\on{cl}})_X\to (\Omega_{q,\Whit}^{\on{Lus}})_X)\right)
\underset{(\Omega^{\sharp,\on{cl}})_X}\otimes \omega_X$$
and
$$\left((\on{Cone}(\Omega^{\sharp,\on{cl}})_X\to (\Omega_{q}^{\on{Lus}})_X)\right)
\underset{(\Omega^{\sharp,\on{cl}})_X}\otimes \omega_X$$
are also equal. 

\medskip

Applying \thmref{t:Frob for Whit} and \thmref{t:Lus and sml}(f), which has been proved over a ground field of characteristic $0$,
we obtain that the Jordan-Holder contents of 
$$\Delta_\lambda^!(\Omega^{\on{sml}}_{q,\Whit}) \text{ and } \Delta_\lambda^!(\Omega^{\on{sml}}_{q})$$
are also equal, as required. 

\qed[\thmref{t:main 1}]

\ssec{Geometric construction of the quantum Frobenius}  \label{ss:Whit quant Frob}

\sssec{}

Recall the maps 
$$\bi^\lambda:X^\lambda\times \CZ^-\to \CZ^-,$$
see \eqref{e:act on Zast}.

\medskip

They combine to an action of the semi-group $\Conf$ on $\CZ^-$, compatible with the projection $\CZ^-\to \Conf$
and the section $\fs:\Conf\to \CZ^-$. 

\medskip

Pre-composing with
$\Conf^\sharp\to \Conf$, we obtain an action of $\Conf^\sharp$ on $\CZ^-$.  This action
induces a monoidal action of $\Shv(\Conf^\sharp)$ on $\Shv_{\CG^{\on{ratio}}_q}(\CZ^-)$. Hence, given an algebra
in $\Shv(\Conf^\sharp)$ (with respect to convolution), we can talk about objects in $\Shv_{\CG^{\on{ratio}}_q}(\CZ^-)$
being modules over this algebra.

\medskip

The following result is a metaplectic extension of \cite[Theorems 4.2 and 6.6]{BG2}:

\begin{thm} \label{t:Frob for semiinf}  \hfill

\smallskip 

\noindent{\em(a)}
There exists a canonically defined action of $\Omega^{\sharp,\on{cl}}$ on $\nabla^-_{q,Z}$.

\smallskip 

\noindent{\em(b)}
The map $\nabla^-_{q,Z}\to \IC^-_{q,Z}$ induces an isomorphism
$$\on{Bar}(\Omega^{\sharp,\on{cl}},\nabla^-_{q,Z})\simeq \IC^-_{q,Z}.$$
\end{thm}

As a formal corollary, we obtain:

\begin{cor}   \label{c:Frob for Gauss}
There exists a canonically defined action of $\Omega^{\sharp,\on{cl}}$ on $\on{Gauss}^-_{q,!}$ so that the map
$$\on{Gauss}^-_{q,!}\to \on{Gauss}^-_{q,!*}$$
induces an isomorphism
$$\on{Bar}(\Omega^{\sharp,\on{cl}},\on{Gauss}^-_{q,!})\simeq \on{Gauss}^-_{q,!*}.$$
\end{cor}

\sssec{Proof of \thmref{t:Frob for Whit}}

The proof is obtained by applying $\pi^-_*$ to the isomorphism of \corref{c:Frob for Gauss},
see Remark \ref{r:bar}.

\qed

\begin{rem}  \label{r:!-fibers IC}

As in \cite{BG2}, \thmref{t:Frob for semiinf} is closely related the following statement, which we have used 
in the proof of \lemref{l:remove other strata}, and which is a metaplectic extension of
\cite[Theorem 4.5]{BFGM} (a full proof in the metaplectic case is given in \cite[Theorem 4.1]{Lys1}):

\medskip

Consider the !-restriction of $\IC^-_{q,Z}$
$$\oCZ\simeq \{x\} \times \oCZ \overset{\iota_\lambda\times \on{id}}\longrightarrow
X^\lambda\times \oCZ\subset X^\lambda\times \CZ^- \overset{\bi^\lambda}\longrightarrow \CZ^-.$$
Then this restriction vanishes unless $\lambda\in \Lambda^\sharp$, and in the latter case identifies
canonically with 
$$\underset{n>0}\oplus\, \Sym^n((\fn^{\sharp})^\vee)(\lambda)[-2n]\otimes \IC_{q,\oCZ}.$$ 

\end{rem}

\section{Geometry and combinatorics of \thmref{t:main 3}}  \label{s:subtop}

\ssec{The goal}

\sssec{}

With no restriction of generality, in this section we will assume that our root system is simple. Write
$$q=\zeta\cdot q^{\on{min}}_\BZ,$$
where $q^{\on{min}}_Z$ is the minimal integer-valued quadratic form on $\Lambda$ for $\zeta\in \fZ$. 

\medskip

We will assume that $q$ is non-degenerate. This means that $\on{ord}(\zeta)$ is not divisible by 
$$d:=\frac{q^{\on{min}}_\BZ(\alpha_l)}{q^{\on{min}}_\BZ(\alpha_s)},$$
where $\alpha_l$ and $\alpha_s$ are the long and the short roots, respectively.

\medskip

In this section we will \emph{not} be assuming that $q$ avoids small torsion. 

\medskip

We will analyze what \thmref{t:main 3} says in geometric terms for a given value of $\on{ord}(\zeta)$. 
We will show that its assertion is combinatorial in nature. 
Our concrete goal is to prove the following:

\begin{thm}  \label{t:indep of char}
The assertion of \thmref{t:main 3} for a given value of $\on{ord}(\zeta)$
over a ground field of characteristic $0$ implies the assertion for the same value of $\on{ord}(\zeta)$
over any ground field. 
\end{thm} 

\sssec{}

Given that we have already proved \thmref{t:main 3} over a ground field of characteristic $0$ (for $q$ that avoids small torsion), 
we obtain that \thmref{t:indep of char} implies \thmref{t:main 3} over any ground field (also, for $q$ that avoids small torsion). 

\medskip

As we have seen in \secref{ss:subtop},
this in turn implies \thmref{t:main 2}, and further by \secref{ss:formal param Whit} also \thmref{t:main 1}. 

\ssec{Scrutinizing irreducible components}

\sssec{}

Recall that \thmref{t:main 3} says that (for $\lambda$ not a negative simple root) the cohomology 
\begin{equation} \label{e:subtop open as plain c again}
H^i_c\left(S^0\cap S^{-,\lambda},\Psi_{q,\lambda} \overset{*}\otimes \chi^*(\on{exp})\right)
\end{equation} 
vanishes in (the sub-top) degree 
$$i=-\langle \lambda,2\check\rho\rangle-1.$$

\medskip

Recall that $\dim(S^0\cap S^{-,\lambda})=-\langle \lambda,\check\rho\rangle$.  Hence, for (the top) degree $i=-\langle \lambda,2\check\rho\rangle$ and 
(the sub-top) degree $i=-\langle \lambda,2\check\rho\rangle-1$, the cohomology in \eqref{e:subtop open as plain c again}
receives a \emph{surjective map} from the direct sum of 
\begin{equation} \label{e:subtop open as plain c irred}
H^i_c\left(Z,\Psi_{q,\lambda} \overset{*}\otimes \chi^*(\on{exp})\right),
\end{equation} 
where $Z$ runs over the set of the irreducible components of  $S^0\cap S^{-,\lambda}$ of the (top) dimension 
$-\langle \lambda,\check\rho\rangle$. (Note, however, that $S^0\cap S^{-,\lambda}$ is known to be equidimensional.) 

\medskip

We will now analyze which irreducible components $Z$ may have a potentially non-vanishing cohomology \eqref{e:subtop open as plain c irred}
in degree $-\langle \lambda,2\check\rho\rangle-1$. 

\medskip

Along the way we will see that
this cohomology automatically vanishes in (the top) degree $-\langle \lambda,2\check\rho\rangle$
(for all $\lambda\in \Lambda^{\on{neg}}$). 

\sssec{}

Note that the map $\chi:S^0\to \BG_a$ naturally factors as
$$S^0 \overset{\chi_I}\to \BG_a^I\to \BG_a,$$
where:

\begin{itemize}

\item $I$ is the set of vertices of the Dynkin diagram, and $\BG_a^I$ is identified with $N/[N,N]$;

\item $\BG_a^I\to \BG_a$ is the sum map.

\end{itemize} 

The map $\chi_I$ is equivariant with respect to the $T$-action on $S^0\cap S^{-,\lambda}$ and the adjoint action of $T$ on 
$\BG_a^I\simeq N/[N,N]$. Hence, for every irreducible component $Z$ of $S^0\cap S^{-,\lambda}$ there exists a subset $I_Z\subset I$,
such that the map $\chi|_W$ is a dominant map to $\BG_a^{I_Z}\subset \BG_a^I$.

\medskip

Moreover, it follows from \cite[Proposition 7.1.7]{FGV} that as long as $\lambda$ is non-zero, $I_Z\neq \emptyset$. 

\sssec{}

We will consider two cases:

\medskip

\noindent{(i)} $|I_Z|\geq 2$; 

\medskip

\noindent{(ii)} $I_Z$ is a singleton, i.e., $\{i\}$ for one vertex of the Dynkin diagram;

\medskip

We will now show that in case (i), the sub-top cohomology of \eqref{e:subtop open as plain c irred} vanishes. 
We declare an irreducible components in case (ii) as \emph{under scrutiny}, and we will analyze them further in the next
subsection. 

\sssec{}  \label{sss:case 2,i}

We interpret
$$\on{C}^\cdot_c\left(Z,\Psi_{q,\lambda} \overset{*}\otimes \chi^*(\on{exp})\right)$$
as the *-fiber of the (un-normalized) Fourier transform of
\begin{equation} \label{e:chi dir im}
(\chi_I|_Z)_!(\Psi_{q,\lambda})
\end{equation} 
at the point of the dual vector space of $\BG_a^I$ corresponding to the sum map
$$\BG_a^I\to \BG_a.$$

\medskip

The object \eqref{e:chi dir im} in $\Shv(\BG_a^I)$ is twisted $T$-equivariant. Hence so is its Fourier transform.
In particular, this Fourier transform is lisse on the subset consisting of non-degenerate characters.  

\medskip

Since Fourier transform maps $\Perv(\BG_a^I)$ to
$\Perv(\BG_a^I)[-|I|]$, it suffices to show that the object \eqref{e:chi dir im} lives in (perverse) cohomological
degrees $\leq -\langle \lambda,2\check\rho\rangle-1$, and that the inequality is strict if $|I_Z|\geq 2$. 

\medskip

This estimate would also imply that the top cohomology of \eqref{e:subtop open as plain c again} vanishes
for all $\lambda\in \Lambda^{\on{neg}}-0$. 

\sssec{}  \label{sss:case 2,ii}

Again, due to the twisted $T$-equivariance, to prove the required cohomological estimate it suffices to show that 
the *-fiber of \eqref{e:chi dir im} at the \emph{generic} point of each $\BA^{I'}$ (for a subset $I'\subset I$) lives in degrees 
$\leq -\langle \lambda,2\check\rho\rangle-|I'|-1$, and that the inequality is strict if $|I_Z|\geq 2$. 

\medskip

A priori, the *-fiber of \eqref{e:chi dir im} at the generic point of $\BA^{I'}$ lives in degrees
$$\leq 2(\dim(F)),$$
where $F$ is the fiber of $\chi_I|_Z$ over this point. Thus, we need to show that
\begin{equation} \label{e:fiber inequality}
2(\dim(F))\leq -\langle \lambda,2\check\rho\rangle-|I'|-1,
\end{equation}
and the inequality is strict if $|I_Z|\geq 2$. 

\medskip

We have
$$\dim(F)\leq \dim(Z)-|I'|=-\langle \lambda,\check\rho\rangle-|I'|.$$
with the equality achieved only if $I'=I_Z$. This implies the inequality in \eqref{e:fiber inequality}. 

\ssec{The suspects}

Let $Z$ be an irreducible component of the intersection $S^0\cap S^{-,\lambda}$, and let $j$ be a vertex of the Dynkin diagram.
We will recall, following \cite{BG3}, a recipe that attaches to the pair $(Z,j)$ a non-negative integer $\phi_j(Z)$.

\medskip

The main conclusion of this subsection will be that if $Z$ is an irreducible component of $S^0\cap S^{-,\lambda}$ under scrutiny
with $I_Z=\{i\}$ and $\phi_i(Z)\geq 2$, then sub-the top cohomology in \eqref{e:subtop open as plain c irred} still vanishes. 

\medskip

We will declare the irreducible components $Z$ under scrutiny for which $\phi_i(Z)=1$ as \emph{suspicious}. 

\sssec{}

Let $P_j\subset G$ be the standard sub-minimal parabolic corresponding to $j$. Consider the corresponding diagram
$$
\CD
\Gr_{P_j}  @>{\sfp_i}>>  \Gr_G  \\
@V{\sfq_j}VV   \\
\Gr_{M_j}.
\endCD
$$

Consider the ind-scheme
$$\Gr_{P_j}\underset{\Gr_G}\times S^{-,\lambda}$$ along with its projection to $\Gr_{M_j}$. Note that 
$\Gr_{P_j}\underset{\Gr_G}\times S^{-,\lambda}$ carries an action of $\fL(N^-_j)$, where $N^-_j\simeq \BG_a$ 
is the (negative) maximal unipotent in $M_j$. 

\medskip

For a given $\mu\in \Lambda$, let $S^{-,\mu}_j\subset \Gr_{M_j}$ denote the corresponding orbit of $\fL(N^-_j)$.
The above action defines an isomorphism 
\begin{equation} \label{e:prod decomp}
(\sfq_j)^{-1}(S^{-,\mu}_j) \underset{\Gr_G}\times S^{-,\lambda}
\simeq S^{-,\mu}_j \times ((\fL(N^j)\cdot t^\mu)\cap S^{-,\lambda}),
\end{equation} 
where $N^j$ denotes the unipotent radical of $P_j$. 

\medskip

In terms of the isomorphism \eqref{e:prod decomp}, the subset
$$(\sfq_j)^{-1}(S^{-,\mu}_j) \underset{\Gr_G}\times (S^0\cap S^{-,\lambda}) \subset (\sfq_j)^{-1}(S^{-,\mu}_j) \underset{\Gr_G}\times S^{-,\lambda}$$
corresponds to
$$(S_j^0\cap S^{-,\mu}_j) \times ((\fL(N^j)\cdot t^\mu)\cap S^{-,\lambda}).$$

Thus, we obtain an isomorphism 
\begin{equation} \label{e:prod decomp MV}
(\sfq_j)^{-1}(S^{-,\mu}_j) \underset{\Gr_G}\times (S^0\cap S^{-,\lambda})  \simeq (S_j^0\cap S^{-,\mu}_j) \times ((\fL(N^j)\cdot t^\mu)\cap S^{-,\lambda}).
\end{equation} 

\sssec{}

Note that the intersection $S_j^0\cap S^{-,\mu}_j$ is taking place in neutral connected component of
the affine Grassmannian of $\Gr_{M_j}$, which is a reductive group of semi-simple rank $1$, so if
the above intersection is non-empty, we have
\begin{equation} \label{e:integer m}
\mu=m\cdot (-\alpha_j)
\end{equation} 
for a non-negative integer $m$. 

\medskip

Furthermore, we have:
$$S_j^0\cap S^{-,\mu}_j=
\begin{cases}
&\on{pt} \text{ if } m=0,\\
&\BG_m \text{ if } m=1,\\
&\BG_m  \times \BG_a^{m-1} \text{ if } m\geq 2.
\end{cases}$$

\sssec{}  \label{sss:phi function}

Let $Z$ be an irreducible component of $S^0\cap S^{-,\lambda}$ of dimension $-\langle \lambda,\check\rho\rangle$. 
One shows (see \cite[Proposition 3.1]{BG3}) that there exists a unique element $\mu\in \Lambda$,
such that the intersection
$$Z_j:=Z\cap (\sfq_j)^{-1}(S^{-,\mu}_j)$$
is dense in $Z$. 

\medskip

We set $\phi_j(Z)$ to be the corresponding integer $m$ from \eqref{e:integer m}. 

\sssec{}

Let $\chi_j$ denote the composite
$$S^0 \overset{\chi_I}\longrightarrow \BG_a^I \to \BG_a,$$
where the last arrow is the projection on the $j$-th coordinate. 

\medskip

For an irreducible component $Z$, let $Z_j$ be as above. 
We obtain that the restriction $\chi_j|_{Z_j}$ can be described as follows:

\begin{itemize}

\item It is the zero map if $\phi_j(Z)=0$;

\medskip

\item It is is the composite 
$$Z_j \hookrightarrow (S_j^0\cap S^{-,\mu}_j) \times ((\fL(N^j)\cdot t^\mu)\cap S^{-,\lambda})\to S_j^0\cap S^{-,\mu}_j\simeq \BG_m \hookrightarrow \BG_a$$
if $\phi_j(Z)=1$;

\medskip

\item It is is the composite 
$$Z_j \hookrightarrow (S_j^0\cap S^{-,\mu}_j) \times ((\fL(N^j)\cdot t^\mu)\cap S^{-,\lambda})\to S_j^0\cap S^{-,\mu}_j\simeq 
\BG_m\times \BG_a^{m-1}\to \BG^{m-1}_a\to \BG_a,$$
where the last arrow is the projection on the first $\BG_a$ factor.

\end{itemize}

\sssec{}

Let $Z$ be a suspicious irreducible component of $S^0\cap S^{-,\lambda}$ and let $I_Z=\{i\}$. We obtain that
$\phi_j(Z)=0$ for all $j\neq i$. 

\medskip

We will now show that if $\phi_i(Z)\geq 2$, then the sub-top cohomology in \eqref{e:subtop open as plain c irred} still vanishes. 
Indeed, since $Z_i$ is dense in $Z$, it suffices to show that the sub-top cohomology in 
\begin{equation} \label{e:subtop open as plain c irred i}
\on{C}^\cdot_c\left(Z_i,\Psi_{q,\lambda} \overset{*}\otimes \chi^*(\on{exp})\right)
\end{equation}
vanishes. However, we claim that \eqref{e:subtop open as plain c irred i} vanishes entirely. 

\medskip

Indeed, let us calculate \eqref{e:subtop open as plain c irred i} via the projection formula. We obtain that it identifies with
$$\on{C}^\cdot_c(\BG_a,(\chi_I|_{Z_i})_!(\Psi_{q,\lambda})\overset{*}\otimes \on{exp}).$$
However, we claim that $(\chi_I|_{Z_i})_!(\Psi_{q,\lambda})$ is a \emph{constant} complex on $\BG_a$.

\medskip

Indeed, in \secref{sss:psi lambda} we will see that the restriction of 
the local system $\Psi_{q,\lambda}$ to $$(\sfq_j)^{-1}(S^{-,\mu}_j)\underset{\Gr_G}\times (S^0\cap S^{-,\lambda})$$
is the pullback of a local system along the projection
\begin{multline*} 
(\sfq_j)^{-1}(S^{-,\mu}_j)\underset{\Gr_G}\times (S^0\cap S^{-,\lambda})
\simeq (S_j^0\cap S^{-,\mu}_j) \times ((\fL(N^j)\cdot t^\mu)\cap S^{-,\lambda})
\simeq \\
\simeq (\BG_m\times \BG_a^{m-1}) \times ((\fL(N^j)\cdot t^\mu)\cap S^{-,\lambda})
\to \BG_m\times ((\fL(N^j)\cdot t^\mu)\cap S^{-,\lambda}).
\end{multline*}

\medskip

This implies our assertion. 

\sssec{}

We declare an irreducible component $Z$ under scrutiny with $I_Z=\{i\}$ as \emph{suspicious} if $\phi_i(Z)=1$. From what we have
seen above, only suspicious irreducible components may contribute to the sub-top cohomology in \eqref{e:subtop open as plain c irred}.

\sssec{}

Here is an example of a suspicious component: take $G=SL_3$ with the simple roots $\alpha$ and $\beta$. Take $\lambda=-\alpha-\beta$.
Then both irreducible components of $S^0\cap S^{-,\lambda}$ are suspicious. 

\begin{rem}   \label{r:few suspicious}

As is explained in \cite[Prop. 1.2.4]{Lys2}, for a given $G$ there are at most finitely many $\lambda$, such that
$S^0\cap S^{-,\lambda}$ contains a suspicious irreducible component. 

\medskip

Namely, it is shown in {\it loc.cit.} that if $Z$ is such a component and
$i$ is the corresponding vertex of the Dynkin diagram, then $\varpi_i+\lambda$ must appear as a weight in 
the irreducible $\cG$-representation with highest weight $\varpi_i$.

\end{rem}

\ssec{Interlude: coordinates on the irreducible components}  \label{ss:coordinates} 

In this subsection we will essentially reproduce a construction from \cite[Sect. 4]{BaGa}. 

\sssec{}   

Let $\CL^{\on{min}}$ be the minimal line bundle on $\Gr_G$. It is $T$-equivariant by construction,
with $Z_G$ acting trivially, so it is in fact $T_{\on{ad}}$-equivariant. 

\sssec{}

The action of $T_{\on{ad}}$ on 
the fiber of $\CL^{\on{min}}$ at $t^\lambda\in \Gr_G$ is given by the character 
\begin{equation} \label{e:char T ad}
T_{\on{ad}}\to \BG_m
\end{equation} 
corresponding to 
\begin{equation} \label{e:b ad}
b^{\on{min}}_\BZ(\lambda,-),\quad \Lambda_{\on{ad}}\to \BZ.
\end{equation} 

Here $b_\BZ^{\on{min}}$ is the integer-valued symmetric bilinear form on $\Lambda$, corresponding to the minimal
quadratic form $q^{\on{min}}_\BZ$, and where we note that $b_\BZ^{\on{min}}$ extends to a pairing 
$$\Lambda\otimes \Lambda_{\on{ad}}\to \BZ$$
by the formula
$$b_\BZ^{\on{min}}(\alpha_i,\mu)=q^{\on{min}}_\BZ(\alpha_i)\cdot \langle \mu,\check\alpha_i\rangle.$$

\sssec{}

The restriction $\CL^{\on{min}}|_{S^0}$ admits a unique $\fL(N)$-equivariant trivialization, compactible with
the trivialization of the fiber of $\CL^{\on{min}}$ at $1\in \Gr_G$. This trivialization is $T_{\on{ad}}$-equivariant. 

\medskip

Similarly, a choice of a trivialization of the 
fiber of $\CL^{\on{min}}$ at $t^\lambda\in \Gr_G$ extends uniquely to an $\fL(N^-)$-equivariant trivialization 
of $\CL^{\on{min}}|_{S^{-,\lambda}}$. This trivialization is twisted $T_{\on{ad}}$-equivariant 
against the character \eqref{e:char T ad}. 

\sssec{}  \label{sss:det funct}

We obtain that the restriction $\CL^{\on{min}}|_{S^0\cap S^{-,\lambda}}$ admits two different trivializations
(one is defined canonically, and another up to a multiplicative constant). Their discrepancy is a function
\begin{equation} \label{e:f lambda}
f_\lambda: S^0\cap S^{-,\lambda}\to \BG_m,
\end{equation} 
well-defined up to a multiplicative constant. This function in twisted $T_{\on{ad}}$-equivariant, against the character 
\eqref{e:char T ad}. 

\medskip

We will now show that irreducible components of the intersections $S^0\cap S^{-,\lambda}$ admit
rational coordinates, such that the function \eqref{e:f lambda} is given by monomials (products of powers
of the coordinates), up to a multiplicative constant. 

%
%
%

\sssec{}

Consider the set 
$$\sB(\lambda):=\underset{\lambda'}\sqcup\, \sB(\lambda)_{\lambda'},$$
where $\sB(\lambda)_{\lambda'}$ is the set if irreducible components of all the intersections
$$S^{\lambda'}\cap S^{-,\lambda}.$$

\medskip

We will now recall, following \cite{BG3}, the construction on the set $\sB(\lambda)$ of a structure of
\emph{Kashiwara's crystal}: 

\medskip

First off, the functions $\phi_j$ are constructed by the recipe of \secref{sss:phi function},
with $S^0$ replaced by a general $S^{\lambda'}$.  The functions $\epsilon_j$ are set to take
value $\infty$.

\medskip

The operators $f_j$ are defined as follows. Let $Z$ be an irreducible component of 
$S^{\lambda'}\cap S^{-,\lambda}$ such that $\phi_j(Z)\neq 0$. Let $\mu$ be the corresponding
element of $\Lambda$, see \secref{sss:phi function}.

\medskip 

Then, in terms of the identification
$$Z_j:=
(\sfq_j)^{-1}(S^{-,\mu}_j)\underset{\Gr_G}\times (S^{\lambda'}\cap S^{-,\lambda})
\simeq (S_j^{\lambda'}\cap S^{-,\mu}_j) \times ((\fL(N^j)\cdot t^\mu)\cap S^{-,\lambda}),$$
$Z_j$ corresponds to a unique irreducible component $Z'_j\subset (\fL(N^j)\cdot t^\mu)\cap S^{-,\lambda}$. 

\medskip

We let $f_j(Z)$ be the closure of the irreducible component 
\begin{multline*}
(S_j^{\lambda'-\alpha_j}\cap S^{-,\mu}_j) \times Z'_j\subset 
(S_j^{\lambda'-\alpha_j}\cap S^{-,\mu}_j) \times ((\fL(N^j)\cdot t^\mu)\cap S^{-,\lambda})\simeq  \\
\simeq (\sfq_j)^{-1}(S^{-,\mu}_j)\underset{\Gr_G}\times (S^{\lambda'-\alpha_j}\cap S^{-,\lambda})
\subset S^{\lambda'-\alpha_j}\cap S^{-,\lambda}.
\end{multline*}

\medskip

The operation $e_j$ is uniquely determined by the requirement that $f_j\circ e_j=\on{id}$. 

\sssec{}  \label{sss:coordinates}

Let $Z$ be an irreducible component of some $S^{\lambda'}\cap S^{-,\lambda}$. We will now use the crystal structure 
on $\sB(\lambda)$ to introduce rational coordinates on $Z$. 

\medskip

Choose a string of vertices of the Dynkin diagram inductively as follows. If $\lambda'=\lambda$, we have $Z=\on{pt}$, 
and there is nothing to do. Otherwise, set $\lambda'_1=\lambda'$, $Z_1=Z$ and \emph{choose} $j_1\in I$ such that $\phi_{j_1}(Z_1)\neq 0$. Set
$Z_2=f_j^{\phi_{j_1}}(Z_1)$. Set
\begin{equation} \label{e:lambfa grows}
\lambda'_2=\lambda'_1-\phi_{j_1}(Z_1)\cdot \alpha_{j_1}.
\end{equation}  

\medskip

Now repeat the process with $Z_1$ replaced by $Z_2$. This process will terminate by \eqref{e:lambfa grows}  because 
$$S^{\lambda'}\cap S^{-,\lambda}\neq \emptyset \, \Rightarrow\, \lambda'-\lambda\in \Lambda^{\on{pos}},$$
so for some $n$ we will have $\lambda'_n=\lambda$.

\medskip

Note that by construction, we have a rational isomorphism
\begin{equation} \label{e:induction step}
Z_n\simeq (\BG_m\times \BG_a^{\phi_{j_n}(Z_n)-1})\times Z_{n+1}.
\end{equation} 

\medskip

The above process gives $Z$ rational coordinates. We will denote them by
\begin{equation} \label{e:coord}
x_{j_1,1},....,x_{j_1,\phi_{j_1}(Z_1)},x_{j_2,1},...,x_{j_2,\phi_{j_2}(Z_2)},....
\end{equation} 

\sssec{}

Note that the construction of \secref{sss:det funct} defines a function (up to a multiplicative scalar) on each intersection
$S^{\lambda'}\cap S^{-,\lambda}$; let us denote it by $f_{\lambda',\lambda}$. We claim:

\begin{prop}  \label{p:det}
The restriction of $f_{\lambda',\lambda}$ to a given irreducible
component $Z$ equals (up to a multiplicative scalar) in terms of the coordinates \eqref{e:coord} to
$$\underset{n}\Pi\, (x_{j_n,1})^{\phi_{j_n}(Z_n)\cdot q^{\on{min}}_\BZ(\alpha_{j_n})}.$$
\end{prop} 

\begin{proof}

It follows from the construction that in terms of the rational isomorphism \eqref{e:induction step}, we have
$$f_{\lambda'_n,\lambda}=g_n\cdot f_{\lambda'_{n+1},\lambda},$$
where $g_n$ is an invertible function on 
$$S_{j_n}^{\lambda'_n}\cap S_{j_n}^{-,\lambda'_n-\phi_{j_n}(Z_n)\cdot \alpha_{j_n}}\simeq 
\BG_m\times \BG_a^{\phi_{j_n}(Z_n)-1},$$
equal to the restriction of $f_{\lambda'_n,\lambda'_n-\phi_{j_n}(Z_n)\cdot \alpha_{j_n}}$ along
$$S_{j_n}^{\lambda'_n}\cap S_{j_n}^{-,\lambda'_n-\phi_{j_n}(Z_n)\cdot \alpha_{j_n}}\hookrightarrow
S^{\lambda'_n}\cap S^{-,\lambda'_n-\phi_{j_n}(Z_n)\cdot \alpha_{j_n}}.$$

\medskip

Being invertible, the function $g_n$ equals, up to a multiplicative scalar, to the pullback of
\emph{some} power $m$ of the standard character on the $\BG_m$ factor. Thus, it remains to
show that the power in question equals $\phi_{j_n}(Z_n)\cdot q^{\on{min}}_\BZ(\alpha_{j_n})$. 

\medskip

The function $x_{j_n,1}$ is $\BG_m$-equivariant against the character equal to $\check\alpha_{j_n}$. 

\medskip

The function $g_n$ is twisted $T_{\on{ad}}$-equivariant against the character equal to the ratio
of the characters corresponding to $f_{\lambda'_n,\lambda}$ and $f_{\lambda'_{n+1},\lambda}$,
respectively. Hence, the character in question corresponds to the homomorphism
$$b^{\on{min}}_\BZ(\phi_{j_n}(Z_n)\cdot \alpha_{j_n},-): \Lambda_{\on{ad}}\to \BG_m.$$

Hence, we obtain
$$m\cdot \check\alpha_{j_n}=b^{\on{min}}_\BZ(\phi_{j_n}(Z_n)\cdot \alpha_{j_n},-).$$

Evaluating on the fundamental coweight
$$\varpi_{j_n}:\BG_m\to T_{\on{ad}},$$
we obtain
$$m=\phi_{j_n}(Z_n)\cdot b^{\on{min}}_\BZ(\alpha_{j_n},\varpi_{j_n})=\phi_{j_n}(Z_n)\cdot q^{\on{min}}_\BZ(\alpha_{j_n}).$$

\end{proof}

\ssec{Indictment}  \label{ss:indict}

\sssec{}  \label{sss:when indict}

Let $Z$ be a suspicious irreducible component. By the same analysis as in Sects. \ref{sss:case 2,i}-\ref{sss:case 2,ii}, we obtain that 
the sub-top cohomology in \eqref{e:subtop open as plain c irred} (equivalently, in \eqref{e:subtop open as plain c irred i})
is non-zero if and only if the restriction of the local system $\Psi_{q,\lambda}$ to $Z_i$ is (generically on $Z_i$) 
the pullback of a local system along the map 
\begin{equation} \label{e:chi i}
\chi_i:Z_i\to \BG_m.
\end{equation} 

We will call such $Z$ \emph{indicted}. We will now analyze explicitly what it takes to be indicted. 

\sssec{}

Recall that our $q$ is written as 
$$\zeta\cdot q^{\on{min}}_\BZ,$$
where $q^{\on{min}}_Z$ is the minimal quadratic form on $\Lambda$. 

\medskip

Note that the gerbe $\CG^G_q$ over $\Gr_G$ identifies canonically with $(\CL^{\on{min}})^\zeta$. 

\sssec{}  \label{sss:psi lambda}

We obtain that the local system $\Psi_{q,\lambda}$ on $S^0\cap S^{-,\lambda}$ is the pullback of the Kummer local system $\Psi_\zeta$ on $\BG_m$ 
by means of $f_\lambda$. 

\medskip

In particular, for a subscheme $Z\subset S^0\cap S^{-,\lambda}$, written as $Z\simeq Z'\times \BG_a$ with $Z'$
reduced, the restriction of $\Psi_{q,\lambda}$ to $Z$ is the pullback from the $Z'$ factor. 

\sssec{}

Let $Z$ be a suspicious component. Let us recall the (rational) coordinates on $Z$ constructed in \secref{sss:coordinates}. 
Note that we necessarily have $j_1=i$ and $\phi_{j_1}(Z)=\phi_i(Z)=1$. 
Note also that the map $$\chi_i|_{Z_i}:Z_i\to \BG_m$$ identifies with the \emph{first coordinate function} 
i.e., $x_{j_1,1}$. 

\medskip

Hence, from \propref{p:det} and \secref{sss:when indict}, we obtain:

\begin{cor} \label{c:when indicted}
A suspicious component $Z$ is indicted if and only if for all $n\geq 2$, the integers
$$\phi_{j_n}(Z_n)\cdot q^{\on{min}}_\BZ(\alpha_{j_n})$$
are divisible by $\on{ord}(\zeta)$.
\end{cor}

\begin{rem} \label{r:criminal}
One may wonder whether our indictment is non-empty: i.e., whether assuming that $q$ is non-degenerate
indicted components exist. In fact, they do:

\medskip

Take $G=G_2$; let $\alpha$ be the short simple root and let $\beta$ be the long simple root. 
Take $\lambda=-2\alpha-\beta$. The intersection $S^0\cap S^{-,\lambda}$ has three
irreducible components, among which exactly one is \emph{not} annihilated by $e_\beta$. 

\medskip

Then this component is indicted for $\zeta=-1$.
\end{rem} 

\begin{rem}
Note that \corref{c:when indicted} implies that if $\zeta$ is non-torsion (i.e., $q$ is not a root of unity),
then there are no indicted components.

\medskip

Indeed, in this case $\on{ord}(\zeta)=\infty$, so we would
obtain that $\lambda$ is of the form $m\cdot (-\alpha_i)$, while the condition that $\phi_i(Z)=1$ forces $m=1$. 
I.e., we obtain that $\lambda$ is a negative simple root. 

\end{rem}

\sssec{}

Thus, from now on we will assume that $\zeta$ is torsion (i.e., $q$ is a root of unity). We claim: 

\begin{cor}  \label{c:not in sharp}
Let $\lambda$ be such that $S^0\cap S^{-,\lambda}$ contains an indicted 
irreducible component $Z$; let $i$ be the corresponding element of $I$. Then:

\smallskip

\noindent{\em(a)} $\lambda+\alpha_i\in \Lambda^\sharp$.

\smallskip

\noindent{\em(b)} $\lambda\notin \Lambda^\sharp$.
\end{cor}

\begin{proof}

Let $Z$ be the indicted component. We have
$$\lambda+\alpha_i=-\underset{j\geq 2}\Sigma\, \phi_{j_n}(Z_n)\cdot \alpha_{j_n}.$$

Hence, to prove point (a), we need to show that 
$$\underset{j\geq 2}\Sigma\, \phi_{j_n}(Z_n)\cdot \alpha_{j_n}\in \Lambda^\sharp.$$
For this, it is sufficient to show that for every $j\geq 2$, the integer $\phi_{j_n}(Z_n)$
is divisible by $\on{ord}(q(\alpha_{j_n}))$. However, this follows from \corref{c:when indicted}.

\medskip

To prove point (b), it suffices to show that $\alpha_i\notin \Lambda^\sharp$. But this follows
from the assumption that $q$ is non-degenerate. 

\end{proof}

\begin{rem}  \label{r:non-tors indict}

The two conditions on $\lambda$, namely, that 
$$\lambda+\alpha_i\in \Lambda^\sharp$$ and
the condition from Remark \ref{r:few suspicious}, impose very stringent constraints. 

\medskip

The result of \cite[Theorem 1.1.6]{Lys2} says that no indicted components exist, except
for a very small number of possibilities for the order of $\zeta$. 

\end{rem}

\ssec{Conviction}

Let $Z$ be an indicted irreducible component of $S^0\cap S^{-,\lambda}$. This means that the cohomology 
\begin{equation} \label{e:subtop open as plain c irred again}
H^i_c\left(Z,\Psi_{q,\lambda} \otimes \chi^*(\on{exp})\right),
\end{equation} 
is \emph{non-zero} in degree $i=-\langle \lambda,2\check\rho\rangle-1$. 

\medskip

In this subsection we will proceed to conviction: we will show that if $S^0\cap S^{-,\lambda}$ contains an indicted 
component $Z$, the cohomology
\begin{equation} \label{e:subtop open as plain c again again}
H^i_c\left(S^0\cap S^{-,\lambda},\Psi_{q,\lambda} \otimes \chi^*(\on{exp})\right)
\end{equation} 
is also \emph{non-zero} in degree $i=-\langle \lambda,2\check\rho\rangle-1$, in violation of \thmref{t:main 3}. 

\sssec{}

First, we prove:

\begin{prop} \label{p:off generic}
Assume that $\lambda\notin \Lambda^\sharp$. 
Let $Z'\subset S^0\cap S^{-,\lambda}$ be a closed $T$-stable subscheme of dimension $\leq -\langle \lambda,\check\rho\rangle-1$. 
Then the cohomology 
\begin{equation} \label{e:subtop on small}
H^i_c\left(Z',\Psi_{q,\lambda} \otimes \chi^*(\on{exp})\right)
\end{equation} 
vanishes in degrees $i\geq -\langle \lambda,2\check\rho\rangle-2$.
\end{prop}

\begin{proof}

By the analysis in Sects \ref{sss:case 2,i} and \ref{sss:case 2,ii}, it suffices to consider the case when the the map
$\chi_I$ sends $Z'$ to $0\in \BG_a^I$. However, we claim that in this case 
$$(\chi_I|_{Z'})_!(\Psi_{q,\lambda})=0.$$

Indeed, the assumption on $\lambda$ implies that the Kummer sheaf on $T$ given by \eqref{e:Kummer lambda}
is \emph{non-trivial}, while the point $0\in \BG_a^I$ does not support sheaves that are twisted $T$-equivariant against
a non-trivial character sheaf. 

\end{proof} 

\begin{cor} \label{c:all comps contribute}
Assume that $\lambda\notin \Lambda^\sharp$. Then the map
$$\underset{k}\oplus\, H^i_c\left(Z_k,\Psi_{q,\lambda} \otimes \chi^*(\on{exp})\right)\to 
H^i_c\left(S^0\cap S^{-,\lambda},\Psi_{q,\lambda} \otimes \chi^*(\on{exp})\right)$$
is injective for $i\geq -\langle \lambda,2\check\rho\rangle-1$, where the direct sum is taken
over the set of irreducible components of $S^0\cap S^{-,\lambda}$ of dimension 
$-\langle \lambda,\check\rho\rangle$.
\end{cor}

\sssec{}

Combining Corollaries \ref{c:all comps contribute} and \ref{c:not in sharp}, we obtain:

\begin{cor}   \label{c:conviction}
Let $\lambda$ be such that  $S^0\cap S^{-,\lambda}$ contains an indicted 
irreducible component. Then the cohomology \eqref{e:subtop open as plain c again} is non-zero 
in degree $-\langle \lambda,2\check\rho\rangle-1$. 
\end{cor} 

The latter corollary says that the indicted components are indeed guilty of a crime: they bring about the
failure of \thmref{t:main 3}. 

\ssec{The verdict}

We are finally ready to prove \thmref{t:indep of char}. 

\begin{proof}

Let us be working in both contexts simultaneously: $\ell$-adic sheaves over a field
of positive characteristic or D-modules over a field of characterostic zero. In either case,
write 
$$q=\zeta\cdot q^{\on{min}}_\BZ$$ 
and choose $\zeta$ to be of the same order in both contexts. 

\medskip

By \corref{c:conviction} it suffices to show the presence of an indicted components (for a given 
$\on{ord}(\zeta)$) is independent of which context we are working in. 

\medskip

However, by \corref{c:when indicted}, the existence of indicted components for a given root system
is a property that can be expressed in terms of the the crystal $\sB(\lambda)$. 

\medskip

Hence, our assertion follows from Kashiwara's uniqueness theorem, which asserts that $\sB(\lambda)$ 
is uniquely recovered from the root system, and, in particular, it does not depend on the ground field.

\end{proof}

\appendix

\section{Sheaves with a formal parameter}  \label{s:h bar}

\ssec{Digression: sheaves with a formal parameter}  \label{ss:with hbar}

\sssec{}  \label{sss:intr hbar}

For a fixed integer $n$ we can consider the sheaf theory, denoted $\Shv_n(-)$, obtained from our initial sheaf theory $\Shv(-)$ by tensoring
with the ring $\sfe[\hbar]/\hbar^n$, i.e., it sends a scheme $Y$ to 
$$\Shv_n(Y):=\left(\on{Ind}(\Shv(Y))\underset{\Vect_\sfe}\otimes \sfe[\hbar]/\hbar^n\mod\right)^c,$$
where the superscript $c$ refers to the subcategory of compact objects. 

\medskip

Following \cite[Sect. 2.3]{GaLu}, define the sheaf theory $\Shv_\hbart$ as the limit
$$\Shv_\hbart(Y):=\underset{n}{\on{lim}}\, \Shv(Y)_n,$$
which we can view as taking values in the category of small $\sfe\hbart$-linear categories, i.e., small DG categories
equipped with an action of the monoidal category $(\sfe\hbart\mod)^{\on{perf}}=(\sfe\hbart\mod)^{\on{f.g.}}$. 

\medskip

We have a tautological reduction mod $\hbar$ functor
$$\CF\mapsto \CF\underset{\sfe\hbart}\otimes \sfe, \quad \Shv_\hbart(Y)\to \Shv(Y).$$

\sssec{}

Define the sheaf theory $\Shv_\hbarl(-)$ to be the localization of $\Shv_\hbart(-)$ with respect to $\hbar$, i.e.,
$$\Shv_\hbarl(Y):=\left(\on{Ind}(\Shv_\hbart(Y)) \underset{\sfe\hbart\mod}\otimes \Vect_{\sfe\hbarl}\right)^c.$$

\medskip

The tautological projection functor $\Shv_\hbart(Y)\to \Shv_\hbarl(Y)$ admits a right adjoint \emph{with values in the
ind-completion} $\on{Ind}(\Shv_\hbart(Y))$ of $\Shv_\hbart(Y)$. The composite functor
$$\Shv_\hbart(Y)\to \Shv_\hbarl(Y)\to \on{Ind}(\Shv_\hbart(Y))$$
is the functor of $\hbar$-localization
$$\CF\mapsto \on{colim}\, (\CF \overset{\hbar}\to \CF \overset{\hbar}\to ...).$$

\sssec{}

The standard functors (the !- and *- inverse and direct images) for $\Shv(-)$ induce the corresponding functors for
$\Shv_\hbart(-)$ by passage to the limit. Localizing with respect to $\hbar$, we obtain the corresponding functors 
for $\Shv_\hbarl(-)$.

\ssec{The t-structure}



\sssec{}   \label{sss:t on hbar}

Repeating \cite[Proposition 2.3.6.1]{GaLu}, one proves:

\begin{prop} \label{p:t h} 
For a scheme $Y$, the category $\Shv_\hbart(Y)$ carries a t-structure uniquely characterized by the condition that an object 
$\CF\in \Shv_\hbart(Y)$ is connective, i.e., lies in $(\Shv_\hbart(Y))^{\leq 0}$, if and only if $$\CF\underset{\sfe\hbart}\otimes \sfe \in \Shv(Y)$$ is 
connective (with respect to the perverse t-structure). 
\end{prop}

\begin{rem}
The content of \propref{p:t h}(a) is that for a given $\CF\in \Shv_\hbart(Y)$ there \emph{exists} a fiber sequence
$$\CF_1\to \CF\to \CF_2$$ with $\CF_1\in (\Shv_\hbart(Y))^{\leq 0}$ (for the above definition of $(\Shv_\hbart(Y))^{\leq 0}$)
and $\CF_2\in ((\Shv_\hbart(Y))^{\leq 0})^\perp$.
\end{rem} 

\begin{rem}  \label{r:gen abelian}
The construction of the t-structure from \propref{p:t h} is applicable in a more general context: we can start with any
(small) DG category $\bC$, equipped with a t-structure and such that $\bC^\heartsuit$ is \emph{Noetherian}, and construct
a t-structure on the corresponding category $\bC_\hbart$.
\end{rem} 

\sssec{} \label{sss:hflat}

We will denote by $\Perv_\hbart(Y)$ the heart of the t-structure on $\Shv_\hbart(Y)$

Let us call an object of $$\CF\in \Perv_\hbart(Y)\subset \Shv_\hbart(Y)$$
$\hbar$-flat if 
$$\CF\underset{\sfe\hbart}\otimes \sfe \in \Shv(Y)$$ 
lies in the heart of the t-structure, i.e., lies in $\Perv(Y)$.

\medskip

For $\CF\in \Perv_\hbart(Y)$ denote 
$$\CF/\hbar:=H^0(\CF\underset{\sfe\hbart}\otimes \sfe).$$

Note that if $\CF$ is $\hbar$-flat, we have 
$$\CF/\hbar\simeq \CF\underset{\sfe\hbart}\otimes \sfe.$$

In general, $\CF\underset{\sfe\hbart}\otimes \sfe$ may have a non-trivial $H^{-1}$, which is isomorphic to
$$\CF[t]:=\on{ker}(t:\CF\to \CF).$$

\sssec{}

We claim: 

\begin{prop} \label{p:t h bis}
The abelian category
$\Perv_\hbart(Y)$ is \emph{Noetherian}, i.e., an increasing chain of sub-objects in a given object stabilizes.  
\end{prop} 

\begin{proof} 

We will prove \propref{p:t h bis} in the more general context of Remark \ref{r:gen abelian}. 

\medskip

For an object $\bc\in (\bC_\hbart)^\heartsuit$ denote
$$\bc[\hbar^d]:=\on{ker}(\hbar^d:\bc\to \bc)\in (\bC_\hbart)^\heartsuit$$
and
$$\hbar^d\bc:=\on{Im}(\hbar^d:\bc\to \bc)\in (\bC_\hbart)^\heartsuit.$$

\begin{lem}  \label{l:kernels stabilize}
The sequence of subobjects
$$d\mapsto \bc[\hbar^d]\subset \bc$$
stabilizes. 
\end{lem}

Let us assume this lemma for a moment and proceed with the proof of \propref{p:t h bis}.

\medskip

Let $\bc[\hbar^\infty]$ denote the subobject of $\bc$ equal to the eventual value of $\bc[\hbar^d]$. 

\medskip

For a subobject $\bc_1\subset \bc$ we define its saturation $\wt\bc_1$ to be the preimage
of $(\bc/\bc_1)[\hbar^\infty]$ under the projection 
$$\bc\to \bc/\bc_1.$$

Let 
$$\bc_1\subset \bc_2\subset ... \subset \bc$$
be a sequence of subobjects. Consider the corresponding sequence of their saturations
$$\wt\bc_1\subset \wt\bc_2... \subset \bc.$$

By construction, the corresponding maps
$$\wt\bc_i/\hbar\wt\bc_i\to \bc/\hbar\bc$$
are injective. 

\medskip

It is easy to see that as soon as 
$$\wt\bc_i/\hbar\wt\bc_i\hookrightarrow \wt\bc_{i+1}/\hbar\wt\bc_{i+1}$$
is an isomorphism (which happens for some $i$ due to the Noetherianness of $\bC^\heartsuit$), we have
$\wt\bc_i=\wt\bc_{i+1}$. 

\medskip

Replacing the initial $\bc$ by the eventual value of $\wt\bc_i$ and reindexing, we can assume that
all $\wt\bc_i=\bc$. Taking the quotient, we can assume that $\bc$ is torsion. Then the assertion follows
from \lemref{l:kernels stabilize}.

\end{proof}

\begin{proof}[Proof of \lemref{l:kernels stabilize} (due to J.~Lurie)]

Multiplication by $\hbar^n$ defines surjective maps
$$\bc/\hbar\bc \to \hbar^n\bc/\hbar^{n+1}\bc.$$

By Noetherianness, the kernels of these maps stabilize. Hence, for some $n_0$ and all $n\geq n_0$, the maps
$$\hbar^n\bc/\hbar^{n+1}\bc \overset{\hbar}\to \hbar^{n+1}\bc/\hbar^{n+2}\bc$$
are isomorphisms. 

\medskip

By $\hbar$-completeness, the maps 
$$\hbar^n\bc  \overset{\hbar}\to \hbar^{n+1}\bc$$
are also isomorphisms. Hence, $\hbar$ is torsion-free on $\hbar^n\bc$.  

\medskip

Hence, for any $d$, 
$$\bc[\hbar^d]\hookrightarrow (\bc/\hbar^n\bc)[\hbar^d].$$

Now, the assertion follows by Noetherianness. 

\end{proof} 

\sssec{}

We define a t-structure on $\Shv_\hbarl(Y)$ to be uniquely characterized by the property that the localization functor
$$\Shv_\hbart(Y)\to \Shv_\hbarl(Y)$$
is t-exact. 

\medskip

The t-structure on $\Shv_\hbart(Y)$ induces one on $\on{Ind}(\Shv_\hbart(Y))$. The functor
$$\Shv_\hbarl(Y)\to \on{Ind}(\Shv_\hbart(Y))$$
right adjoint to the projection is also t-exact. 

\ssec{Gerbes with a formal parameter}  \label{ss:h bar gerbes} 

\sssec{} \label{sss:twisted shvs h}

For a fixed $n$, one can talk about $(\sfe[\hbar]/\hbar^n)^\times$-gerbes on a given scheme/prestack $Y$,
defined as in \secref{sss:K}. Denote this category
by $\on{Ge}_n(Y)$. Define the category $\on{Ge}_\hbar(Y)$ by
$$\on{Ge}_\hbar(Y):=\underset{n}{\on{lim}}\, \on{Ge}_n(Y).$$

\medskip

Given an object $\CG\in \on{Ge}_\hbar(Y)$, one can consider the category
$$\Shv_{\CG,\hbart}(\CY)$$
and its localization $\Shv_{\CG,\hbarl}(\CY)$.

\sssec{}

Define the group $\fZ_n$ in each of the sheaf theories from \secref{sss:sh th} as follows:

\begin{itemize}

\item For $\Shv(-)$ being constructible sheaves in classical topology with $\sfe$-coefficients, set $$\fZ_n=(\sfe[\hbar]/\hbar^n)^\times;$$

\medskip

\item For $\Shv(-)$ being constructible $\ell$-adic sheaves, set $\fZ_n=(\bar\BZ_\ell[\hbar]/\hbar^n)^\times$; 

\medskip

\item For $\Shv(-)$ being holonomic D-modules, set $\fZ_n=(k[\hbar]/\hbar^n)/\BZ$.

\end{itemize} 

In all of these cases, an element $\zeta\in \fZ_n$ gives rise to a Kummer character sheaf $\Psi_\zeta\in \Shv_n(\BG_m)$. 
Hence, for a line bundle $\CL$ on $Y$ and $\zeta$ as above, we obtain a well-defined object
$$\CL^\zeta\in \on{Ge}_n(Y).$$

\sssec{}   \label{sss:K h bar}

Define 
$$\fZ_\hbar:=\underset{n}{\on{lim}}\, \fZ_n.$$

We obtain that for a line bundle $\CL$ on $Y$ and $\zeta\in \fZ_\hbar$, we obtain a well-defined object 
$$\CL^\zeta\in \on{Ge}_\hbar(Y).$$

\end{document}